\documentclass[11pt]{article}
\pdfoutput=1
\usepackage{a4}
\usepackage{amsmath,amsfonts,amssymb,amscd,amsthm}
\usepackage[pdftex]{graphicx}

\usepackage{subfigure}
\usepackage{pict2e}
\usepackage{color}
\usepackage[percent]{overpic}

\newtheorem{theorem}{Theorem}
\newtheorem{corollary}[theorem]{Corollary}

\newtheorem{lemma}[theorem]{Lemma}
\newtheorem{remark}[theorem]{Remark}

\begin{document}
\title{Stabilized mixed $hp$-BEM for frictional contact problems in linear elasticity}
\author{Lothar Banz\thanks{Institute of Applied Mathematics, Leibniz University Hannover, Welfengarten 1, 30167 Hannover, Germany, email: banz@ifam.uni-hannover.de} \and Heiko Gimperlein\thanks{Maxwell Institute for Mathematical Sciences and Department of Mathematics, Heriot-Watt University, EH14 4AS Edinburgh, United Kingdom} \and Abderrahman Issaoui\thanks{Institute of Applied Mathematics, Leibniz University Hannover, Welfengarten 1, 30167 Hannover, Germany, email: issaoui@ifam.uni-hannover.de} \and Ernst P.~Stephan\thanks{Institute of Applied Mathematics, Leibniz University Hannover, Welfengarten 1, 30167 Hannover, Germany, email: stephan@ifam.uni-hannover.de }}
\date{}

\maketitle \vskip 0.5cm
\begin{abstract}
We investigate $hp$-stabilization for variational inequalities and boundary element methods based on the approach introduced by Barbosa and Hughes for finite elements. Convergence of a stabilized mixed boundary element method is shown for unilateral frictional contact problems for the Lam\'{e} equation. { Without stabilization, the inf-sup constant need not be bounded away from zero for natural discretizations, even for fixed $h$ and $p$}. Both a priori and a posteriori error estimates are presented in the case of Tresca friction, for discretizations based on Bernstein or Gauss-Lobatto-Lagrange polynomials as test and trial functions. We also consider an extension of the a posteriori estimate to Coulomb friction. Numerical experiments underline our theoretical results.
\end{abstract}

\section{Introduction}
\label{sec:Introduction}
Many mechanical applications can be modeled by frictional contact problems. These consist of a differential equation balancing the forces within the object at hand and special contact and friction constraints on one part of the object's boundary. The latter significantly complicates the numerical analysis and computations.\\
Following the seminal FEM-paper \cite{barbosahughes1991}, we investigate the concept of $hp$-stabilization for variational inequalities and boundary elements. With this approach we show convergence of the mixed method which in the non-stabilized approach is not assured as { the inf-sup constant need not be bounded away from zero for natural discretizations, even for fixed $h$ and $p$}. We present a priori and a posteriori error estimates for the $hp$-stabilized mixed boundary element method for Tresca problems. The use of Bernstein polynomials or Gauss-Lobatto-Lagrange polynomials as test and trial functions proves convenient. Also an extension to Coulomb friction is discussed.\\
The paper is structured as follows. In Section~\ref{sec:Amixedboundaryintegralformulation} we introduce a mixed boundary element method with the help of the Poincar\'{e}-Steklov operator which maps the displacement $u$ on the boundary into the boundary traction $-\lambda$. The unique existence of a solution $(u,\lambda)$ of the mixed formulation of the original Tresca friction contact problem is based on the coercivity of the underlying bilinear form  $\left\langle S\cdot,\cdot \right\rangle$
 on the trace space $\tilde{H}^{1/2}(\Gamma_\Sigma) $ and the inf-sup condition for $\lambda$ in the dual space   $H^{-1/2}(\Gamma_C)$ (see Theorem~\ref{thm:infsup}). In Section~\ref{sec:stabalizedDiscretization} we discretize the mixed formulation in suitable piecewise polynomial subspaces of $\tilde{H}^{1/2}(\Gamma_\Sigma) $ and $H^{-1/2}(\Gamma_C)$. On a locally quasi-uniform mesh we use linear combinations of affinely transformed Bernstein polynomials or Gauss-Lobatto-Lagrange polynomials, imposing in both cases additional side conditions which reflect the constraints of non-penetration and stick-slip of the original contact problem. Based on these hp-boundary element spaces we present a stabilized mixed method with stabilization parameter $\gamma|_{E} \sim h_{E}p_{E}^{-2}$ for elements E of the subdivision $\mathcal{T}_h$  of $\Gamma_\Sigma$ into straight line segments. As in \cite{hildrenard2010} for the h-version FEM, the stabilized discrete mixed scheme admits a unique solution
   $(u^{hp},\lambda^{kq})$. We derive a priori error estimates for the Galerkin error in the displacement $u$ and the Lagrange multiplier $\lambda$ showing improved convergence rates for higher polynomial degrees $p$, $q$. Our results (Theorem~\ref{thm:Lame:apriori2} and Theorem~\ref{thm:Lame:apriori22}) include the $h$-version which is considered for lowest order test and trial functions in the FEM by \cite{hild2009residual,hildrenard2007}. In Section~\ref{sec:Aposteriorierrorestimations} we derive an a posteriori error estimate of residual type. After discussing implementational challenges
     in Section~\ref{sec:Implementation} we give an extension of our approach to Coulomb friction in Section~\ref{sec:CoulombFriction} by suitably modifying the test and ansatz spaces.
     Finally, our numerical experiments in Section~\ref{sec:NumericalExperiments} underline our theoretical results. They clearly show that the classical hp-stabilization technique extends to variational inequalities, here for contact problems, handled with boundary integral equations and hp-BEM.

\section{A mixed boundary integral formulation}
\label{sec:Amixedboundaryintegralformulation}

Let $\Omega \subset \mathbb{R}^2$ be a bounded polygonal domain with boundary $\Gamma$ and outward unit normal $n$. Furthermore, let $\bar{\Gamma}=\bar{\Gamma}_D \cup \bar{\Gamma}_N \cup \bar{\Gamma}_C$ be decomposed into
non-overlapping homogeneous Dirichlet, Neumann and contact boundary parts with $\bar{\Gamma}_D \cap \bar{\Gamma}_C = \emptyset$ for simplicity. For a given gap function $g \in H^{1/2}(\Gamma_C)$, friction threshold $0< \mathcal{F} \in L^2(\Gamma_C)$, Neumann data $t\in \tilde{H}^{-1/2}(\Gamma_N)$ and elasticity tensor $\mathcal{C}$ the considered Tresca frictional contact problem is to find a function $u \in H^1_{\Gamma_D}(\Omega):=\left\{v \in H^1(\Omega):\ v|_{\Gamma_D}=0\right\}$ such that
\begin{subequations}  \label{prob:strong_formulation}
\begin{alignat}{2}
-\operatorname{div} \sigma(u)&=0 & \quad & \text{in } \Omega\\
  \sigma(u)&=\mathcal{C}:\epsilon(u) & \quad & \text{in } \Omega \label{eq:prog:materialLaw} \\
	u&=0 & \quad & \text{on } \Gamma_D \\
  \sigma(u)n&=t & \quad & \text{on } \Gamma_N \\
  \sigma_n\leq 0,\ u_n\leq g,\ \sigma_n(u_n-g)&=0 & \quad & \text{on } \Gamma_C\\
  \left\| \sigma_t \right\|  \leq \mathcal{F},\ \sigma_tu_t+\mathcal{F}\left\| u_t \right\| &=0 & \quad & \text{on } \Gamma_C. \label{eq:prog:FrictionLaw}
\end{alignat}
\end{subequations}
Here, $\sigma_n$, $\sigma_t$ are the normal and tangential components of $\sigma(u)n$, respectively and \eqref{eq:prog:materialLaw} describes Hooke's law with the linearized strain tensor $\epsilon(u)=\frac{1}{2} \left( \nabla u + \nabla u^\top \right)$. Often \eqref{eq:prog:FrictionLaw} is written in the form
\begin{align}
 \left\| \sigma_t \right\|  \leq \mathcal{F}, \qquad
 \left\| \sigma_t \right\|  < \mathcal{F} \Rightarrow u_t=0, \qquad
 \left\| \sigma_t \right\|  = \mathcal{F} \Rightarrow \exists \alpha \geq 0 : u_t=-\alpha \sigma_t.
\end{align}
Testing \eqref{prob:strong_formulation} with $v_\Omega \in K_\Omega:=\left\{v_\Omega \in H^1_{\Gamma_D}(\Omega): (v_\Omega)_n \leq g \text{ a.e.~on } \Gamma_C \right\}$ and introducing the friction functional $j(v):=\int_{\Gamma_C} \mathcal{F} \left\|v_t\right\| ds$ yields the (domain) variational inequality formulation:
\begin{align}
 u_\Omega \in K_\Omega:\ \left(\sigma(u_\Omega),\epsilon(v_\Omega-u_\Omega)\right)_{0,\Omega} +j(v_\Omega)-j(u_\Omega) \geq \left\langle t,v_\Omega-u_\Omega\right\rangle_{\Gamma_N} \quad \forall v_\Omega \in K_\Omega
\end{align}
where $\left(u,v \right)_{0,\Omega}=\int_{\Omega}u v\;dx$ and $\left\langle t,v\right\rangle_{\Gamma_N}=\int_{\Gamma_N} t v\; ds$ are defined by duality.

{Boundary integral formulations can be advantageous for problems with non-linear boundary conditions and with no source terms in $\Omega$.} For that let
\begin{alignat}{2}
V\mu (x) &= \int_\Gamma G(x,y) \mu(y) ds_y, &\qquad Kv(x)&=\int_\Gamma  \left(\mathcal{T}_y G(x,y)\right)^\top v(y) ds_y \label{eq:lame:integoperat1}\\
K^\top\mu (x) &= \mathcal{T}_x\int_\Gamma G(x,y) \mu(y) ds_y, &\qquad Wv(x)&=-\mathcal{T}_x \int_\Gamma  \left(\mathcal{T}_y G(x,y)\right)^\top v(y) ds_y \label{eq:lame:integoperat2}
\end{alignat}
be the single layer potential $V$, double layer potential $K$, adjoined double layer potential $K^\top$ and hypersingular integral operator $W$ with the fundamental solution for the Lam\'{e} equation in $\mathbb{R}^2$
\begin{align*}
G(x,y)=\frac{\lambda +3\mu}{4\pi \mu(\lambda+2\mu)}\left\{\log |x-y|I + \frac{\lambda + \mu}{\lambda +3\mu}\frac{(x-y)(x-y)^\top}{|x-y|^2}\right\}
\end{align*}
and traction operator $(\mathcal{T}u)_i=\lambda n_i \operatorname{div} u + \mu \partial_n u_i + \mu\left\langle \frac{\partial u}{\partial x_i},n\right\rangle$, (see \cite{costabelstephan1990}). The Poincar\'e-Steklov operator $S:=W+(K+\frac{1}{2})^\top V^{-1}(K+\frac{1}{2})$, which is $H^{\frac{1}{2}}(\Gamma)$-continuous and $\tilde{H}^{\frac{1}{2}}(\Gamma_\Sigma)$-coercive is a Dirichlet-to-Neumann mapping \cite{carstensen1995adaptive}.
\begin{align*}
\left\langle Su,v\right\rangle = \left\langle \partial_n u,v\right\rangle = \left(\sigma(u_\Omega),\epsilon(v_\Omega)\right)_{0,\Omega}
\end{align*}
 Hence the (domain) variational inequality immediately yields the (boundary) variational inequality formulation: Find $u \in K$ with $K:=\left\lbrace v\in \tilde{H}^{1/2}(\Gamma_\Sigma): u_n \leq g \text{~a.e.~on~} \Gamma_C \right\rbrace$ such that
\begin{align} \label{eq:VarIneq}
\left\langle Su,v-u \right\rangle_{\Gamma_\Sigma} + j(v)-j(u) \geq \left\langle t,v-u \right\rangle_{\Gamma_N} \quad \forall v\in K
\end{align}
where $\bar{\Gamma}_\Sigma:=\bar{\Gamma}_N \cup \bar{\Gamma}_C$. It is well known, e.g.~\cite[Theorems 3.13 and 3.14]{Chernov2006}, \cite{ChernovmaischakStephan2008} that there exists a unique solution to \eqref{eq:VarIneq}. Since neither $K$ is trivial to discretize nor is the non-differentiable friction function $j(v)$ easy to handle  it may be favorable to use an equivalent mixed formulation. To do so, let
\begin{align}
M^+(\mathcal{F}):=\left\lbrace \mu \in \tilde{H}^{-1/2}(\Gamma_C): \left\langle \mu,v\right\rangle_{\Gamma_C} \leq \left\langle \mathcal{F},\left\|v_t\right\| \right\rangle_{\Gamma_C} \forall v \in \tilde{H}^{1/2}(\Gamma_\Sigma), v_n \leq 0 \right\rbrace
\end{align}
be the set of admissible Lagrange multipliers, in which the representative $\lambda=-\sigma(u)n$ is sought. Then, the mixed method is to find the pair $(u,\lambda) \in \tilde{H}^{1/2}(\Gamma_\Sigma) \times M^+(\mathcal{F})$ such that (see \cite{Banz2014BEM})
\begin{subequations} \label{eq:MixedProblem}
\begin{alignat}{2}
 \left\langle Su,v \right\rangle_{\Gamma_\Sigma} + \left\langle \lambda,v \right\rangle_{\Gamma_C} &= \left\langle t,v \right\rangle_{\Gamma_N} &\quad &\forall v\in \tilde{H}^{1/2}(\Gamma_\Sigma) \label{eq:WeakMixedVarEq} \\
\left\langle \mu -\lambda,u \right\rangle_{\Gamma_C} & \leq \left\langle g,\mu_n-\lambda_n \right\rangle_{\Gamma_C} &\quad &\forall \mu \in M^+(\mathcal{F}) . \label{eq:ContContactConstraints}
\end{alignat}
\end{subequations}

\begin{theorem}\label{thm:infsup}
There hold the following results:

\begin{enumerate}
  \item The inf-sup condition is satisfied with a constant $\tilde{\beta} >0$, i.e.
	\begin{align}\label{eq:lame:infsup}
	  \tilde{\beta} \left\| \mu \right\|_{\tilde{H}^{-1/2}(\Gamma_C)} \leq \sup_{v \in \tilde{H}^{1/2}(\Gamma_\Sigma) \setminus \left\{0\right\}} \frac{\left\langle \mu, v \right\rangle_{\Gamma_C} }{ \left\|v\right\|_{\tilde{H}^{1/2}(\Gamma_\Sigma)} } \qquad \forall \mu \in \tilde{H}^{-1/2}(\Gamma_C)
	\end{align}
	\item Any solution of \eqref{eq:MixedProblem} is also a solution of \eqref{eq:VarIneq}.
	\item For the solution $u \in K$ of \eqref{eq:VarIneq} there exists a $\lambda \in M^+(\mathcal{F})$ such that $(u,\lambda)$ is a solution of \eqref{eq:MixedProblem}
	\item There exists a unique solution to \eqref{eq:MixedProblem}
\end{enumerate}
\begin{proof}
1. The inf-sup condition has been proven in \cite[Theorem 3.2.1]{Chernov2006}.\\
2. and 3. follow as in \cite[Section 3]{schroder2012posteriori} with $\left\langle Su,v\right\rangle_{\Gamma_\Sigma} = \left(\sigma(u_\Omega),\epsilon(v_\Omega)\right)_{0,\Omega}$ for volume force $f_\Omega \equiv 0$.\\
4. follows from the equivalence results 2. and 3., the inf-sup condition 1. and from the unique existence of the solution of \eqref{eq:VarIneq} proven in \cite[Theorems 3.13 and 3.14]{Chernov2006}.
\end{proof}
\end{theorem}

\section{Stabilized mixed hp-boundary element discretization including Lagrange multiplier}
\label{sec:stabalizedDiscretization}
Let $\mathcal{T}_h$ be a subdivison of $\Gamma_\Sigma$ into straight line segments. Furthermore, let $p$ be a distribution of polynomial degrees over $\mathcal{T}_h$ which on each element specifies the polynomial degree on the reference interval. We consider the ansatz spaces
\begin{align}
\mathcal{V}_{hp}&:=\left\{v^{hp} \in \tilde{H}^{1/2}(\Gamma_\Sigma): v^{hp}|_E \in \left[\mathbb{P}_{p_E}(E)\right]^2 \forall E \in \mathcal{T}_h \right\}\ ,\\
\mathcal{V}^D_{hp}&:=\left\{\phi^{hp} \in H^{-1/2}(\Gamma_\Sigma): \phi^{hp}|_E \in \left[\mathbb{P}_{p_E-1}(E)\right]^2 \forall E \in \mathcal{T}_h \right\}\ .
\end{align}
In particular, the displacement field $u^{hp} $ is sought in $\mathcal{V}_{hp}$,  $\mathcal{V}^D_{hp}$ is
used to construct the standard approximation \cite{carstensen1995adaptive} $S_{hp}:=W+\left(K^\top+\frac{1}{2}\right)V_{hp}^{-1}\left(K+\frac{1}{2}\right)$ of $S$,
where $V_{hp}$ is the Galerkin realization of the single layer potential over $\mathcal{V}^D_{hp}$. For the discrete Lagrange multiplier let $\hat{\mathcal{T}}_k$ be an additional
subdivision of $\Gamma_C$. The discrete Lagrange multiplier  is sought in
\begin{align}
M_{kq}^+(\mathcal{F})&:= \Big\{ \mu^{kq} \in L^2(\Gamma_C): \mu^{kq}|_E = \sum_{i=0}^{q_E} \mu_i^E B_{i,q_E}^{E} \in \left[\mathbb{P}_{q_E}(E)\right]^2 \forall E\in \hat{\mathcal{T}}_k,\\ &\qquad \ (\mu^E_i)_n \geq 0, \ -\mathcal{F}(\Psi_E(iq_E^{-1})) \leq (\mu_i^E)_t \leq \mathcal{F}(\Psi_E(iq_E^{-1})) \Big\}
\end{align}
where $B_{i,q_E}^{E}$ is the $i$-th Bernstein polynomial of degree $q_E$ affinely transformed onto the interval $E$.  $\Psi_E$ is the
affine mapping from $[0,1]$ onto $E\in \hat{\mathcal{T}}_k$. Since the Bernstein polynomials are non-negative and form a partition of unity, it is
straight forward to show that $M_{kq}^+(\mathcal{F})$ is conforming, i.e.~$M_{kq}^+(\mathcal{F}) \subset M^+(\mathcal{F})$, if $\mathcal{F}$ is linear. Since $M_{kq}^+(\mathcal{F})$ is chosen
independently of $\mathcal{V}_{hp}$ it cannot be expected that the discrete inf-sup condition holds uniformly, i.e.~independently of $h$, $k$, $p$ and $q$, or at all if $\hat{\mathcal{T}}_k= \mathcal{T}_h|_{\Gamma_C}$. To circumvent the need
to restrict the set $M_{kq}^+(\mathcal{F})$, the discrete mixed formulation is stabilized analogously to \cite{barbosahughes1991} for FEM. That is, find the pair $(u^{hp},\lambda^{kq}) \in \mathcal{V}_{hp} \times M^+_{kq}(\mathcal{F})$ such that
\begin{subequations} \label{eq:DiscreteMixedProblem}
\begin{alignat}{2}
 \left\langle S_{hp}u^{hp},v^{hp} \right\rangle_{\Gamma_\Sigma} + \left\langle \lambda^{kq},v^{hp} \right\rangle_{\Gamma_C} - \left\langle  \gamma \left(\lambda^{kq} + S_{hp}u^{hp}\right), S_{hp}v^{hp} \right\rangle_{\Gamma_C}&= \left\langle t,v^{hp} \right\rangle_{\Gamma_N} &\quad &\forall v^{hp}\in \mathcal{V}_{hp} \label{eq:DiscreteVariationalEqualityPart} \\
\left\langle \mu^{kq} -\lambda^{kq},u^{hp} \right\rangle_{\Gamma_C} -  \left\langle  \gamma \left(\mu^{kq}-\lambda^{kq}\right),  \lambda^{kq} + S_{hp}u^{hp}  \right\rangle_{\Gamma_C} & \leq \left\langle g,\mu^{kq}_n-\lambda^{kq}_n \right\rangle_{\Gamma_C} &\quad &\forall \mu^{kq} \in M^+_{hp}(\mathcal{F}) \label{eq:DiscreteContactConstraints}
\end{alignat}
\end{subequations}
Here, $\gamma$ is a piecewise constant function on $\Gamma_C$ such that $\gamma|_E = \gamma_0 h_E p_E^{-2}$ with constant $\gamma_0>0$  for all elements $E \in \mathcal{T}_h$.

\begin{remark}
Often $M^+(\mathcal{F})$ is discretized  such that the constraints are only satisfied in a discrete set of points, namely
\begin{align}
\tilde{M}_{kq}^+(\mathcal{F}):= \left\{ \mu^{kq} \in L^2(\Gamma_C): \mu^{kq}|_E \in \left[\mathbb{P}_{q_E}(E)\right]^2,\ \mu^{kq}_n(x) \geq 0, \ -\mathcal{F} \leq \mu^{kq}_t(x) \leq \mathcal{F}\; \text{for }x\in G_{kq} \right\}\ ,
\end{align}
where $G_{kq}$ is a set of discrete points on $\Gamma_C$, e.g.~affinely transformed Gauss-Lobatto points.
\end{remark}
Unless specifically stated otherwise, the proven results are true for both discretizations $M_{kq}^+(\mathcal{F})$ and $\tilde{M}_{kq}^+(\mathcal{F})$.

In the following we collect some results on $S_{hp}$ which allow to prove existence and uniqueness of the solution of the mixed formulation \eqref{eq:DiscreteMixedProblem}.
\begin{lemma}[Lemma~15 in \cite{maischak2005adaptive}] \label{lem:SteklovApprox}
There holds:

\begin{enumerate}
	\item $S_{hp}$ is  continuous from $\tilde{H}^{1/2}(\Gamma_\Sigma)$ into $H^{-1/2}(\Gamma_\Sigma)$ and coercive on $\tilde{H}^{1/2}(\Gamma_\Sigma)\times \tilde{H}^{1/2}(\Gamma_\Sigma)$ with constants $C_S$ and $\alpha_S$.
	\item $E_{hp}:=S-S_{hp}$ is bounded from $\tilde{H}^{1/2}(\Gamma_\Sigma)$ into ${H}^{1/2}(\Gamma_\Sigma)$, and there exists constants $C_E$, $C>0$ such that
	\begin{align*}
	   \left\|E_{hp} v \right\|_{H^{-1/2}(\Gamma_\Sigma)} \leq C_E \left\| v \right\|_{\tilde{H}^{1/2}(\Gamma_\Sigma)} \quad \text{and} \quad
	   \left\|E_{hp} v \right\|_{H^{-1/2}(\Gamma_\Sigma)} \leq C \inf_{\phi \in \mathcal{V}_{hp}^D } \left\| V^{-1}(K+\frac{1}{2})v -\phi\right\|_{\tilde{H}^{-1/2}(\Gamma_\Sigma)}\ .
	\end{align*}
\end{enumerate}
\end{lemma}

\begin{lemma}[Lemma~3.2.7 in \cite{Chernov2006}] \label{lem:V_ortho}
Let $i_{hp}: \mathcal{V}_{hp}^D \mapsto H^{-1/2}(\Gamma_\Sigma)$ be the canonical embedding and $i_{hp}^*$ its dual. Furthermore, let
\begin{align}\label{eq:Lame:defpsi}
 \psi=V^{-1}(K+\frac{1}{2})u, \qquad \psi_{hp}^*=V^{-1}(K+\frac{1}{2})u^{hp}, \qquad \psi^{hp}=i_{hp}V_{hp}^{-1}i_{hp}^*(K+\frac{1}{2})u^{hp}.
\end{align}
Then there holds
\begin{align*}
 \left\langle V (\psi^*_{hp}-\psi^{hp}),\phi^{hp}\right\rangle_{\Gamma_\Sigma}=0 \quad \forall \phi^{hp} \in \mathcal{V}^D_{hp}.
\end{align*}
\end{lemma}

\begin{theorem} \label{thm:inverseEstimate}
Let $\mathcal{T}_h$ be a locally quasi-uniform mesh. Then there holds
\begin{align}
\sum_{E \in \mathcal{T}_h} \left\| \frac{h_E^{1/2}}{p_E} S_{hp} v^{hp} \right\|_{L^2(E)}^2 \leq C^2 \left\|  v^{hp}\right\|^2_{\tilde{H}^{1/2}(\Gamma_\Sigma)} \quad \forall v^{hp} \in \mathcal{V}_{hp}.
\end{align}
\begin{proof}
From the definition of $S_{hp}$ follows that $S_{hp} v^{hp} = Wv^{hp}  + (K^\top +\frac{1}{2})\eta^{hp}$ with  $\eta^{hp}= V_{hp}^{-1}(K+\frac{1}{2})v^{hp} \in \mathcal{V}_{hp}^D$. In \cite[Theorem 4.4]{Karkulik2012} it is shown that
\begin{align*}
\sum_{E \in \mathcal{T}_h} \left\| \frac{h_E^{1/2}}{p_E} W v^{hp} \right\|_{L^2(E)}^2 \leq C^2 \left\|  v^{hp}\right\|^2_{\tilde{H}^{1/2}(\Gamma_\Sigma)}, \qquad
\sum_{E \in \mathcal{T}_h} \left\| \frac{h_E^{1/2}}{p_E} K^\top \eta^{hp} \right\|_{L^2(E)}^2 \leq C^2 \left\| \eta^{hp}\right\|^2_{\tilde{H}^{-1/2}(\Gamma_\Sigma)}
\end{align*}
for the boundary integral operators associated to the Laplacian. For the integral operators \eqref{eq:lame:integoperat1}, \eqref{eq:lame:integoperat2} of the Lam\'{e} equation this can be done analogously.  The assertion follows with the mapping properties of $V_{hp}^{-1}(K+\frac{1}{2})$.
\end{proof}
\end{theorem}

\begin{lemma}[Coercivity] \label{lem:coerciveDiscrete}
For $\gamma_0$ sufficiently small, there exists a constant $\alpha>0$ independent of $h$, $p$, $k$ and $q$, such that
\begin{align}
 \left\langle S_{hp}v^{hp},v^{hp} \right\rangle_{\Gamma_\Sigma}  - \left\langle  \gamma  S_{hp}v^{hp}, S_{hp}v^{hp} \right\rangle_{\Gamma_C} \geq \alpha \left\|v^{hp}\right\|^2_{\tilde{H}^{1/2}(\Gamma_\Sigma)} \qquad \forall v\in \mathcal{V}_{hp}.
\end{align}
\begin{proof}
From Theorem~\ref{thm:inverseEstimate} it follows that
\begin{align}
\left\langle  \gamma  S_{hp}v^{hp}, S_{hp}v^{hp} \right\rangle_{\Gamma_C} \leq \gamma_0 C \left\|v^{hp}\right\|^2_{\tilde{H}^{1/2}(\Gamma_\Sigma)}
\end{align}
with $C>0$ independent of of $h$, $p$, $k$ and $q$. Hence, from the coercivity of $S_{hp}$ there holds
\begin{align}
\left\langle S_{hp}v^{hp},v^{hp} \right\rangle_{\Gamma_\Sigma}  - \left\langle   \gamma  S_{hp}v^{hp}, S_{hp}v^{hp} \right\rangle_{\Gamma_C} \geq (\alpha_S - \gamma_0 C) \left\|v^{hp}\right\|^2_{\tilde{H}^{1/2}(\Gamma_\Sigma)}.
\end{align}
\end{proof}
\end{lemma}

In the following it is assumed, that $\gamma_0$ is sufficiently small and, therefore, Lemma~\ref{lem:coerciveDiscrete} is always applicable.

\begin{theorem}[Existence / Uniqueness]
For $\gamma_0$ sufficiently small, the discrete, stabilized problem \eqref{eq:DiscreteMixedProblem} has a unique solution.
\begin{proof}
In the standard manner it can be shown that \eqref{eq:DiscreteMixedProblem} is equivalent to the saddle-point problem: Find $(u^{hp},\lambda^{kq}) \in \mathcal{V}_{hp} \times M^+_{kq}(\mathcal{F})$ such that
\begin{align}\label{eq:stab:saddle}
 \mathcal{L}_{\gamma}(u^{hp}, \mu^{kq})\leq \mathcal{L}_{\gamma}(u^{hp}, \lambda^{kq})\leq\mathcal{L}_{\gamma}(v^{hp}, \lambda^{kq})
\qquad \forall v^{hp} \in \mathcal{V}_{hp},\  \forall  \mu^{kq}\in M^+_{kq}(\mathcal{F}),
\end{align}
with
\begin{align}\label{eq:stab:saddle1}
 \mathcal{L}_{\gamma}(v^{hp}, \mu^{kq})=\frac{1}{2}\langle S_{hp} v^{hp} ,v^{hp}\rangle_{\Gamma_\Sigma}-L(v^{hp})+\left\langle  \mu^{kq},v^{hp} \right\rangle_{\Gamma_C}-\frac{1}{2} \left\langle \gamma(\mu^{kq}+S_{hp}v^{hp}),\mu^{kq}+S_{hp}v^{hp}\right\rangle_{\Gamma_C}.
\end{align}
Due to
\begin{align*}
 \mathcal{L}_{\gamma}(v^{hp}, 0)=\frac{1}{2}\langle S_{hp} v^{hp} ,v^{hp}\rangle_{\Gamma_\Sigma}-L(v^{hp})-\frac{1}{2}\int_{\Gamma_{C}}\gamma(S_{hp}v^{hp})^{2}ds \geq \frac{\alpha}{2} \left\|v^{hp}\right\|^2_{\tilde{H}^{1/2}(\Gamma_\Sigma)} - \left\|t\right\|_{H^{-1/2}(\Gamma_N)} \left\|v^{hp}\right\|_{\tilde{H}^{1/2}(\Gamma_\Sigma)}
\end{align*}
and $\mathcal{L}_{\gamma}(0, \mu^{kq})=-\frac{1}{2}\int_{\Gamma_{C}}\gamma(\mu^{kq})^{2}ds$, $\mathcal{L}_{\gamma}$ is strictly convex in $v^{hp}$ and strictly concave in $\mu^{kq}$. Since it is also continuous on $\mathcal{V}_{hp} \times M^+_{kq}(\mathcal{F})$ and $\mathcal{V}_{hp}$, $M^+_{kq}(\mathcal{F})$ are non-empty convex sets (standard arguments) provide the existence of the solution.\\
Let $(u_1,\lambda_1)$ and $(u_2,\lambda_2)$ be two solutions of \eqref{eq:DiscreteMixedProblem}. Then, choosing $\mu_1=\lambda_2$ and $\mu_2=\lambda_1$ in \eqref{eq:DiscreteContactConstraints} yields after adding these two inequalities
\begin{align} \label{eq:unique_u_proof1}
\left\langle \lambda_1-\lambda_2,u_1-u_2\right\rangle_{\Gamma_C} - \left\langle \gamma (\lambda_1-\lambda_2),\lambda_1-\lambda_2+S_{hp}(u_1-u_2) \right\rangle_{\Gamma_C} \geq 0.
\end{align}
Furthermore, inserting $u_1$ and $u_2$ in \eqref{eq:DiscreteVariationalEqualityPart} respectively and  subtracting the  two resulting equations, setting  $v^{hp}=u_1-u_2$ implies
\begin{align*}
0&= \left\langle S_{hp}(u_1-u_2),u_1-u_2 \right\rangle_{\Gamma_\Sigma} + \left\langle \lambda_1-\lambda_2,u_1-u_2 \right\rangle_{\Gamma_C} -  \left\langle\gamma \left(\lambda_1-\lambda_2 + S_{hp}(u_1-u_2)\right) ,S_{hp}(u_1-u_2) \right\rangle_{\Gamma_C}  \\
&\geq \left\langle S_{hp}(u_1-u_2),u_1-u_2 \right\rangle_{\Gamma_\Sigma} - \left\langle  \gamma S_{hp}(u_1-u_2), S_{hp}(u_1-u_2) \right\rangle_{\Gamma_C} +  \left\langle  \gamma (\lambda_1-\lambda_2),\lambda_1-\lambda_2\right\rangle_{\Gamma_C}\\
&\geq \alpha \left\|u_1-u_2\right\|^2_{\tilde{H}^{1/2}(\Gamma_\Sigma)} + \left\|\gamma^{1/2} (\lambda_1-\lambda_2)\right\|^2_{L^2(\Gamma_C)}.
\end{align*}
This yields the asserted uniqueness of the solution.
\end{proof}
\end{theorem}

Due to the conformity in the primal variable there trivially holds the following Galerkin orthogonality.
\begin{lemma} \label{lem:galerkinOrtho}
Let $(u,\lambda)$, $(u^{hp},\lambda^{kq})$ be the solution of \eqref{eq:MixedProblem}, \eqref{eq:DiscreteMixedProblem} respectively. Then there holds
\begin{align*}
 \langle Su-S_{hp} u^{hp} ,v^{hp}\rangle_{\Gamma_\Sigma} + \left\langle \lambda-\lambda^{kq},v^{hp}\right\rangle_{\Gamma_C}+\left\langle\gamma(\lambda^{kq}+S_{hp}u^{hp}),S_{hp}v^{hp} \right\rangle_{\Gamma_C}=0\quad &\forall v^{hp}\in \mathcal{V}_{hp}.
\end{align*}
\end{lemma}
The next result will be used in our error analysis in Section \ref{sec:Apriorierrorestimations}.

\begin{lemma}[Stability]\label{lem:stability}
There exists a constant $C>0$, independent of $h$, $p$, $k$ and $q$, such that
\begin{align}
\alpha \left\|u^{hp}\right\|^2_{\tilde{H}^{1/2}(\Gamma_\Sigma)} + \left\|\gamma^{1/2} \lambda^{kq} \right\|^2_{L^2(\Gamma_C)}   &\leq \left(C \left\|u\right\|_{\tilde{H}^{1/2}(\Gamma_\Sigma)} + \left\|\lambda\right\|_{\tilde{H}^{-1/2}(\Gamma_C)} \right) \left\|u^{hp}\right\|_{\tilde{H}^{1/2}(\Gamma_\Sigma)}\\
& \qquad  + \left\|g\right\|_{{H}^{1/2}(\Gamma_C)}  \left\|\lambda_n^{kq}\right\|_{\tilde{H}^{-1/2}(\Gamma_C)}
\end{align}
\begin{proof}
Choosing $\mu_{n}^{kq}=0,2\lambda_n^{kq}$ and $\mu_{t}^{kq}=\lambda_{t}^{kq}$ in \eqref{eq:DiscreteContactConstraints} yields
\begin{align*}
  \left\langle \lambda_{n}^{kq}, u_{n}^{hp} \right\rangle_{\Gamma_C} - \left\langle \gamma \lambda_{n}^{kq},\lambda_{n}^{kq}+(S_{hp}u^{hp})n \right\rangle_{\Gamma_C}  = \left\langle g,\lambda_n^{kq}\right\rangle_{\Gamma_C},
\end{align*}
whereas $\mu_n^{kq}=\lambda_n^{kq}$ and $\mu_t^{kq}=0$ yields
\begin{align*}
  \left\langle -\lambda_{t}^{kq}, u_{t}^{hp} \right\rangle_{\Gamma_C} +  \left\langle \gamma \lambda_{t}^{kq},\lambda_{t}^{kq}+(S_{hp}u^{hp})t \right\rangle_{\Gamma_C}\leq 0.
\end{align*}
Hence, \eqref{eq:DiscreteMixedProblem} yields with $v^{hp}=u^{hp}$ and Lemma~\ref{lem:coerciveDiscrete}
\begin{align*}
\left\langle t,u^{hp}\right\rangle_{\Gamma_N} &= \left\langle S_{hp}u^{hp},u^{hp} \right\rangle_{\Gamma_\Sigma} -\left\langle\gamma S_{hp} u^{hp}, S_{hp}u^{hp} \right\rangle_{\Gamma_C} + \left\langle \lambda^{kq}, u^{hp} \right\rangle_{\Gamma_C} - \left\langle\gamma \lambda^{kq}, S_{hp}u^{hp} \right\rangle_{\Gamma_C} \\
&\geq \left\langle S_{hp}u^{hp},u^{hp} \right\rangle_{\Gamma_\Sigma} - \left\langle \gamma S_{hp} u^{hp}, S_{hp}u^{hp} \right\rangle_{\Gamma_C} + \left\langle \gamma \lambda^{kq}, \lambda^{kq} \right\rangle_{\Gamma_C} + \left\langle g, \lambda_n^{kq}\right\rangle_{\Gamma_C} \\
& \geq \alpha \left\|u^{hp}\right\|^2_{\tilde{H}^{1/2}(\Gamma_\Sigma)} + \left\|\gamma^{1/2} \lambda^{kq} \right\|^2_{L^2(\Gamma_C)} + \left\langle g, \lambda_n^{kq}\right\rangle_{\Gamma_C}
\end{align*}
On the other hand from the \eqref{eq:WeakMixedVarEq} with $v = u^{hp} \in \mathcal{V}_{hp} \subset \tilde{H}^{1/2}(\Gamma_\Sigma)$ follows
\begin{align*}
\left\langle t,u^{hp}\right\rangle_{\Gamma_N} \leq \left(C \left\|u\right\|_{\tilde{H}^{1/2}(\Gamma_\Sigma)} + \left\|\lambda\right\|_{\tilde{H}^{-1/2}(\Gamma_C)} \right) \left\|u^{hp}\right\|_{\tilde{H}^{1/2}(\Gamma_\Sigma)},
\end{align*}
which completes the proof.
\end{proof}
\end{lemma}

\begin{corollary}
If $ \lambda^{kq} \in M_{kq}^+(\mathcal{F})$ or $\lambda^{kq} \in \tilde{M}^+_{kq}(\mathcal{F})$ with $q=1$, i.e.~$\lambda_n^{kq}\geq 0$, and if $g \geq 0$, then there exists a constant $C>0$, independent of $h$, $p$, $k$ and $q$, such that
\begin{align}
\alpha \left\|u^{hp}\right\|^2_{\tilde{H}^{1/2}(\Gamma_\Sigma)} + \left\|\gamma^{1/2} \lambda^{kq} \right\|^2_{L^2(\Gamma_C)}   \leq \left(C \left\|u\right\|_{\tilde{H}^{1/2}(\Gamma_\Sigma)} + \left\|\lambda\right\|_{\tilde{H}^{-1/2}(\Gamma_C)} \right) \left\|u^{hp}\right\|_{\tilde{H}^{1/2}(\Gamma_\Sigma)}.
\end{align}
\begin{proof}
 The assertion follows directly from Lemma \ref{lem:stability}.
\end{proof}

\end{corollary}

\section{A priori error estimates}
\label{sec:Apriorierrorestimations}

\begin{lemma}\label{lem:Lame:lagrangestaestimate}
Let $(u,\lambda) \in H^1(\Gamma_\Sigma) \times L^2(\Gamma_C)$, $(u^{hp},\lambda^{kq})$ be the solutions of \eqref{eq:MixedProblem}, \eqref{eq:DiscreteMixedProblem} respectively. There holds
\begin{align}\label{eq:Lame:lagrangestabesti}
  \left\| \gamma^{\frac{1}{2}}(\lambda-\lambda^{kq})\right\|^{2}_{ L_{2}(\Gamma_{C})}\leq-\left\langle \lambda-\lambda^{kq},u^{hp}-u\right\rangle_{\Gamma_C}+R
\end{align}
where for any $\mu \in L^2(\Gamma_C) \cap M^+(\mathcal{F})$, $\mu^{kq} \in M^+_{kq}(\mathcal{F})$ we define
\begin{align}\label{eq:Lame:lagrangestabestiR}
R:=&\left\langle \lambda^{kq}-\mu ,u\right\rangle_{\Gamma_C}
 +\left\langle \lambda-\mu^{kq},u^{hp}+\gamma(-\lambda^{kq}-Su^{hp})\right\rangle_{\Gamma_C}
 -\left\langle \gamma(\lambda-\lambda^{kq}),S( u-u^{hp})\right\rangle_{\Gamma_C}\nonumber\\
&-\left\langle \gamma(\mu^{kq}-\lambda^{kq}) ,E_{hp} u^{hp}\right\rangle_{\Gamma_C}+\left\langle g,\mu_{n}^{kq}-\lambda_{n}^{kq}+\mu_n-\lambda_n\right\rangle_{\Gamma_C}\ .
\end{align}

\begin{proof}
First note that
\begin{align}
  \left\|\gamma^{\frac{1}{2}}(\lambda-\lambda^{kq}) \right\|^{2}_{ L_{2}(\Gamma_{C})}= \left\langle \gamma\lambda,\lambda \right\rangle_{\Gamma_C}-2\left\langle \gamma\lambda,\lambda^{kq} \right\rangle_{\Gamma_C}+
 \left\langle \gamma \lambda^{kq},\lambda^{kq}\right\rangle_{\Gamma_C}.
 \end{align}
Rearranging \eqref{eq:DiscreteContactConstraints} we get for all $\mu^{kq} \in  M^+_{kq}(\mathcal{F})$
\begin{align*}
 \left\langle \gamma\lambda^{kq},\lambda^{kq}\right\rangle_{\Gamma_C} \leq  \left\langle \gamma\lambda^{kq},\mu^{kq} \right\rangle_{\Gamma_C}- \left\langle \mu^{kq}-\lambda^{kq},u^{hp} \right\rangle_{\Gamma_C}
+ \left\langle \gamma(\mu^{kq}-\lambda^{kq}), S_{hp} u^{hp}\right\rangle_{\Gamma_C}+\left\langle g,\mu_{n}^{kq}-\lambda_{n}^{kq}\right\rangle_{\Gamma_C}
\end{align*}
and from \eqref{eq:ContContactConstraints} with $\lambda=-Su$ in $L^2(\Gamma_{C})$ if $\lambda \in L^2(\Gamma_C)$ we get
\begin{align*}
 \left\langle \gamma\lambda,\lambda \right\rangle_{\Gamma_C}  \leq  \left\langle \gamma\lambda,\mu  \right\rangle_{\Gamma_C} - \left\langle \mu-\lambda,u \right\rangle_{\Gamma_C} +
 \left\langle \gamma(\mu-\lambda) ,Su\right\rangle_{\Gamma_C} + \left\langle g,\mu_{n} -\lambda_{n}\right\rangle_{\Gamma_C} \quad \forall \mu \in L^2(\Gamma_C) \cap M^+(\mathcal{F}).
\end{align*}
This gives
\begin{align*}
 &\!\!\!\!  \left\|\gamma^{\frac{1}{2}}(\lambda-\lambda^{kq}) \right\|^{2}_{ L_{2}(\Gamma_{C})}\leq
 \left\langle \gamma(\mu-\lambda^{kq}),\lambda\right\rangle_{\Gamma_C} +\left\langle \gamma(\mu^{kq}-\lambda),\lambda^{kq}\right\rangle_{\Gamma_C}
+ \left\langle \lambda-\mu,u\right\rangle_{\Gamma_C} \\
&+ \left\langle \gamma(\mu-\lambda) ,Su\right\rangle_{\Gamma_C}
+ \left\langle \lambda^{kq}-\mu^{kq},u^{hp} \right\rangle_{\Gamma_C}+ \left\langle\gamma(\mu^{kq}-\lambda^{kq}), S_{hp}u^{hp}\right\rangle_{\Gamma_C}+\left\langle g,\mu_{n}^{kq}-\lambda_{n}^{kq}+\mu_n-\lambda_n \right\rangle_{\Gamma_C}\\
=&  \left\langle\gamma(\mu-\lambda^{kq}),\lambda \right\rangle_{\Gamma_C}+ \left\langle\gamma(\mu^{kq}-\lambda),\lambda^{kq}\right\rangle_{\Gamma_C}
+ \left\langle \lambda-\mu,u\right\rangle_{\Gamma_C}
+ \left\langle \gamma(\mu-\lambda), Su\right\rangle_{\Gamma_C}\\
&+\left\langle (\lambda^{kq}-\mu^{kq}),u^{hp} \right\rangle_{\Gamma_C}+\left\langle g,\mu_{n}^{kq}-\lambda_{n}^{kq}+\mu_n-\lambda_n \right\rangle_{\Gamma_C}
+\left\langle \gamma(\mu^{kq}-\lambda^{kq}) ,S u^{hp}\right\rangle_{\Gamma_C}\\
&-\left\langle \gamma(\mu^{kq}-\lambda^{kq}) ,E_{hp} u^{hp}\right\rangle_{\Gamma_C}.
\end{align*}
\end{proof}

\end{lemma}

\begin{theorem}\label{thm:Lame:apriori1}
 Let $( u, \lambda)$, $( u^{hp}, \lambda^{kq})$  be the solutions of \eqref{eq:MixedProblem}, \eqref{eq:DiscreteMixedProblem}, respectively. If $u\in{H}^{1}(\Gamma_\Sigma)$
 and $ \lambda\in L^{2}(\Gamma_{C})$, then there holds with arbitrary $v^{hp}\in\mathcal{V}_{hp}$, $v\in\tilde{H}^{1/2}(\Gamma_\Sigma)$, $\phi^{hp}\in\mathcal{V}_{hp}^{D}$, $\mu\in M^+(\mathcal{F})\cap L^2(\Gamma_{C})$
\begin{align*}
&\alpha_{1}\lVert u- u^{hp}\lVert^{2}_{\tilde{H}^{\frac{1}{2}}(\Gamma_\Sigma)}+
\alpha_{2} \lVert \psi-\psi^{hp}\lVert^{2}_{\tilde{H}^{-\frac{1}{2}}(\Gamma_\Sigma)}+
\alpha_{3}\lVert\gamma^{\frac{1}{2}}(\lambda-\lambda^{kq})\lVert^{2}_{L^2(\Gamma_{C})}\\
&\leq\alpha_{4}\lVert u- v^{hp}\lVert^{2}_{\tilde{H}^{\frac{1}{2}}(\Gamma_\Sigma)}
+\alpha_{5}\lVert\psi-\phi^{hp}\lVert^{2}_{\tilde{H}^{-\frac{1}{2}}(\Gamma_\Sigma)}
+\frac{1}{\epsilon}\lVert\gamma^{\frac{1}{2}}S( u- v^{hp})\lVert^{2}_{L^2(\Gamma_{C})}\\
&+\lVert t-S u\lVert_{L^2(\Gamma_\Sigma)}(\lVert u^{hp}- v\lVert_{L^2(\Gamma_\Sigma)}+\lVert u- v^{hp}\lVert_{L^2(\Gamma_\Sigma)})
+\langle \lambda, v- u^{hp}\rangle_{\Gamma_{C}}-\langle \lambda^{kq}, u-v^{hp}\rangle_{\Gamma_{C}}
+\langle \lambda^{kq}-\mu, u\rangle_{\Gamma_{C}}\\
&+\langle\gamma(\lambda+Su^{hp}),S(u^{hp}-v^{hp})\rangle_{\Gamma_{C}}
- \langle\gamma(\lambda^{kq}+Su^{hp}),E_{hp}(u^{hp}-v^{hp})\rangle_{\Gamma_{C}}\\
&{+}\langle \lambda-\mu^{kq}, u^{hp}+\gamma(-\lambda^{kq}-Su^{hp})\rangle_{\Gamma_{C}}
+\langle\gamma E_{hp}(u^{hp}),S(u^{hp}-v^{hp})\rangle_{\Gamma_{C}}
 + \langle\gamma E_{hp}(u^{hp}),E_{hp}(u^{hp}-v^{hp})\rangle_{\Gamma_{C}}\\
&+\langle \gamma (\lambda^{kq}-\mu^{kq}),E_{hp}(u^{hp})\rangle_{\Gamma_{C}}+\left\langle g,\mu_{n}^{kq}-\lambda_{n}^{kq}+\mu_{n}-\lambda_{n}\right\rangle_{\Gamma_C}
\end{align*}
where the constants $\alpha_{1}=2C_{W}-3\epsilon$, $\alpha_{2}=2C_{V}-\epsilon$, $\alpha_{3}=2-\epsilon$,  $\alpha_{4}=\frac{C_{S}^{2}}{\epsilon}+\frac{C_{E_{h}}^{2}}{\epsilon}+C_{0}$, $\alpha_{5}=C_{0}+\frac{1}{\epsilon}(C_{K}+\frac{1}{2})^{2}+\frac{1}{\epsilon}C_{V}^{2}$
are independent of h, k, p and q;  $\alpha_{1}$, $\alpha_{2}$ and $\alpha_{3}$ are positive if $\epsilon$ is small enough.
\begin{proof}
Recall that $E_{hp}=S-S_{hp}$, i.e.~$Su - S_{hp}u^{hp}=S(u-u^{hp}) -E_{hp}u^{hp}$. Then by the construction of $\psi^*_{hp}$ and the coercivity of $W$ and $V$, there holds for all $v^{hp} \in \mathcal{V}_{hp}$
\begin{align*}
C_{W} &  \left\|u-u^{hp}\right\|^2_{\tilde{H}^{1/2}(\Gamma_\Sigma)} +C_{V} \left\|\psi-\psi^{hp}\right\|^2_{\tilde{H}^{-1/2}(\Gamma_\Sigma)} \\
& \leq  \left\langle Su - S_{hp}u^{hp},u-u^{hp}\right\rangle_{\Gamma_\Sigma} +  \left\langle V(\psi^*_{hp}-\psi^{hp}),\psi-\psi^{hp}\right\rangle_{\Gamma_\Sigma} \\
&= \left\langle S(u - u^{hp}),u-v^{hp}\right\rangle_{\Gamma_\Sigma} +  \left\langle S(u^{hp}-u),u^{hp}-v^{hp}\right\rangle_{\Gamma_\Sigma} \\&\qquad + \left\langle E_{hp} u^{hp},u-u^{hp}\right\rangle_{\Gamma_\Sigma} + \left\langle V(\psi^*_{hp}-\psi^{hp}),\psi-\psi^{hp}\right\rangle_{\Gamma_\Sigma}
\end{align*}
Since $\mathcal{V}_{hp} \subset \tilde{H}^{1/2}(\Gamma_\Sigma)$, $E_{hp}=S-S_{hp}$, using \eqref{eq:DiscreteMixedProblem}, we have
\begin{align*}
\left\langle Su^{hp},u^{hp}-v^{hp}\right\rangle_{\Gamma_\Sigma} & = \left\langle E_{hp}u^{hp},u^{hp}-v^{hp}\right\rangle_{\Gamma_\Sigma}
-\left\langle \lambda^{kq},u^{hp}-v^{hp} \right\rangle_{\Gamma_C} \\ &\qquad +  \left\langle  \gamma \left(\lambda^{kq} + S_{hp}u^{hp}\right), S_{hp}(u^{hp}-v^{hp}) \right\rangle_{\Gamma_C} + \left\langle t,u^{hp}-v^{hp} \right\rangle_{\Gamma_N}
\end{align*}
Hence,
\begin{align*}
C_{W} & \left\|u-u^{hp}\right\|^2_{\tilde{H}^{1/2}(\Gamma_\Sigma)} + C_{V}\left\|\psi-\psi^{hp}\right\|^2_{\tilde{H}^{-1/2}(\Gamma_\Sigma)} \\ 
&\leq  \left\langle S(u - u^{hp}),u-v^{hp}\right\rangle_{\Gamma_\Sigma} +\left\langle Su,v^{hp}-u^{hp}\right\rangle_{\Gamma_\Sigma}+  \left\langle E_{hp}u^{hp},u-v^{hp}\right\rangle_{\Gamma_\Sigma}\\
&\qquad - \left\langle \lambda^{kq},u^{hp}-v^{hp} \right\rangle_{\Gamma_C}
 +  \left\langle  \gamma \left(\lambda^{kq} + S_{hp}u^{hp}\right), S_{hp}(u^{hp}-v^{hp}) \right\rangle_{\Gamma_C} \\&\qquad + \left\langle V(\psi^*_{hp}-\psi^{hp}),\psi-\psi^{hp}\right\rangle_{\Gamma_\Sigma}+\left\langle t,u^{hp}-v^{hp} \right\rangle_{\Gamma_N} \ .
\end{align*}
Now, the individual terms can be bounded by Cauchy-Schwarz inequality and Young's inequality, namely
\begin{align*}
  \left\langle S(u - u^{hp}),u-v^{hp}\right\rangle_{\Gamma_\Sigma} & \leq C \left\|u - u^{hp}\right\|_{\tilde{H}^{1/2}(\Gamma_\Sigma)} \left\|u - v^{hp}\right\|_{\tilde{H}^{1/2}(\Gamma_\Sigma)}\\& \leq  \epsilon \left\|u - u^{hp}\right\|^2_{\tilde{H}^{1/2}(\Gamma_\Sigma)} + \frac{C^2}{4\epsilon} \left\|u - v^{hp}\right\|^2_{\tilde{H}^{1/2}(\Gamma_\Sigma)}
\end{align*}
and additionally with Lemma~\ref{lem:SteklovApprox} we have
\begin{align*}
  \left\langle E_{hp}u^{hp},u-v^{hp}\right\rangle_{\Gamma_\Sigma} &=\left\langle E_{hp}(u^{hp}-u +u),u-v^{hp}\right\rangle_{\Gamma_\Sigma} \\&\leq C_E \left( \left\|u\right\|_{\tilde{H}^{1/2}(\Gamma_\Sigma)} \left\|u - v^{hp}\right\|_{\tilde{H}^{1/2}(\Gamma_\Sigma)}  + \left\|u-u^{hp}\right\|_{\tilde{H}^{1/2}(\Gamma_\Sigma)} \left\|u - v^{hp}\right\|_{\tilde{H}^{1/2}(\Gamma_\Sigma)} \right) \\
	&\leq C_E \left( \left\|u\right\|_{\tilde{H}^{1/2}(\Gamma_\Sigma)} \left\|u - v^{hp}\right\|_{\tilde{H}^{1/2}(\Gamma_\Sigma)} + \epsilon \left\|u-u^{hp}\right\|^2_{\tilde{H}^{1/2}(\Gamma_\Sigma)} +\frac{1}{4\epsilon}\left\|u - v^{hp}\right\|^2_{\tilde{H}^{1/2}(\Gamma_\Sigma)} \right).
\end{align*}
From Lemma~\ref{lem:V_ortho} it follows with $\phi^{hp} \in \mathcal{V}_{hp}^D$ that
\begin{align*}
 \left\langle V(\psi^*_{hp}-\psi^{hp}),\psi-\psi^{hp}\right\rangle_{\Gamma_\Sigma}&=\left\langle V(\psi^*_{hp}-\psi^{hp}),\psi-\phi^{hp}\right\rangle_{\Gamma_\Sigma}\\
 &\leq \left[ \left(C_K +\frac{1}{2} \right)\left\|u-u^{hp}\right\|_{\tilde{H}^{1/2}(\Gamma_\Sigma)} + \left\|\psi-\psi^{hp}\right\|_{\tilde{H}^{-1/2}(\Gamma_\Sigma)}\right] \left\|\psi-\phi^{hp}\right\|_{\tilde{H}^{-1/2}(\Gamma_\Sigma)}.
\end{align*}
Collecting the above terms with \eqref{eq:WeakMixedVarEq}  we obtain for arbitrary $v\in\tilde{H}^{1/2}(\Gamma_\Sigma)$
\begin{align*}
&C_W \lVert u- u^{hp}\lVert^{2}_{\tilde{H}^{\frac{1}{2}}(\Gamma_\Sigma)}+C_V \lVert \psi-\psi^{hp}\lVert^{2}_{\tilde{H}^{-\frac{1}{2}}(\Gamma_\Sigma)}
-\langle  \lambda- \lambda^{kq}, u^{hp}- u\rangle_{\Gamma_{C}}\\&\leq   \langle S( u- u^{hp}) , u- v^{hp}\rangle_{\Gamma_\Sigma}+\langle E_{hp} u^{hp} , u- v^{hp}\rangle_{\Gamma_\Sigma}\\
  &\qquad +\langle t-S u, u^{hp}- v\rangle_{\Gamma_N}+\langle t-S u, u- v^{hp}\rangle_{\Gamma_N}
 +\left\langle \lambda, v- u^{hp}\right\rangle_{\Gamma_C}-\left\langle  \lambda^{kq}, u- v^{hp}\right\rangle_{\Gamma_C}\\ &\qquad 
 + \langle V(\psi^{*}_{hp}-\psi^{hp}),\psi-\psi^{hp}\rangle_{\Gamma_\Sigma}
  +\left\langle  \gamma(\lambda^{kq}+S_{hp} u^{hp}),S_{hp}( u^{hp}-v^{hp})\right\rangle_{\Gamma_C}\ .
\end{align*}
Now with Lemma \ref{lem:Lame:lagrangestaestimate} we have for arbitrary $v^{hp}\in\mathcal{V}_{hp}$ and $v\in\tilde{H}^{1/2}(\Gamma_\Sigma)$
\begin{align*}
 &C_W \lVert u- u^{hp}\lVert^{2}_{\tilde{H}^{\frac{1}{2}}(\Gamma_\Sigma)}+ C_V\lVert \psi-\psi^{hp}\lVert^{2}_{\tilde{H}^{-\frac{1}{2}}(\Gamma_\Sigma)}
  +\lVert\gamma^{\frac{1}{2}}(\lambda-\lambda^{kq})\lVert^{2}_{ L^2(\Gamma_{C})}\\
&  \leq\left\langle S( u- u^{hp}) , u- v^{hp}\right\rangle_{\Gamma_\Sigma}+\left\langle E_{hp} u^{hp} , u- v^{hp}\right\rangle_{\Gamma_C}
 +\left\langle t-S u, u^{hp}- v\right\rangle_{\Gamma_N}
 +\left\langle t-S u, u- v^{hp}\right\rangle_{\Gamma_N}\\
&\qquad + \left\langle  \lambda, v- u^{hp}\right\rangle_{\Gamma_C}- \left\langle \lambda^{kq}, u- v^{hp}\right\rangle_{\Gamma_C}
 + \langle V(\psi^{*}_{hp}-\psi^{hp}),\psi-\psi^{hp}\rangle_{\Gamma_\Sigma}  \\
&\qquad+\left\langle  \gamma(\lambda^{kq}+S_{hp} u^{hp}),S_{hp}( u^{hp}- v^{hp})\right\rangle_{\Gamma_C}
+\left\langle  \lambda^{kq}-\mu,u\right\rangle_{\Gamma_C}+\left\langle  \lambda-\mu^{kq},u^{hp}+\gamma(-\lambda^{kq}-S u^{hp}\right\rangle_{\Gamma_C} \\
&\qquad+\left\langle  \gamma(\lambda^{kq}-\mu^{kq}), E_{hp}( u^{hp})\right\rangle_{\Gamma_C}+\left\langle g,\mu_{n}^{kq}-\lambda_{n}^{kq}+\mu_{n}-\lambda_{n}\right\rangle_{\Gamma_C}
 -\left\langle  \gamma(\lambda-\lambda^{kq}),S( u- u^{hp})\right\rangle_{\Gamma_C}
\end{align*}
Note that
\begin{align*}
&\left\langle  \gamma(\lambda^{kq}+S_{hp} u^{hp}),S_{hp}( u^{hp}- v^{hp})\right\rangle_{\Gamma_C}-\left\langle  \gamma(\lambda-\lambda^{kq}),S( u- u^{hp})\right\rangle_{\Gamma_C}\\
&=
  -\left\langle  \gamma(\lambda-\lambda^{kq}),S( u-v^{hp})\right\rangle_{\Gamma_C} +\left\langle  \gamma(\lambda+Su^{hp}),S( u^{hp}- v^{hp})\right\rangle_{\Gamma_C}
  -\left\langle  \gamma(\lambda^{kq}+Su^{hp}),E_{hp}( u^{hp}- v^{hp})\right\rangle_{\Gamma_C}\\
\qquad  &-\left\langle  \gamma E_{hp}( u^{hp}),S( u^{hp}- v^{hp})\right\rangle_{\Gamma_C}
+\left\langle  \gamma E_{hp}( u^{hp}),E_{hp}( u^{hp}- v^{hp})\right\rangle_{\Gamma_C}\ .
\end{align*}
Finally, applying continuity of $S, S_{hp}, E_{hp}$ and Cauchy and Young's inequalities yields the assertion of the theorem.

\end{proof}

\end{theorem}

\begin{theorem}\label{thm:Lame:apriori2}
Let $( u, \lambda)\in \tilde{H}^{1/2}(\Gamma_\Sigma)\times M^+(\mathcal{F})$ with $u\in H^{1+\alpha}(\Gamma_\Sigma)$, $\lambda \in H^\alpha(\Gamma_C)$ and $(u^{hp}, \lambda^{kq}) \in \mathcal{V}_{hp} \times \tilde{M}^+_{kq}(\mathcal{F}) $ be the solutions of \eqref{eq:MixedProblem}, \eqref{eq:DiscreteMixedProblem}, respectively, with $g \equiv 0$ and $\alpha\in[0,\frac{1}{2})$. Suppose $\lVert\lambda_{n}\lVert_{H^{\alpha}(\Gamma_{C})}+\lVert\lambda_{t}\lVert_{H^{\alpha}(\Gamma_{C})}+ \lVert \mathcal{F}\lVert_{L^{2}(\Gamma_{C})} \lesssim\lVert u\lVert_{{ H}^{1+\alpha}(\Gamma_\Sigma)}$, then there exists a constant $C>0$ independent of $h$, $p$, $k$ and $q$, such that there holds with $ \psi$, $\psi^{hp}$ in \eqref{eq:Lame:defpsi}
\begin{align}\label{eq:Lame:apriori2}
&  \lVert u-u^{hp}\lVert_{\tilde{ H}^{\frac{1}{2}}(\Gamma_\Sigma)}+ \lVert \psi-\psi^{hp}\lVert_{\tilde{ H}^{-\frac{1}{2}}(\Gamma)}+
\lVert\gamma^{\frac{1}{2}}(\lambda-\lambda^{kq}) \lVert_{L^{2}(\Gamma_{C})}\\
 &\leq C\left(\frac{k^{\alpha+\frac{1}{2}}}{q^{\alpha+\frac{1}{2}}}+\frac{h^{\frac{1}{2}} k^{\alpha}}{pq^{\alpha}}\right)\lVert  u\lVert_{{ H}^{1+\alpha}(\Gamma_\Sigma)}
 +\inf_{\mu\in M^+(\mathcal{F})}\int_{\Gamma_{C}} (\lambda^{kq}-\mu)u\:ds \ .
\end{align}
\end{theorem}
\begin{proof}
We apply Theorem \ref{thm:Lame:apriori1} with $v = u^{hp}$. Employing Cauchy Schwarz and  Young's inequality with $\epsilon>0$, we note that
\begin{align*}
& \langle \lambda^{kq},u-v^{hp}\rangle_{\Gamma_C}=\langle \lambda^{kq}-\lambda,u-v^{hp}\rangle_{\Gamma_C}+\langle \lambda,u-v^{hp}\rangle_{\Gamma_C} \\
& \leq \lVert \lambda^{kq}-\lambda\lVert_{L_{2}(\Gamma_{C})}\lVert u-v^{hp}\lVert_{L_{2}(\Gamma_{C})}+
 \lVert \lambda\lVert_{L^{2}(\Gamma_{C})}\lVert u-v^{hp}\lVert_{L^{2}(\Gamma_{C})}\nonumber\\
  &\leq\frac{\epsilon}{2} \lVert\gamma^{\frac{1}{2}}(\lambda^{kq}-\lambda)\lVert^{2}_{L^{2}(\Gamma_{C})}+
  \frac{1}{2\epsilon\gamma_{0}}\frac{p^{2}}{h}\lVert u-v^{hp}\lVert^{2}_{L^{2}(\Gamma_{C})}+
  \lVert \lambda\lVert_{L^{2}(\Gamma_{C})}\lVert u-v^{hp}\lVert_{L^{2}(\Gamma_{C})}
\end{align*}
Setting $v^{hp}=\mathcal{I}_{hp} u$, where $\mathcal{I}_{hp}$ is the Lagrange interpolation operator, we have
\begin{align} \label{eq:Lame:estimateA3}
 \lVert u-v^{hp}\lVert^{2}_{\tilde{ H}^{\frac{1}{2}}(\Gamma_\Sigma)}
\leq C \frac{h^{1+2\alpha}}{p^{1+2\alpha}} \lVert u\lVert^{2}_{{ H}^{1+\alpha}(\Gamma_\Sigma)}, \qquad
 \lVert u-v^{hp}\lVert^{2}_{ H^{1}(\Gamma_\Sigma)}
\leq C \left(\frac{h}{p}\right)^{2\alpha} \lVert u\lVert^{2}_{{ H}^{1+\alpha}(\Gamma_\Sigma)}.
\end{align}
Thus
\begin{align}\label{eq:Lame:estimateA33}
 \lVert\gamma^{\frac{1}{2}}S( u-v^{hp})\lVert^{2}_{L^{2}(\Gamma_{C})}&\leq \alpha\lVert\gamma^{\frac{1}{2}}( u-v)\lVert^{2}_{ {H}^{1}(\Gamma_\Sigma)} \leq C\alpha\gamma_{0}\frac{h^{1+2\alpha}}{p^{2+2\alpha}} \lVert u\lVert^{2}_{{ H}^{1+\alpha}(\Gamma_\Sigma)}.
\end{align}
For $ u^{hp}-v^{hp}\in\mathcal{V}_{hp}$ we have
\begin{align}\label{eq:Lame:estimateA41}
 \frac{h}{p^{2}}\lVert S( u^{hp}-v^{hp})\lVert^{2}_{L^{2}(\Gamma_{C})}\leq
\alpha\frac{h}{p^{2}}\lVert u^{hp}-v^{hp}\lVert^{2}_{ \tilde{H}^{1}(\Gamma_\Sigma)}
\leq\alpha\lVert u^{hp}-v^{hp}\lVert^{2}_{\tilde{ H}^{\frac{1}{2}}(\Gamma_\Sigma)}.
\end{align}
Again using Cauchy Schwarz and Young's inequality yields with $\lambda=-Su$
\begin{align}\label{eq:Lame:estimateA42}
\langle\gamma Su^{hp}+\lambda,S(u^{hp}-v^{hp})\rangle_{\Gamma_C}&\leq\frac{1}{2}\gamma_{0}\frac{h}{p^{2}}\lVert S(u^{hp}-v^{hp})\lVert^{2}_{L^{2}(\Gamma_{C})}+
\frac{1}{2}\gamma_{0}\frac{h}{p^{2}}\lVert S( u-v^{hp}+v^{hp}- u^{hp})\lVert^{2}_{L^{2}(\Gamma_{C})}\nonumber\\
&\leq\frac{3}{2}\gamma_{0}\frac{h}{p^{2}}\lVert S( u^{hp}-v^{hp})\lVert^{2}_{L^{2}(\Gamma_{C})}+\gamma_{0}\frac{h}{p^{2}}\lVert S( u-
v^{hp})\lVert^{2}_{L^{2}(\Gamma_{C})}\nonumber\\
&\leq C\left(\frac{h^{1+2\alpha}}{p^{2+2\alpha}}\lVert u\lVert^{2}_{{ H}^{1+\alpha}(\Gamma_\Sigma)}+
\frac{h^{1+2\alpha} }{p^{1+2\alpha} } \lVert u\lVert^{2}_{{ H}^{1+\alpha}(\Gamma_\Sigma)}
+\gamma_{0}\lVert u- u^{hp}\lVert^{2}_{\tilde{ H}^{\frac{1}{2}}(\Gamma_\Sigma)}\right).
\end{align}
Choosing $\mu^{kq}=\pi_{M_{kq}}\lambda$ with the  $L^2$-projection $\pi_{M_{kq}}$ onto $M_{kq}^+(\mathcal{F})$ we have
\begin{align}\label{eq:Lame:estimateA5nt}
 A:=\inf_{\mu^{kq} \in M_{kq}^+(\mathcal{F})}\langle \mu^{kq}-\lambda, u^{hp}+\gamma(-\lambda^{kq}- Su^{hp})\rangle_{\Gamma_{C}} \leq \langle (\pi_{M_{kq}}\lambda-\lambda), u^{hp}-\gamma (
 \lambda^{kq}+Su^{hp}) \rangle_{\Gamma_C}.
\end{align}
Using the $L^2$-orthogonality of the projection and standard approximation properties gives
\begin{align*}
 \langle (\pi_{M_{kq}}\lambda-\lambda),u^{hp} \rangle_{\Gamma_C} &= \langle (\pi_{M_{kq}}\lambda-\lambda),(u^{hp}-u) \rangle_{\Gamma_C}
 +\langle (\pi_{M_{kq}}\lambda-\lambda),(u-\pi_{M_{kq}}u)\rangle_{\Gamma_C} \\
&\leq C \left(\frac{k^{\frac{1}{2}+\alpha}}{q^{\frac{1}{2}+\alpha}}\lVert \lambda\lVert_{H^{\alpha}(\Gamma_{C})}\lVert  u- u^{hp}\lVert_{\tilde{ H}^{\frac{1}{2}}(\Gamma_\Sigma)}+
 \frac{k^{1+2\alpha}}{q^{1+2\alpha}}\lVert \lambda\lVert_{H^{\alpha}(\Gamma_{C})}\lVert  u\lVert_{{ H}^{1+\alpha}(\Gamma_\Sigma)}\right)
\end{align*}
Hence, employing Young's inequality we obtain
\begin{align*}
&\langle (\pi_{M_{kq}}\lambda-\lambda), -\gamma (  \lambda^{kq}+ Su^{hp}) \rangle_{\Gamma_C} \\
& = \langle \gamma(\pi_{M_{kq}}\lambda-\lambda),(-\lambda^{kq}+\lambda)\rangle_{\Gamma_C} +\langle\gamma(\pi_{M_{kq}}\lambda-\lambda), S( u-\mathcal{I}_{hp} u)\rangle_{\Gamma_C}+\langle\gamma(\pi_{M_{kq}}\lambda-\lambda),S(\mathcal{I}_{hp} u- u^{hp})\rangle_{\Gamma_C} \\
&\leq C \gamma_{0}^{\frac{1}{2}}\left(\frac{h}{p^{2}}\right)^{\frac{1}{2}}\frac{k^{\alpha}}{q^{\alpha}}\lVert \lambda\lVert_{H^{\alpha}(\Gamma_{C})} ( \lVert \gamma^{\frac{1}{2}}(\lambda-\lambda^{kq})\lVert_{L^{2}(\Gamma_{C})} + \lVert \gamma^{\frac{1}{2}}S( u-\mathcal{I}_{hp} u)\lVert_{L^{2}(\Gamma_{C})} \lVert \gamma^{\frac{1}{2}}S(\mathcal{I}_{hp} u- u^{hp})\lVert_{L^{2}(\Gamma_{C})}).
\end{align*}
Using Young's inequality, \eqref{eq:Lame:estimateA33}, \eqref{eq:Lame:estimateA3} and \eqref{eq:Lame:estimateA42}, we finally obtain
%
\begin{align}\label{eq:Lame:estimateA5nfinal}
 A&\leq C\Big(\epsilon\lVert \gamma^{\frac{1}{2}}(\lambda-\lambda^{kq})\lVert^{2}_{L_{2}(\Gamma_{C})}
+\gamma_{0}\lVert  u- u^{hp}\lVert^{2}_{\tilde{ H}^{\frac{1}{2}}(\Gamma_\Sigma)}\\&\qquad +\frac{k^{2\alpha+1}}{q^{2\alpha+1}}\lVert  u\lVert^{2}_{{ H}^{1+\alpha}(\Gamma_\Sigma)}+\frac{h k^{2\alpha}}{p^{2}q^{2\alpha}}\lVert  u\lVert^{2}_{{ H}^{1+\alpha}(\Gamma_\Sigma)}+
\frac{h^{2\alpha+1}}{p^{2\alpha+1}}\lVert u\lVert^{2}_{{ H}^{1+\alpha}(\Gamma_\Sigma)}\Big).
\end{align}
We now estimate the term
\begin{align}\label{eq:Lame:estimateA7n}
 B:= \inf_{v^{hp}\in \mathcal{V}_{hp}} \langle \gamma(-\lambda^{kq} + \lambda + Su-Su^{hp}),E_{hp}( u^{hp}-v^{hp}) \rangle_{\Gamma_C}\ .
\end{align}
As above we have 
\begin{align*}
\gamma_{0}\frac{h}{p^{2}}\lVert S( u- u^{hp})\lVert^{2}_{L^{2}(\Gamma_{C})}
\leq C\left(\frac{h^{1+2\alpha}}{p^{2+2\alpha}}\lVert u\lVert^{2}_{{ H}^{1+\alpha}(\Gamma_\Sigma)}+\frac{h^{1+2\alpha}}{p^{1+2\alpha}} \lVert u\lVert^{2}_{{ H}^{1+\alpha}(\Gamma_\Sigma)}
+\gamma_{0}\lVert u- u^{hp}\lVert^{2}_{\tilde{ H}^{\frac{1}{2}}(\Gamma_\Sigma)}\right)\ ,
\end{align*}
and with continuity of $E_{hp}=S-S_{hp}$
\begin{align*}
\gamma_{0}\frac{h}{p^{2}}\lVert E_{hp}( u^{hp}-v^{hp})\lVert^{2}_{L^{2}(\Gamma_{C})}
\leq C\left(\frac{h^{1+2\alpha}}{p^{1+2\alpha}} \lVert u\lVert^{2}_{{ H}^{1+\alpha}(\Gamma_\Sigma)}
+\gamma_{0}\lVert u- u^{hp}\lVert^{2}_{\tilde{ H}^{\frac{1}{2}}(\Gamma_\Sigma)}\right),
\end{align*}
yielding altogether
\begin{align}\label{eq:Lame:estimateA7n2}
 B\leq C\left(\frac{h^{1+2\alpha}}{p^{1+2\alpha}} \lVert u\lVert^{2}_{{ H}^{1+\alpha}(\Gamma_\Sigma)}
+\gamma_{0}\lVert u- u^{hp}\lVert^{2}_{\tilde{ H}^{\frac{1}{2}}(\Gamma_\Sigma)}+\epsilon\lVert \gamma^{\frac{1}{2}}(\lambda-\lambda^{kq})\lVert^{2}_{L_{2}(\Gamma_{C})}\right).
\end{align}
Similar arguments yield (see \cite{IssaouiDiss})
\begin{align}\label{eq:Lame:estimateA8n}
  \inf_{v^{hp}\in \mathcal{V}_{hp}} \langle \gamma E_{hp}( u^{hp}),S( u^{hp}-v^{hp}) \rangle_{\Gamma_C} &\leq C' \frac{h^{1+2\alpha} }{p^{1+2\alpha} } \lVert u\lVert^{2}_{{ H}^{1+\alpha}(\Gamma_\Sigma)}
 +\left[c \gamma_{0}^{2} (\frac{1}{2\epsilon}+\frac{1}{2})+\frac{\epsilon}{2} C_{E_{h}}\right]\lVert  u^{hp}- u\lVert^{2}_{\tilde{H}^{\frac{1}{2}}(\Gamma_\Sigma)} \\
  \inf_{v^{hp}\in \mathcal{V}_{hp}} \langle \gamma E_{hp}( u^{hp}),E_{hp}( u^{hp}-v^{hp})\rangle_{\Gamma_C}  &\leq C' \frac{h^{1+2\alpha} }{p^{1+2\alpha} } \lVert u\lVert^{2}_{{ H}^{1+\alpha}(\Gamma_\Sigma)}  +\left[c \gamma_{0}^{2} (\frac{1}{2\epsilon}+\frac{1}{2})+\frac{\epsilon}{2} C_{E_{hp}}\right]\lVert  u^{hp}- u\lVert^{2}_{\tilde{H}^{\frac{1}{2}}(\Gamma_\Sigma)}
\end{align}
In order to estimate the term
\begin{align}\label{eq:Lame:estimateAn9}
 C:=\inf_{\mu^{kq}\in M_{kq}^+(\mathcal{F})} \langle \gamma(\lambda^{kq}-\mu^{kq}), E_{hp}( u^{hp}) \rangle_{\Gamma_C}
\end{align}
we write
\begin{align}\label{eq:Lame:estimateAn91}
 \langle \gamma(\lambda^{kq}-\mu^{kq}), E_{hp} u^{hp} \rangle_{\Gamma_C}= \langle \gamma(\lambda^{kq}-\lambda), E_{hp} u^{hp} \rangle_{\Gamma_C}+ \langle \gamma(\lambda-\mu^{kq}) ,E_{hp}u^{hp} \rangle_{\Gamma_C}
\end{align}
and estimate the two terms separately. Inserting $u^{hp}=u^{hp}-v^{hp}+v^{hp}-u+u$, we have
\begin{align*}
\langle \gamma(\lambda^{kq}-\lambda), E_{hp} u^{hp} \rangle_{\Gamma_C} &
\leq \frac{3\epsilon}{2}\lVert \gamma^{\frac{1}{2}}(\lambda-\lambda^{kq})\lVert^{2}_{L_{2}(\Gamma_{C})}\\ &\qquad +\frac{\gamma_{0} h}{\epsilon p^{2}}
 \lVert  E_{hp}( u^{hp}-v^{hp})\lVert^{2}_{L^{2}(\Gamma_{C})}+\frac{\gamma_{0} h}{\epsilon p^{2}}\lVert  E_{hp}( u-v^{hp})\lVert^{2}_{L^{2}(\Gamma_{C})}+\frac{\gamma_{0} h}{\epsilon p^{2}}\lVert  E_{hp} u\lVert^{2}_{L^{2}(\Gamma_{C})} \\
& \leq \frac{3\epsilon}{2}\lVert \gamma^{\frac{1}{2}}(\lambda-\lambda^{kq})\lVert^{2}_{L^{2}(\Gamma_{C})}+
\frac{\gamma_{0}}{\epsilon}  \frac{h^{1+2\alpha}}{p^{2+2\alpha}}  \lVert u\lVert^{2}_{{ H}^{1+\alpha}(\Gamma_\Sigma)}\\ & \qquad+
\frac{\gamma_{0}}{\epsilon}  \frac{h^{1+2\alpha}}{p^{1+2\alpha}}  \lVert u\lVert^{2}_{{ H}^{1+\alpha}(\Gamma_\Sigma)}+\frac{\gamma_{0}}{\epsilon}\lVert  u^{hp}- u\lVert^{2}_{H^{\frac{1}{2}}(\Gamma_\Sigma)}+ \alpha\frac{\gamma_{0}}{\epsilon}  \frac{h}{p^{2}}  \lVert u\lVert^{2}_{{ H}^{1+\alpha}(\Gamma_\Sigma)}
\end{align*}
The second term in \eqref{eq:Lame:estimateAn91} is estimated by
\begin{align}\label{eq:Lame:estimateAn9122}
 \left\langle \gamma(\lambda-\mu^{kq}), E_{hp} u^{hp}\right\rangle_{\Gamma_C} &\leq \gamma_{0} \frac{h^{\frac{1}{2}}}{p} \lVert \lambda-\mu^{kq}\lVert_{L_{2}(\Gamma_{C})}\frac{h^{\frac{1}{2}}}{p}\lVert E_{hp} u^{hp}\lVert^{2}_{L_{2}(\Gamma_{C})} \leq \gamma_{0}\frac{h^{\frac{1}{2}}}{p}\frac{k^{\alpha}}{q^{\alpha}}\lVert \lambda\lVert_{H^{\alpha}(\Gamma_{C})}\lVert u\lVert_{{ H}^{1+\alpha}(\Gamma_\Sigma)}.
\end{align}
Hence, the term C in \eqref{eq:Lame:estimateAn9} is bounded by the sum of the previous two right hand sides.
Note that $\gamma_{0}$ is sufficiently small and, hence, moving the terms
$\gamma_{0}\lVert  u- u^{h}\lVert^{2}_{\tilde{ H}^{\frac{1}{2}}(\Gamma_\Sigma)}$ and $\epsilon\lVert \gamma^{\frac{1}{2}}(\lambda -\lambda ^{hp})\lVert^{2}_{L_{2}(\Gamma_{C})}$ to the left hand side, we obtain the a priori error estimate of the theorem.
\end{proof}

For the conforming approximation of \eqref{eq:DiscreteMixedProblem} by Bernstein polynomials $\lambda^{kq}\in M_{kq}^+(\mathcal{F})\subset M^+(\mathcal{F})$, the term
\begin{align}
  \inf_{\mu\in M^+(\mathcal{F})}\int_{\Gamma_{C}} (\lambda^{kq}-\mu)u\:ds=0
\end{align}
vanishes. {However, the properties of a corresponding quasi--interpolation operator to replace $\pi_{M_{kq}}$ do not seem to be available in the literature. Assuming that one can define an $H^k$-stable quasi-interpolation operator $\tilde{\pi}_{M_{kq}}: L^2(\Gamma_C)\cap M^+(\mathcal{F}) \rightarrow M^+_{kq}(\mathcal{F})$, such that $\tilde{\pi}_{M_{kq}}$ satisfies the approximation property
\begin{align}
\left\| \eta - \tilde{\pi}_{M_{kq}} \eta \right\|_{H^{k}(\Gamma_C)} \leq C \left(\frac{h}{p}\right)^{l+1-k} \left| \eta \right|_{H^{l+1}(\Gamma_C)}\ ,
\end{align}
the proof of Theorem~\ref{thm:Lame:apriori2} yields: }
{
\begin{remark}\label{thm:Lame:apriori22}
Let $( u, \lambda)\in \tilde{H}^{1/2}(\Gamma_\Sigma)\times M^+(\mathcal{F})$ be the solution of the problem \eqref{eq:MixedProblem}  and $(u^{hp}, \lambda^{kq})$ the solution of the discrete problem  \eqref{eq:DiscreteMixedProblem} with Bernstein polynomials, i.e.~$\lambda^{kq}\in M_{kq}^+(\mathcal{F})$. Under the same assumptions as in Theorem~\ref{thm:Lame:apriori2} there holds with a constant $C>0$ independent of $h$, $p$, $k$ and $q$
\begin{align}\label{eq:Lame:aprioriBernstein}
  \lVert u-u^{hp}\lVert_{\tilde{ H}^{\frac{1}{2}}(\Gamma_\Sigma)}+
\lVert\gamma^{\frac{1}{2}}(\lambda-\lambda^{kq}) \lVert_{L_{2}(\Gamma_{C})}
 \leq C\left(\frac{k^{\alpha+\frac{1}{2}}}{q^{\alpha+\frac{1}{2}}}+\frac{h^{\frac{1}{2}} k^{\alpha}}{pq^{\alpha}}\right)\lVert  u\lVert_{{ H}^{1+\alpha}(\Gamma_\Sigma)}\ .
\end{align}
\end{remark}
}

\begin{remark}
 Our convergence analysis (Theorem~\ref{thm:Lame:apriori2}) for the $h$-version with $p=q=1$ covers the result of Hild and Renard \cite{hildrenard2010} for the FEM.
\end{remark}

\section{A posteriori error estimates}
\label{sec:Aposteriorierrorestimations}
In this section we present an a posteriori error estimate of residual type for the mixed hp-BEM scheme.

\begin{lemma} \label{lem:aposterioriContactPart}
Let $(u,\lambda)$, $(u^{hp},\lambda^{kq})$ be the solution of \eqref{eq:MixedProblem}, \eqref{eq:DiscreteMixedProblem} respectively. Then there holds
\begin{align*}
&\left\langle \lambda-\lambda^{kq}, u^{hp}-u\right\rangle_{\Gamma_C} \\& \leq
\left\langle \left(\lambda^{kq}_n\right)^+,\left(g-u^{hp}_n\right)^+\right\rangle_{\Gamma_C}+ \left\|\lambda^{kq}_n -\lambda_n \right\|_{-\frac{1}{2},\Gamma_C} \left\| \left(g-u^{hp}_n\right)^-\right\|_{\frac{1}{2},\Gamma_C} \\ & \qquad+ \left\|\left(\lambda^{kq}_n\right)^-\right\|_{-\frac{1}{2},\Gamma_C} \left\|u^{hp}_n-u_n \right\|_{\frac{1}{2},\Gamma_C} 
+ \left\| \left(\left\|\lambda^{kq}_t\right\|_2-\mathcal{F} \right)^+ \right\|_{-\frac{1}{2},\Gamma_C} \left\| \:\left\|u_t-u^{hp}_t\right\|_2 \right\|_{\frac{1}{2},\Gamma_C}\\ &\qquad  -\left\langle \left(\left\|\lambda^{kq}_t\right\|_2-\mathcal{F} \right)^-, \left\|u^{hp}_t\right\|_2 \right\rangle_{\Gamma_C} 
-\left\langle  \lambda^{kq}_t,u^{hp}_t \right\rangle_{\Gamma_C}
 +\left\langle \left\|\lambda^{kq}_t\right\|_2 , \left\|u^{hp}_t\right\|_2 \right\rangle_{\Gamma_C}
\end{align*}
where $v^+=\max\left\{0,v\right\}$ and $v^-=\min\left\{0,v\right\}$, i.e.~$v=v^+ +v^-$.
\begin{proof}
Utilizing that $\left\langle \lambda_n,u_n-g\right\rangle_ {\Gamma_C}=0$ by \eqref{eq:ContContactConstraints}, $ u_n-g \leq 0$ almost everywhere in $\Gamma_C$ and $ \left(\lambda^{kq}_n\right)^+ \in L^2(\Gamma_C)$ where $v^+=\max\left\{0,v\right\}$ and $v^-=\min\left\{0,v\right\}$, i.e.~$v=v^+ +v^-$, there holds
\begin{align*}
  \left\langle \lambda_n-\lambda^{kq}_n,u^{hp}_n-u_n\right\rangle_{\Gamma_C} &= \left\langle \lambda_n-\left(\lambda^{kq}_n\right)^+,u^{hp}_n-g\right\rangle_{\Gamma_C} + \left\langle \lambda_n,g-u_n\right\rangle_{\Gamma_C} \\&\qquad -\left\langle \left(\lambda^{kq}_n\right)^+,g-u_n\right\rangle_{\Gamma_C} - \left\langle \left(\lambda^{kq}_n\right)^-,u^{hp}_n-u_n\right\rangle_{\Gamma_C}\\
	&\leq \left\langle \lambda_n-\left(\lambda^{kq}_n\right)^+,u^{hp}_n-g\right\rangle_{\Gamma_C}  - \left\langle \left(\lambda^{kq}_n\right)^-,u^{hp}_n-u_n\right\rangle_{\Gamma_C}
\end{align*}
and with $\lambda \in M^+(\mathcal{F})$
\begin{align*}
\left\langle \lambda_n-\left(\lambda^{kq}_n\right)^+,u^{hp}_n-g\right\rangle_{\Gamma_C} &=
\left\langle \left(\lambda^{kq}_n\right)^+,g-u^{hp}_n\right\rangle_{\Gamma_C} + \left\langle -\lambda_n,\left(g-u^{hp}_n\right)^+ +  \left(g-u^{hp}_n\right)^- \right\rangle_{\Gamma_C} \\
&\leq \left\langle \left(\lambda^{kq}_n\right)^+,g-u^{hp}_n\right\rangle_{\Gamma_C} + \left\langle \lambda^{kq}_n -\lambda_n - \left(\lambda^{kq}_n\right)^+ -\left(\lambda^{kq}_n\right)^-,  \left(g-u^{hp}_n\right)^- \right\rangle_{\Gamma_C} \\
&= \left\langle \left(\lambda^{kq}_n\right)^+,\left(g-u^{hp}_n\right)^+\right\rangle_{\Gamma_C} + \left\langle \lambda^{kq}_n -\lambda_n ,  \left(g-u^{hp}_n\right)^- \right\rangle_{\Gamma_C}\\ & \qquad - \left\langle \left(\lambda^{kq}_n\right)^-,  \left(g-u^{hp}_n\right)^- \right\rangle_{\Gamma_C}
\\
&\leq \left\langle \left(\lambda^{kq}_n\right)^+,\left(g-u^{hp}_n\right)^+\right\rangle_{\Gamma_C} + \left\langle \lambda^{kq}_n -\lambda_n ,  \left(g-u^{hp}_n\right)^- \right\rangle_{\Gamma_C} .
\end{align*}
Application of Cauchy-Schwarz inequality yields
\begin{align*}
\left\langle \lambda_n-\lambda^{kq}_n,u^{hp}_n-u_n\right\rangle_{\Gamma_C} &\leq  \left\langle \left(\lambda^{kq}_n\right)^+,\left(g-u^{hp}_n\right)^+\right\rangle_{\Gamma_C}\\ &\qquad + \left\|\lambda^{kq}_n -\lambda_n \right\|_{-\frac{1}{2},\Gamma_C} \left\| \left(g-u^{hp}_n\right)^-\right\|_{\frac{1}{2},\Gamma_C} + \left\|\left(\lambda^{kq}_n\right)^-\right\|_{-\frac{1}{2},\Gamma_C} \left\|u^{hp}_n-u_n \right\|_{\frac{1}{2},\Gamma_C}.
\end{align*}
For the tangential component there holds by exploiting $ \left\langle \lambda_t,u_t\right\rangle_{\Gamma_C} = \left\langle \mathcal{F}, \left\|u_t\right\|_2\right\rangle_{\Gamma_C} $, $\left\langle \lambda_t, u^{hp}_t\right\rangle_{\Gamma_C} \leq \left\langle \mathcal{F}, \left\|u^{hp}_t\right\|_2 \right\rangle_{\Gamma_C}$, \mbox{$v=v^+ + v^-$} and triangle inequality that
\begin{align*}
\left\langle \lambda_t-\lambda^{kq}_t\right.&\left.,u^{hp}_t-u_t\right\rangle_{\Gamma_C} \leq  \left\langle -\mathcal{F} ,\left\|u_t\right\|_2 \right\rangle_{\Gamma_C} + \left\langle \lambda^{kq}_t , u_t \right\rangle_{\Gamma_C} +\left\langle \mathcal{F} , \left\|u^{hp}_t\right\|_2 \right\rangle_{\Gamma_C}-\left\langle  \lambda^{kq}_t,u^{hp}_t \right\rangle_{\Gamma_C} \\
 \leq & \left\langle \left(\left\|\lambda^{kq}_t\right\|_2-\mathcal{F} \right)^+,\left\|u_t\right\|_2 \right\rangle_{\Gamma_C}  +\left\langle \mathcal{F} , \left\|u^{hp}_t\right\|_2 \right\rangle_{\Gamma_C}-\left\langle  \lambda^{kq}_t,u^{hp}_t \right\rangle_{\Gamma_C} \\
 \leq & \left\langle \left(\left\|\lambda^{kq}_t\right\|_2-\mathcal{F} \right)^+,\left\|u_t-u^{hp}_t\right\|_2 \right\rangle_{\Gamma_C}  +\left\langle \left(\left\|\lambda^{kq}_t\right\|_2-\mathcal{F} \right)^+ + \mathcal{F} , \left\|u^{hp}_t\right\|_2 \right\rangle_{\Gamma_C}-\left\langle  \lambda^{kq}_t,u^{hp}_t \right\rangle_{\Gamma_C}  \\
=&  \left\langle \left(\left\|\lambda^{kq}_t\right\|_2-\mathcal{F} \right)^+,\left\|u_t-u^{hp}_t\right\|_2 \right\rangle_{\Gamma_C} \\& \qquad -\left\langle \left(\left\|\lambda^{kq}_t\right\|_2-\mathcal{F} \right)^-, \left\|u^{hp}_t\right\|_2 \right\rangle_{\Gamma_C} -\left\langle  \lambda^{kq}_t,u^{hp}_t \right\rangle_{\Gamma_C}
 +\left\langle \left\|\lambda^{kq}_t\right\|_2 , \left\|u^{hp}_t\right\|_2 \right\rangle_{\Gamma_C} \\
 \leq & \left\| \left(\left\|\lambda^{kq}_t\right\|_2-\mathcal{F} \right)^+ \right\|_{-\frac{1}{2},\Gamma_C} \left\| \:\left\|u_t-u^{hp}_t\right\|_2 \right\|_{\frac{1}{2},\Gamma_C} \\& \qquad -\left\langle \left(\left\|\lambda^{kq}_t\right\|_2-\mathcal{F} \right)^-, \left\|u^{hp}_t\right\|_2 \right\rangle_{\Gamma_C}
-\left\langle  \lambda^{kq}_t,u^{hp}_t \right\rangle_{\Gamma_C}
 +\left\langle \left\|\lambda^{kq}_t\right\|_2 , \left\|u^{hp}_t\right\|_2 \right\rangle_{\Gamma_C}
\end{align*}

\end{proof}
\end{lemma}

\begin{lemma} \label{thm:aposterioriError_only_u}
Let $(u,\lambda)$, $(u^{hp},\lambda^{kq})$ be the solution of \eqref{eq:MixedProblem}, \eqref{eq:DiscreteMixedProblem} respectively. Then there exists a constant $C$ independent of $h$, $p$, $k$ and $q$ such that
\begin{align*}
&C\left(\left\|u-u^{hp}\right\|^2_{1/2,\Gamma_\Sigma} + \left\|\psi-\psi^{hp} \right\|^2_{-1/2,\Gamma_\Sigma} \right)
\\& \leq
 \sum_{E\in \mathcal{T}_h \cap \Gamma_N} \left(\frac{h_E}{p_E}\right) \left\|  t-S_{hp}u^{hp}\right\|^2_{0,E} +  \sum_{E\in \mathcal{T}_h \cap \Gamma_C} \left(\frac{h_E}{p_E} + \frac{h_{E}}{p_{E}^{2}}  \right)\left\|   -\lambda^{kq}-S_{hp}u^{hp}\right\|^2_{0,E}
\\
&\qquad + \sum_{E \in \mathcal{T}_{h}}h_E \left\|\frac{\partial}{\partial s} \left(V\psi^{hp} -(K+\frac{1}{2})u^{hp}\right) \right\|^2_{L^2(E)}
+\left\langle \left(\lambda^{kq}_n\right)^+,\left(g-u^{hp}_n\right)^+\right\rangle_{\Gamma_C}\\ \qquad &+
\epsilon \left\|\lambda^{kq}_n -\lambda_n \right\|^2_{-\frac{1}{2},\Gamma_C} + \frac{1}{4\epsilon} \left\| \left(g-u^{hp}_n\right)^-\right\|^2_{\frac{1}{2},\Gamma_C} \\
&\qquad+ \left\|\left(\lambda^{kq}_n\right)^-\right\|^2_{-\frac{1}{2},\Gamma_C} + \left\| \left(\left\|\lambda^{kq}_t\right\|_2-\mathcal{F} \right)^+ \right\|^2_{-\frac{1}{2},\Gamma_C}\\& \qquad  -\left\langle \left(\left\|\lambda^{kq}_t\right\|_2-\mathcal{F} \right)^-, \left\|u^{hp}_t\right\|_2 \right\rangle_{\Gamma_C}
-\left\langle  \lambda^{kq}_t,u^{hp}_t \right\rangle_{\Gamma_C}\\ & \qquad
 +\left\langle \left\|\lambda^{kq}_t\right\|_2 , \left\|u^{hp}_t\right\|_2 \right\rangle_{\Gamma_C}
\end{align*}
with $\epsilon>0$ arbitrary.

\begin{proof}
Since $u-u^{hp} \in \tilde{H}^{1/2}(\Gamma_\Sigma)$ there holds
\begin{align*}
C\left(\left\|u-u^{hp}\right\|^2_{1/2,\Gamma_\Sigma} + \left\|\psi-\psi^{hp} \right\|^2_{-1/2,\Gamma_\Sigma} \right) &\leq \left\langle W(u-u^{hp}),u-u^{hp}\right\rangle_{\Gamma_\Sigma} + \left\langle V(\psi-\psi^{hp}),\psi-\psi^{hp}\right\rangle_{\Gamma_\Sigma} \\
&=\left\langle Su-S_{hp}u^{hp},u-u^{hp}\right\rangle_{\Gamma_\Sigma} + \left\langle V(\psi_{hp}^*-\psi^{hp}),\psi-\psi^{hp}\right\rangle_{\Gamma_\Sigma}
\end{align*}
From Lemma~\ref{lem:galerkinOrtho} and \ref{eq:WeakMixedVarEq} it follows that
\begin{align*}
\left\langle Su-\right.\left.S_{hp}u^{hp},u-u^{hp}\right\rangle_{\Gamma_\Sigma} &= \left\langle Su-S_{hp}u^{hp},u-u^{hp}\right\rangle_{\Gamma_\Sigma} +  \langle Su-S_{hp} u^{hp} ,u^{hp}-v^{hp}\rangle_{\Gamma_\Sigma} \\&+ \left\langle \lambda-\lambda^{kq},u^{hp}-v^{hp}\right\rangle_{\Gamma_C} 
+ \left\langle \gamma(\lambda^{kq}+S_{hp}u^{hp}),S_{hp}(u^{hp}-v^{hp}) \right\rangle_{\Gamma_C}\\
=& \left\langle t-S_{hp}u^{hp},u-v^{hp}\right\rangle_{\Gamma_N} +  \left\langle -\lambda^{kq}-S_{hp}u^{hp},u-v^{hp}\right\rangle_{\Gamma_C}
 + \left\langle \lambda-\lambda^{kq},u^{hp}-u\right\rangle_{\Gamma_C}\\ & + \left\langle  \gamma(\lambda^{kq}+S_{hp}u^{hp}),S_{hp}(u^{hp}-v^{hp}) \right\rangle_{\Gamma_C}.
\end{align*}
Let $I_{hp}$ be the Clement-Interpolation operator mapping onto $\mathcal{V}_{hp}$ with the property (see \cite{melenk2001residual} and interpolation between $L^2$ and $H^1$)
\begin{align*}
\left\|v-I_{hp} v\right\|_{L^2(E)} \leq C \left(\frac{h_E}{p_E}\right)^{1/2} \left\|v\right\|_{H^{1/2}(\omega(E))}.
\end{align*}
with $\omega(E)$ a net around $E$. Then, an application of the Cauchy-Schwarz inequality yields with $v^{hp}:=u^{hp}+ I_{hp}(u-u^{hp})$
\begin{align*}
\left\langle t-S_{hp}u^{hp},u-v^{hp}\right\rangle_{\Gamma_N}  &\leq C \sum_{E\in \mathcal{T}_h \cap \Gamma_N} \left(\frac{h_E}{p_E}\right)^{1/2} \left\|  t-S_{hp}u^{hp}\right\|_{0,E} \left\|u-u^{hp}\right\|_{1/2,\omega(E)} \\
\left\langle -\lambda^{kq}-S_{hp}u^{hp},u-v^{hp}\right\rangle_{\Gamma_C} &\leq  C \sum_{E\in \mathcal{T}_h \cap \Gamma_C} \left(\frac{h_E}{p_E}\right)^{1/2} \left\|   -\lambda^{kq}-S_{hp}u^{hp}\right\|_{0,E} \left\|u-u^{hp}\right\|_{1/2,\omega(E)}
\end{align*}
Since $u_{hp} \in \mathcal{V}_{hp} \subset H^{1}_0(\Gamma_\Sigma)$ and $\psi_{hp} \in V^D_{hp} \subset L^2(\Gamma_\Sigma)$
the mapping properties of $V$ and $K$ \cite{costabel1988boundary} yield
\[
V (\psi^{hp} - \psi_{hp}^*) = V \psi^{hp} - (K+\frac{1}{2}) u^{hp} \in H^1(\Gamma_\Sigma) \subset C^0(\Gamma_\Sigma).
\]
Furthermore, $ V (\psi^{hp} - \psi_{hp}^*) $ is orthogonal in $L^2(\Gamma_\Sigma)$ to $V^D_{hp}$, Lemma~\ref{lem:V_ortho}. Hence, for the characteristic function $\chi_E \in V^D_{hp}$ of an element $E \in \mathcal{T}_{h}$ there holds
\[
0 =\left\langle  V (\psi^{hp} - \psi_{hp}^*),\chi_E \right\rangle_{\Gamma_\Sigma} = \int_E V (\psi^{hp} - \psi_{hp}^*)\;ds\ ,
\]
and therefore the continuous function $ V (\psi_{hp} - \psi_{hp}^*) $ has a root on each boundary segment~$E$. Since $V (\psi_{hp} - \psi_{hp}^*) \in H^1(\Gamma_\Sigma)$, the application of \cite[Theorem 5.1]{carstensen1996posteriori} yields
\begin{align*}
\left\langle V (\psi^{hp} - \psi_{hp}^*), \psi - \psi^{hp} \right\rangle_{\Gamma_\Sigma}
&\leq \left\| V (\psi^{hp} - \psi_{hp}^*) \right\|_{{H}^{\frac{1}{2}}(\Gamma_\Sigma)} \left\| \psi^{hp} - \psi \right\|_{\tilde{H}^{-\frac{1}{2}}(\Gamma_\Sigma)} \\
& \leq C \left(\sum_{E \in \mathcal{T}_{h}}h_E \left\|\frac{\partial}{\partial s} (V(\psi^{hp} - \psi_{hp}^*) \right\|^2_{L^2(E)}\right)^{\frac{1}{2}}\left\| \psi^{hp} - \psi \right\|_{\tilde{H}^{-\frac{1}{2}}(\Gamma_\Sigma)}.
\end{align*}
Since  $v^{hp}=u^{hp}+I_{hp}(u-u^{hp})$, there holds by Cauchy-Schwarz inequality (twice), Theorem~\ref{thm:inverseEstimate} and the $H^{1/2}$-stability of $I_{hp}$ that
 \begin{align*}
  & \left\langle  \gamma(\lambda^{kq}+S_{hp}u^{hp}),S_{hp}(u^{hp}-v^{hp}) \right\rangle_{\Gamma_C} \\
  &= \gamma_{0} \sum_{E\in\Gamma_{C}}\int_{E}\left(\frac{h_{E}^{\frac{1}{2}}}{p_{E}}\right)(\lambda^{kq}+S_{hp}u^{hp})\left(\frac{h_{E}^{\frac{1}{2}}}{p_{E}}\right)S_{hp}(I_{hp}(u^{hp}-u)) \:ds \\
  &\leq\gamma_{0}\left(\sum_{E\in\Gamma_{C}}\frac{h_{E}}{p_{E}^{2}}\left\|\lambda^{kq}+S_{hp}u^{hp}\right\|^{2}_{L^{2}(E)}\right)^{\frac{1}{2}}\left(\sum_{E\in\Gamma_{C}} \left\|\frac{h_{E}^{\frac{1}{2}}}{p_{E}}S_{hp}(I_{hp}(u^{hp}-u))\right\|^{2}_{L^{2}(E)}\right)^{\frac{1}{2}}\\
  &\leq C \left(\sum_{E\in\Gamma_{C}}\frac{h_{E}}{p_{E}^{2}} \left\|\lambda^{kq}+Su^{hp} \right\|^{2}_{L^2(E)}\right)^{\frac{1}{2}} \left\| u-u^{hp} \right\|_{\tilde{H}^{\frac{1}{2}}(\Gamma_\Sigma)}.
\end{align*}
In total this yields with Lemma~\ref{lem:aposterioriContactPart} that
\begin{align*}
&C\left(\left\|u-u^{hp}\right\|^2_{1/2,\Gamma_\Sigma} + \left\|\psi-\psi^{hp} \right\|^2_{-1/2,\Gamma_\Sigma} \right)\\&
 \leq \sum_{E\in \mathcal{T}_h \cap \Gamma_N} \left(\frac{h_E}{p_E}\right)^{1/2} \left\|  t-S_{hp}u^{hp}\right\|_{0,E} \left\|u-u^{hp}\right\|_{1/2,\omega(E)} \\
&+  \sum_{E\in \mathcal{T}_h \cap \Gamma_C} \left(\frac{h_E}{p_E}\right)^{1/2} \left\|   -\lambda^{kq}-S_{hp}u^{hp}\right\|_{0,E} \left\|u-u^{hp}\right\|_{1/2,\omega(E)} \\
&+
\left(\sum_{E\in\Gamma_{C}}\frac{h_{E}}{p_{E}^{2}} \left\|\lambda^{kq}+Su^{hp} \right\|^{2}_{L^2(E)}\right)^{\frac{1}{2}} \left\| u-u^{hp} \right\|_{\tilde{H}^{\frac{1}{2}}(\Gamma_\Sigma)}\\
&+\left(\sum_{E \in \mathcal{T}_{h}}h_E \left\|\frac{\partial}{\partial s} (V(\psi^{hp} - \psi_{hp}^*) \right\|^2_{L^2(E)}\right)^{\frac{1}{2}}\left\| \psi^{hp} - \psi \right\|_{\tilde{H}^{-\frac{1}{2}}(\Gamma_\Sigma)}\\
&+\left\langle \left(\lambda^{kq}_n\right)^+,\left(g-u^{hp}_n\right)^+\right\rangle_{\Gamma_C}+ \left\|\lambda^{kq}_n -\lambda_n \right\|_{-\frac{1}{2},\Gamma_C} \left\| \left(g-u^{hp}_n\right)^-\right\|_{\frac{1}{2},\Gamma_C} \\
&+ \left\|\left(\lambda^{kq}_n\right)^-\right\|_{-\frac{1}{2},\Gamma_C} \left\|u^{hp}_n-u_n \right\|_{\frac{1}{2},\Gamma_C} + \left\| \left(\left\|\lambda^{kq}_t\right\|_2-\mathcal{F} \right)^+ \right\|_{-\frac{1}{2},\Gamma_C} \left\| \:\left\|u_t-u^{hp}_t\right\|_2 \right\|_{\frac{1}{2},\Gamma_C}\\ &  -\left\langle \left(\left\|\lambda^{kq}_t\right\|_2-\mathcal{F} \right)^-, \left\|u^{hp}_t\right\|_2 \right\rangle_{\Gamma_C}
-\left\langle  \lambda^{kq}_t,u^{hp}_t \right\rangle_{\Gamma_C}
 +\left\langle \left\|\lambda^{kq}_t\right\|_2 , \left\|u^{hp}_t\right\|_2 \right\rangle_{\Gamma_C}\ .
\end{align*}
The assertion follows with Young's inequality.
\end{proof}
\end{lemma}

\begin{lemma} \label{lem:aposteriori_lambda}
Let $(u,\lambda)$, $(u^{hp},\lambda^{kq})$ be the solution of \eqref{eq:MixedProblem}, \eqref{eq:DiscreteMixedProblem} respectively. Then there holds
\begin{align*}
\frac{ \tilde{\beta}}{C} \left\| \lambda-\lambda^{kq} \right\|_{\tilde{H}^{-1/2}(\Gamma_C)}  \leq & \left\|u-u^{hp}\right\|_{\tilde{H}^{1/2}(\Gamma_\Sigma)}  + \left\|\psi-\psi^{hp}\right\|_{\tilde{H}^{-1/2}(\Gamma_\Sigma)}  + \left(\sum_{E\in\Gamma_{C}}\frac{h_{E}}{p_{E}^{2}} \left\|\lambda^{kq}+Su^{hp} \right\|^{2}_{L^2(E)}\right)^{\frac{1}{2}}  \\
&+ \left(\sum_{E\in \mathcal{T}_h \cap \Gamma_N} \frac{h_E}{p_E} \left\|  t-S_{hp}u^{hp}\right\|^2_{0,E} \right)^{1/2} + \left( \sum_{E\in \mathcal{T}_h \cap \Gamma_C}  \frac{h_E}{p_E} \left\|   -\lambda^{kq}-S_{hp}u^{hp}\right\|^2_{0,E} \right)^{1/2}
 \end{align*}
\begin{proof}
Let $v \in \tilde{H}^{1/2}(\Gamma_\Sigma)$ and $v^{hp}:=I_{hp}v \in \mathcal{V}_{hp}$, then by Lemma~\ref{lem:galerkinOrtho} and \eqref{eq:WeakMixedVarEq} there holds
\begin{align*}
\left\langle \lambda-\lambda^{kq},v\right\rangle_{\Gamma_C} &= \left\langle \lambda-\lambda^{kq},v-v^{hp}\right\rangle_{\Gamma_C} - \langle Su-S_{hp} u^{hp} ,v^{hp}\rangle_{\Gamma_\Sigma} -  \left\langle \gamma(\lambda^{kq}+S_{hp}u^{hp}),S_{hp}v^{hp} \right\rangle_{\Gamma_C} \\
&= \left\langle t,v-v^{hp}\right\rangle_{\Gamma_N} - \left\langle Su,v-v^{hp}\right\rangle_{\Gamma_\Sigma} -\left\langle \lambda^{kq},v-v^{hp}\right\rangle_{\Gamma_C} - \langle Su-S_{hp} u^{hp} ,v^{hp}\rangle_{\Gamma_\Sigma} \\ &\qquad -  \left\langle \gamma(\lambda^{kq}+S_{hp}u^{hp}),S_{hp}v^{hp}  \right\rangle_{\Gamma_C}\\
&= \left\langle t-S_{hp}u^{hp},v-v^{hp}\right\rangle_{\Gamma_N} +\left\langle -\lambda^{kq}-S_{hp}u^{hp},v-v^{hp}\right\rangle_{\Gamma_C}\\&\qquad - \left\langle Su-S_{hp} u^{hp},v\right\rangle_{\Gamma_\Sigma}  -  \left\langle \gamma(\lambda^{kq}+S_{hp}u^{hp}),S_{hp}v^{hp}  \right\rangle_{\Gamma_C}
 \end{align*}
For the third term we obtain by the definition of $\psi$ and $\psi^{hp}$ in \eqref{eq:Lame:defpsi} and by the continuity of the operators that
\begin{align*}
\left\langle Su-S_{hp} u^{hp},v\right\rangle_{\Gamma_\Sigma} &= \left\langle W(u-u^{hp}) + (K^\top+\frac{1}{2})(\psi-\psi^{hp}),v\right\rangle_{\Gamma_\Sigma} \\
& \leq C_W \left\|u-u^{hp}\right\|_{\tilde{H}^{1/2}(\Gamma_\Sigma)}\left\|v\right\|_{\tilde{H}^{1/2}(\Gamma_\Sigma)} \\&\qquad + (C_{K^\top}+\frac{1}{2})  \left\|\psi-\psi^{hp}\right\|_{\tilde{H}^{-1/2}(\Gamma_\Sigma)}\left\|v\right\|_{\tilde{H}^{1/2}(\Gamma_\Sigma)}\ .
\end{align*}
The first two and the last term can be handled as in Lemma~\ref{thm:aposterioriError_only_u}, leading to
\begin{align*}
\frac{1}{C} \left\langle \lambda-\lambda^{kq},v\right\rangle_{\Gamma_C} & \leq \left\|u-u^{hp}\right\|_{\tilde{H}^{1/2}(\Gamma_\Sigma)}\left\|v\right\|_{\tilde{H}^{1/2}(\Gamma_\Sigma)} + \left\|\psi-\psi^{hp}\right\|_{\tilde{H}^{-1/2}(\Gamma_\Sigma)}\left\|v\right\|_{\tilde{H}^{1/2}(\Gamma_\Sigma)}  \\ &\qquad + \left(\sum_{E\in\Gamma_{C}}\frac{h_{E}}{p_{E}^{2}} \left\|\lambda^{kq}+Su^{hp} \right\|^{2}_{L^2(E)}\right)^{\frac{1}{2}} \left\| v \right\|_{\tilde{H}^{\frac{1}{2}}(\Gamma_\Sigma)}\\
&\qquad + \sum_{E\in \mathcal{T}_h \cap \Gamma_N} \left(\frac{h_E}{p_E}\right)^{1/2} \left\|  t-S_{hp}u^{hp}\right\|_{0,E} \left\|v\right\|_{1/2,\omega(E)}\\
&\qquad +  \sum_{E\in \mathcal{T}_h \cap \Gamma_C} \left(\frac{h_E}{p_E}\right)^{1/2} \left\|   -\lambda^{kq}-S_{hp}u^{hp}\right\|_{0,E} \left\|v\right\|_{1/2,\omega(E)}
 \end{align*}
The assertion follows from the continuous inf-sup condition \eqref{eq:lame:infsup} and Cauchy-Schwarz inequality.
\end{proof}
\end{lemma}

Combining the two Lemmas~\ref{thm:aposterioriError_only_u} and \ref{lem:aposteriori_lambda} immediately yields the following theorem if in Lemma~\ref{thm:aposterioriError_only_u} $\epsilon>0$ is chosen sufficiently small.

\begin{theorem}[Residual based error estimate] \label{thm:aposteriori}
Let $(u,\lambda)$, $(u^{hp},\lambda^{kq})$ be the solution of \eqref{eq:MixedProblem}, \eqref{eq:DiscreteMixedProblem} respectively. Then there holds
\begin{align*}
&C\left(\left\|u-u^{hp}\right\|^2_{1/2,\Gamma_\Sigma} + \left\|\psi-\psi^{hp} \right\|^2_{-1/2,\Gamma_\Sigma} + \left\|\lambda^{kq} -\lambda \right\|^2_{-\frac{1}{2},\Gamma_C}  \right)\\
 &\leq
 \sum_{E\in \mathcal{T}_h \cap \Gamma_N} \left(\frac{h_E}{p_E}\right) \left\|  t-S_{hp}u^{hp}\right\|^2_{0,E} +  \sum_{E\in \mathcal{T}_h \cap \Gamma_C} \left(\frac{h_E}{p_E} + \frac{h_{E}}{p_{E}^{2}}  \right)\left\|   -\lambda^{kq}-S_{hp}u^{hp}\right\|^2_{0,E}
\\
&\qquad + \sum_{E \in \mathcal{T}_{h}}h_E \left\|\frac{\partial}{\partial s} \left(V\psi^{hp} -(K+\frac{1}{2})u^{hp}\right) \right\|^2_{L^2(E)}
+\left\langle \left(\lambda^{kq}_n\right)^+,\left(g-u^{hp}_n\right)^+\right\rangle_{\Gamma_C}+  \left\| \left(g-u^{hp}_n\right)^-\right\|^2_{\frac{1}{2},\Gamma_C} \\
&\qquad + \left\|\left(\lambda^{kq}_n\right)^-\right\|^2_{-\frac{1}{2},\Gamma_C} + \left\| \left(\left\|\lambda^{kq}_t\right\|_2-\mathcal{F} \right)^+ \right\|^2_{-\frac{1}{2},\Gamma_C}  -\left\langle \left(\left\|\lambda^{kq}_t\right\|_2-\mathcal{F} \right)^-, \left\|u^{hp}_t\right\|_2 \right\rangle_{\Gamma_C}
\\& \qquad-\left\langle  \lambda^{kq}_t,u^{hp}_t \right\rangle_{\Gamma_C}
 +\left\langle \left\|\lambda^{kq}_t\right\|_2 , \left\|u^{hp}_t\right\|_2 \right\rangle_{\Gamma_C} \ .
\end{align*}

\end{theorem}

It is worth pointing out, that the stabilization implies no additional term in the a posteriori error estimate compared to the non-stabilized case in \cite[Theorem~11]{Banz2014BEM} and does not even effect the scaling.

\begin{corollary}
For $\lambda^{kq} \in M^+_{kq}(\mathcal{F})$ or $\lambda^{kq} \in \tilde{M}^+_{kq}(\mathcal{F})$ with $q=1$, i.e.~$\lambda^{kq} \in M^+(\mathcal{F})$, the estimate of Theorem~\ref{thm:aposteriori} is reduced by  non-conformity terms and simplifies the complementarity and stick error contributions.
\begin{align*}
&C\left(\left\|u-u^{hp}\right\|^2_{1/2,\Gamma_\Sigma} + \left\|\psi-\psi^{hp} \right\|^2_{-1/2,\Gamma_\Sigma} + \left\|\lambda^{kq} -\lambda \right\|^2_{-\frac{1}{2},\Gamma_C}  \right)
\\ & \leq
 \sum_{E\in \mathcal{T}_h \cap \Gamma_N} \left(\frac{h_E}{p_E}\right) \left\|  t-S_{hp}u^{hp}\right\|^2_{0,E} +  \sum_{E\in \mathcal{T}_h \cap \Gamma_C} \left(\frac{h_E}{p_E} + \frac{h_{E}}{p_{E}^{2}}  \right)\left\|   -\lambda^{kq}-S_{hp}u^{hp}\right\|^2_{0,E}
\\
&\qquad + \sum_{E \in \mathcal{T}_{h}}h_E \left\|\frac{\partial}{\partial s} \left(V\psi^{hp} -(K+\frac{1}{2})u^{hp}\right) \right\|^2_{L^2(E)}\\ & \qquad
+\left\langle \lambda^{kq}_n,\left(g-u^{hp}_n\right)^+\right\rangle_{\Gamma_C}+  \left\| \left(g-u^{hp}_n\right)^-\right\|^2_{\frac{1}{2},\Gamma_C} -\left\langle \left\|\lambda^{kq}_t\right\|_2-\mathcal{F} , \left\|u^{hp}_t\right\|_2 \right\rangle_{\Gamma_C}  \\
& \qquad -\left\langle  \lambda^{kq}_t,u^{hp}_t \right\rangle_{\Gamma_C}
 +\left\langle \left\|\lambda^{kq}_t\right\|_2 , \left\|u^{hp}_t\right\|_2 \right\rangle_{\Gamma_C}
\end{align*}

\end{corollary}

\section{Implementational challenges}
\label{sec:Implementation}

For the contact stabilized BEM \eqref{eq:DiscreteMixedProblem} two non-standard matrices must be implemented, namely $\left\langle \gamma \lambda^{kq}, S_{hp} v^{hp}\right\rangle_{\Gamma_C}$ and $\left\langle \gamma S_{hp}u^{hp}, S_{hp} v^{hp}\right\rangle_{\Gamma_C}$. To restrain from additional difficulties we use the same mesh for $\lambda^{hp}$ and $u^{hp}$ on $\Gamma_C$. Hence, the singularities of $S_{hp}v^{hp}$ for the outer quadrature coincide with the nodes of the mesh for $\lambda^{kq}$ and the standard outer quadrature technique for the BE-potentials can be applied. In the implementation we utilize the representations $S_{hp}v^{hp} = Wv^{hp}+ (K+\frac{1}{2})^\top \eta^{hp}$ where $\eta^{hp}= V_{hp}^{-1}(K+\frac{1}{2})v^{hp} \in \mathcal{V}_{hp}^D$ and $Wv = -\frac{d}{ds} V^* \frac{dv}{ds}$ where $V^*$ is the modified single layer potential \cite[p.~157]{steinbach2003numerische}. Hence, performing integration by parts elementwise yields
\begin{align*}
\left\langle \gamma \lambda^{kq}, S_{hp} v^{hp}\right\rangle_{\Gamma_C}&=\left\langle \gamma \lambda^{kq}, -\frac{d}{ds} V^* \frac{d}{ds}v^{hp} \right\rangle_{\Gamma_C} + \left\langle \gamma \lambda^{kq}, (K+\frac{1}{2})^\top \eta^{hp}\right\rangle_{\Gamma_C} \\
&= \gamma_0 \sum_{E \in \mathcal{T}_h \cap \Gamma_C}  \frac{h_E}{p_E^2} \left\langle \frac{d}{ds} \lambda^{kq}, V^* \frac{d }{ds}v^{hp}\right\rangle_E - \left[\lambda^{kq} V^* \frac{d}{ds}v^{hp}\right]_{\partial E} + \left\langle \gamma \lambda^{kq}, (K+\frac{1}{2})^\top \eta^{hp}\right\rangle_{\Gamma_C}
\end{align*}
where $\partial E$ are the two endpoints of the interval $E$. Each of these terms can now be computed by standard BEM techniques, e.g.~decomposition into farfield, nearfield and selfelement with the corresponding (hp-composite) Gauss-Quadrature for the outer integration and the analytic computation for the inner integration \cite{maischak1999analytical}. The algebraic representation of $\eta^{hp}$ for the computation of the standard Steklov-operator $\left\langle S_{hp}u^{hp},v^{hp}\right\rangle_{\Gamma_\Sigma}$ is reused, i.e.
$$ \vec{\lambda}^\top \mathbf{\tilde{S}} \vec{v}= \vec{\lambda}^\top \mathbf{\tilde{W}} \vec{v} + \vec{\lambda}^\top \left(\mathbf{\tilde{K}} + \frac{1}{2}\mathbf{\tilde{I}}\right)^\top\mathbf{V}^{-1}\left( \mathbf{K} + \frac{1}{2}\mathbf{I} \right) \vec{v}.$$
For the second matrix we obtain
\begin{gather*}
\left\langle \gamma S_{hp}u^{hp}, S_{hp} v^{hp}\right\rangle_{\Gamma_C} = \left\langle \gamma Wu^{hp}+ \gamma(K+\frac{1}{2})^\top \zeta^{hp} ,Wv^{hp}+ (K+\frac{1}{2})^\top \eta^{hp} \right\rangle_{\Gamma_C}\\
= \left\langle \gamma Wu^{hp},Wv^{hp} \right\rangle_{\Gamma_C} + \left\langle \gamma Wu^{hp},(K+\frac{1}{2})^\top \eta^{hp}  \right\rangle_{\Gamma_C}\\ + \left\langle \gamma(K+\frac{1}{2})^\top \zeta^{hp} ,Wv^{hp} \right\rangle_{\Gamma_C} + \left\langle \gamma(K+\frac{1}{2})^\top \zeta^{hp} ,(K+\frac{1}{2})^\top \eta^{hp}  \right\rangle_{\Gamma_C}
\end{gather*}
Here, elementwise integration by parts of the hypersingular integral operator yields no advantages, except for $\left\langle \gamma Wu^{hp},\eta^{hp}  \right\rangle_{\Gamma_C}$ and $\left\langle \gamma  \zeta^{hp} ,Wv^{hp} \right\rangle_{\Gamma_C}$, and, therefore, the tangential derivative is approximated by a central finite difference quotient with a step length of $10^{-4}$ on the reference interval. This yields the matrix representation
\begin{align*}
\vec{v}^\top \mathbf{\hat{S}} \vec{u}=& \:\vec{v}^\top \mathbf{\widehat{WW}} \vec{u}
+\vec{v}^\top \left( \mathbf{K} + \frac{1}{2}\mathbf{I} \right)^\top  \mathbf{V}^{-\top}\left(\mathbf{\widehat{WK^\top}} + \frac{1}{2} \mathbf{\widehat{WI} }\right) \vec{u}
+\vec{v}^\top \left(\mathbf{\widehat{WK^\top}} + \frac{1}{2} \mathbf{\widehat{WI} }\right)^\top \mathbf{V}^{-1}\left( \mathbf{K} + \frac{1}{2}\mathbf{I} \right) \vec{u}\\
&+\vec{v}^\top\left( \mathbf{K} + \frac{1}{2}\mathbf{I} \right)^\top  \mathbf{V}^{-\top} \left( \mathbf{\widehat{K^\top K^\top}} + \frac{1}{2}\mathbf{\widehat{K^\top I}} + \frac{1}{2}\mathbf{\widehat{K^\top I}}^\top + \frac{1}{4}\mathbf{\widehat{II}}  \right) \mathbf{V}^{-1}\left( \mathbf{K} + \frac{1}{2}\mathbf{I} \right)\vec{u}
\end{align*}
Most of the computational time is required for the matrices $\mathbf{\widehat{WW}}$, $\mathbf{\widehat{K^\top K^\top}}$ and $\mathbf{\widehat{WK^\top}}$. Hence, their symmetry and other optimization possibilities should be exploited thoroughly.

%
and

 \section{Modifications for Coulomb friction}
 \label{sec:CoulombFriction}

 Tresca friction often yields unphysical behavior, namely non-zero tangential traction and stick-slip transition outside the actual contact zone. Therefore, in many applications the more realistic Coulomb friction is applied, in which the friction threshold $\mathcal{F}$ is replaced by $\mathcal{F}|\sigma_n(u)|$. In the here presented discretization only the Lagrangian multiplier set must be adapted, namely
 \begin{align}
 M^+(\mathcal{F}\lambda_n)&:=\left\lbrace \mu \in \tilde{H}^{-1/2}(\Gamma_C): \left\langle \mu,v\right\rangle_{\Gamma_C} \leq \left\langle \mathcal{F}\lambda_n,\left\|v_t\right\| \right\rangle_{\Gamma_C} \forall v \in \tilde{H}^{1/2}(\Gamma_\Sigma), v_n \leq 0 \right\rbrace \label{eq:CoulombModLagrangeSet} \\
M_{kq}^+(\mathcal{F}\lambda^{kq}_n)&:= \left\{ \mu^{kq} \in L^2(\Gamma_C): \mu^{kq}|_E = \sum_{i=0}^{q_E} \mu_i^E B_{i,q_E}^{E} \in \left[\mathbb{P}_{q_E}(E)\right]^2,\ (\mu^E_i)_n \geq 0, \ \left|(\mu_i^E)_t\right| \leq \mathcal{F}(\Psi_E(iq_E^{-1}))(\lambda_i^E)_n \right\} \\
\tilde{M}_{kq}^+(\mathcal{F}\lambda^{kq}_n)&:= \left\{ \mu^{kq} \in L^2(\Gamma_C): \mu^{kq}|_E \in \left[\mathbb{P}_{q_E}(E)\right]^2,\ \mu^{kq}_n \geq 0, \ -\mathcal{F}\lambda^{kq}_n \leq \mu^{kq}_t \leq \mathcal{F}\lambda^{kq}_n \text{ in } G_{kq} \right\}
 \end{align}
 A standard iterative solver technique for Coulomb friction is to solve a sequence of Tresca frictional problems in which the friction threshold $\mathcal{F}\lambda_n$ of the current Tresca subproblem is obtained from the previous iterative Tresca solution. Since that solution is updated in the next Tresca iteration anyway we solve the subproblem inexactly by a single semi-smooth Newton step.\\

 \begin{theorem} \label{thm:errorEst_ContactSplit_Coulomb}
 Let $(u,\lambda)$, $(u^{hp},\lambda^{kq})$ be the solution of \eqref{eq:MixedProblem}, \eqref{eq:DiscreteMixedProblem} respectively,  with the Lagrange multiplier sets modified according to Coulomb friction. Under the assumption that $\lambda_t = \mathcal{F} \lambda_n \xi$,  $\xi \in Dir_t(u_t)$ where $Dir_t(u_t)$ is the subdifferential of the convex map $u_t \mapsto \left|u_t\right|$ (see \cite{hild2009residual}), $\mathcal{F} \geq 0$ constant and $\mathcal{F} \left\|\xi\right\|$ sufficiently small, there holds
 \begin{align*}
 &C\left(\left\|u-u^{hp}\right\|^2_{1/2,\Gamma_\Sigma} + \left\|\psi-\psi^{hp} \right\|^2_{-1/2,\Gamma_\Sigma} + \left\|\lambda^{kq} -\lambda \right\|^2_{-\frac{1}{2},\Gamma_C}  \right)\\
&  \leq
  \sum_{E\in \mathcal{T}_h \cap \Gamma_N} \left(\frac{h_E}{p_E}\right) \left\|  t-S_{hp}u^{hp}\right\|^2_{0,E} +  \sum_{E\in \mathcal{T}_h \cap \Gamma_C} \left(\frac{h_E}{p_E} + \frac{h_{E}^{2\alpha-1}}{p_{E}^{2\beta-2}}  \right)\left\|   -\lambda^{kq}-S_{hp}u^{hp}\right\|^2_{0,E}
 \\
 &+ \sum_{E \in \mathcal{T}_{h}}h_E \left\|\frac{\partial}{\partial s} \left(V\psi^{hp} -(K+\frac{1}{2})u^{hp}\right) \right\|^2_{L^2(E)}
 + \left\| (\lambda_{hp})_n^- \right\|^2_{-1/2,\Gamma_C}
 +\left\| \left( |(\lambda_{hp})_t|-\mathcal{F}(\lambda_{hp})_n^+ \right)^+ \right\|_{-1/2,\Gamma_C}  \\
 & - \left\langle \left( |(\lambda_{hp})_t|-\mathcal{F}(\lambda_{hp})_n^+ \right)^-  ,  \left|(u_{hp})_t\right|\right\rangle_{\Gamma_C}  + \left\langle \left|(\lambda_{hp})_t\right|, \left|(u_{hp})_t\right| \right\rangle_{\Gamma_C}  - \left\langle (\lambda_{hp})_t, (u_{hp})_t \right\rangle_{\Gamma_C}  \\
 &\left.+\left\| (g-(u_{hp})_n)^- \right\|_{1/2,\Gamma_C}
 +\left\langle \left(\lambda_{hp}\right)^+_n,\left(g-\left(u_{hp}\right)_n\right)^+\right\rangle_{\Gamma_C} \right)
 \end{align*}
 \end{theorem}
 \begin{proof}
 The same arguments as for Theorem~\ref{thm:aposteriori} apply, only the estimate of the tangential component in Lemma~\ref{lem:aposterioriContactPart} changes, \cite{hild2009residual,Banz2014BEM}. From $\lambda_t = \mathcal{F} \lambda_n \xi$ follows
 \begin{align*}
 \left\langle \lambda^t-\lambda^{hp}_t,u^{hp}_t-u_t\right\rangle_{\Gamma_C} &=
 \left\langle \lambda^{hp}_t - \mathcal{F} \xi  \left(\lambda_{hp}\right)_n,u_t - \left(u_{hp}\right)_t\right\rangle_{\Gamma_C}
 + \left\langle \mathcal{F} \xi \left(\left(\lambda_{hp}\right)_n - \lambda_n \right),u_t-\left(u_{hp}\right)_t\right\rangle_{\Gamma_C} \\
 &\leq \left\langle \left(\lambda_{hp}\right)_t - \mathcal{F} \xi  \left(\lambda_{hp}\right)_n,u_t - \left(u_{hp}\right)_t\right\rangle_{\Gamma_C} + \mathcal{F} \left\|\xi\right\| \left\|u-u_{hp}\right\|_{1/2,\Gamma_\Sigma}\left\|\lambda-\lambda_{hp}\right\|_{-1/2,\Gamma_C}
 \end{align*}
 For the other term there holds, similarly to \cite[Eq.~27]{hild2009residual},
 \begin{align*}
& \left\langle \left(\lambda^{hp}\right)_t - \mathcal{F} \xi  \left(\lambda_{hp}\right)_n,u_t - \left(u^{hp}\right)_t\right\rangle_{\Gamma_C} \\&\leq  \left\|u-u_{hp}\right\|_{1/2,\Gamma_\Sigma} \left\{ \left\| \left( |(\lambda_{hp})_t|-\mathcal{F}(\lambda^{hp})_n^+ \right)^+ \right\|_{-1/2,\Gamma_C} + \mathcal{F} \left\| (\lambda^{hp})_n^-\right\|_{-1/2,\Gamma_C}   \right\}\\
 &- \left\langle \left( |(\lambda^{hp})_t|-\mathcal{F}(\lambda_{hp})_n^+ \right)^-  ,  \left|(u_{hp})_t\right|\right\rangle_{\Gamma_C} + \left\langle \left|(\lambda_{hp})_t\right|, \left|(u_{hp})_t\right| \right\rangle_{\Gamma_C} - \left\langle (\lambda_{hp})_t, (u_{hp})_t \right\rangle_{\Gamma_C}.
 \end{align*}
 This gives the assertion if $\alpha-\epsilon_1-\epsilon_2-\frac{C}{\beta}\mathcal{F}\left\|\xi\right\|(1+\epsilon_3)>0$, i.e.~$\mathcal{F}\left\|\xi\right\|$ is sufficiently small.

 \end{proof}

 \section{Numerical experiments}
 \label{sec:NumericalExperiments}

The following numerical experiments are carried out on an Intel Xeon compute-server with the software package of Lothar Banz. We choose $\gamma_0=10^{-3}$, and use the adaptivity algorithm as in \cite[Alg.~12]{Banz2014BEM} with D\"orfler marking parameter $\theta=0.3$ and analyticity estimation parameter $\delta=0.5$. The discrete problems are solved with a semi-smooth Newton method for which the constraint \eqref{eq:DiscreteContactConstraints} is written as two projection equations, one in the normal and one in the tangential component.

 \subsection{Mixed boundary value problem with linear Tresca-friction threshold}
 For the following numerical experiments, the domain is $\Omega=[-\frac{1}{2},\frac{1}{2}]^2$ with $\Gamma_C=[-\frac{1}{2},\frac{1}{2}]\times \left\{-\frac{1}{2}\right\}$, $\Gamma_D=[\frac{1}{4},\frac{1}{2}]\times \left\{\frac{1}{2}\right\}$ and $\Gamma_N=\partial \Omega \setminus \left(\Gamma_C \cup \Gamma_D\right)$. The material parameters are $E=500$ and $\nu=0.3$, the gap function is $g=1-\sqrt{1-\frac{x_1^2}{100}}$ and the Tresca friction function is $\mathcal{F}=0.211+0.412x_1$. The Neumann force is
 \begin{alignat*}{2}
 t_{\text{left}}&=  \left(
 \begin{array}{c}
 	 -(\frac{1}{2}-x_2)(-\frac{1}{2}-x_2)\\
 	0
 \end{array}\right) , &\qquad \text{on } &\left\{-\frac{1}{2}\right\}\times \left[-\frac{1}{2},\frac{1}{2}\right] \\
 t_{\text{top}}&=  \left(
 \begin{array}{c}
 	0\\
 	20(-\frac{1}{2}-x_1)(-\frac{1}{4}-x_1)
 \end{array}\right) , &\qquad \text{on } &\left[-\frac{1}{2},-\frac{1}{4}\right]\times \left\{\frac{1}{2}\right\}
 \end{alignat*}
 and zero elsewhere.  An example with similar obstacle and friction function is considered in \cite{schroder2012posteriori} for FEM and in \cite{Banz2014BEM} for BEM with biorthogonal basis functions. The solution is characterized by two singular points at the interface from Neumann to Dirichlet boundary condition. These singularities are more sever than the loss of regularity from the contact conditions. At the contact boundary the solution has a long interval in which it is sliding and in which the absolute value of the tangential Lagrangian multiplier increases linearly like $\mathcal{F}$. The actual contact set is slightly to the right of the middle of $\Gamma_C$.

 Figure~\ref{fig:Tresca:errorDiri} shows the reduction of the error indicator for different families of discrete solutions. The residual error indicator for the uniform $h$-version with $p=1$ has a convergence rate of $0.25$ which is the same as in the non-stabilized case with biorthogonal basis function presented in \cite{Banz2014BEM}. Here, the residual contribution of the residual error indicator is divided by ten to compensate for the residual estimator's typical large reliability constant. This factor is purely heuristic and based on a comparison of the residual and bubble error estimator for contact problems with biorthogonal basis functions in \cite{Banz2014BEM}. Employing an $h$-adaptive scheme improves the convergence to $2.5$, which later on deteriorates, by local refinements at the Neumann to Dirichlet interface and by locally resolving the contact interface as well as the first kink in $\lambda_t$. If both, $h$- and $p$-refinements are carried out, the convergence rate is only slightly further
improved to $2.6$. This is a very different behavior to the non-stabilized case with biorthogonal basis functions. There $h$-adaptivity has a convergence rate of $1.3$ and $hp$-adaptivity of $2.4$ and a significant fraction of the adaptive refinements is carried out on the contact boundary $\Gamma_C$. In fact, the $h$-adaptive scheme there shows an almost uniform mesh refinement on a large part of $\Gamma_C$ which is not the case here, Figure~\ref{fig:Tresca:AdaptiveMeshes} (a). The reason for that might be that the residual of the variational equation is the dominant contribution of the error indicator, Figure~\ref{fig:Tresca:errorContributions}. On the contact boundary, this is $\lambda^{kq}+S_{hp}u^{hp}$. However, the employed stabilization tries to achieve that $\lambda^{kq}+S_{hp}u^{hp}=0$ for each discrete solution. Hence, the estimated error on $\Gamma_C$ is correctly small and no local refinements are needed there.\\
Noting that the Bernstein basis functions (and constraints) are the same as for Gauss-Lobatto-Lagrange (GLL) if $p=1$, it is clear that the error estimation does not differ between these two approach for both the uniform and the $h$-adaptive scheme, Figure~\ref{fig:Tresca:errorDiri}. When looking at the $hp$-refined meshes for these two approaches, Figure~\ref{fig:Tresca:AdaptiveMeshes} (b)-(c), it becomes clear why the difference in the error estimates is that small. Nevertheless, in the GLL approach the consistency error in $\lambda_n$ is non zero, Figure~\ref{fig:Tresca:errorContributions}(c), contrary to the conforming Bernstein polynomial case, Figure~\ref{fig:Tresca:errorContributions}(b). In both cases the consistency error in $\lambda_t$ is on machine precision and is therefore omitted in the plots.

 \begin{figure}[tbp]
   \centering
   \includegraphics[trim = 22mm 4mm 25mm 10mm, clip,width=100.0mm, keepaspectratio]{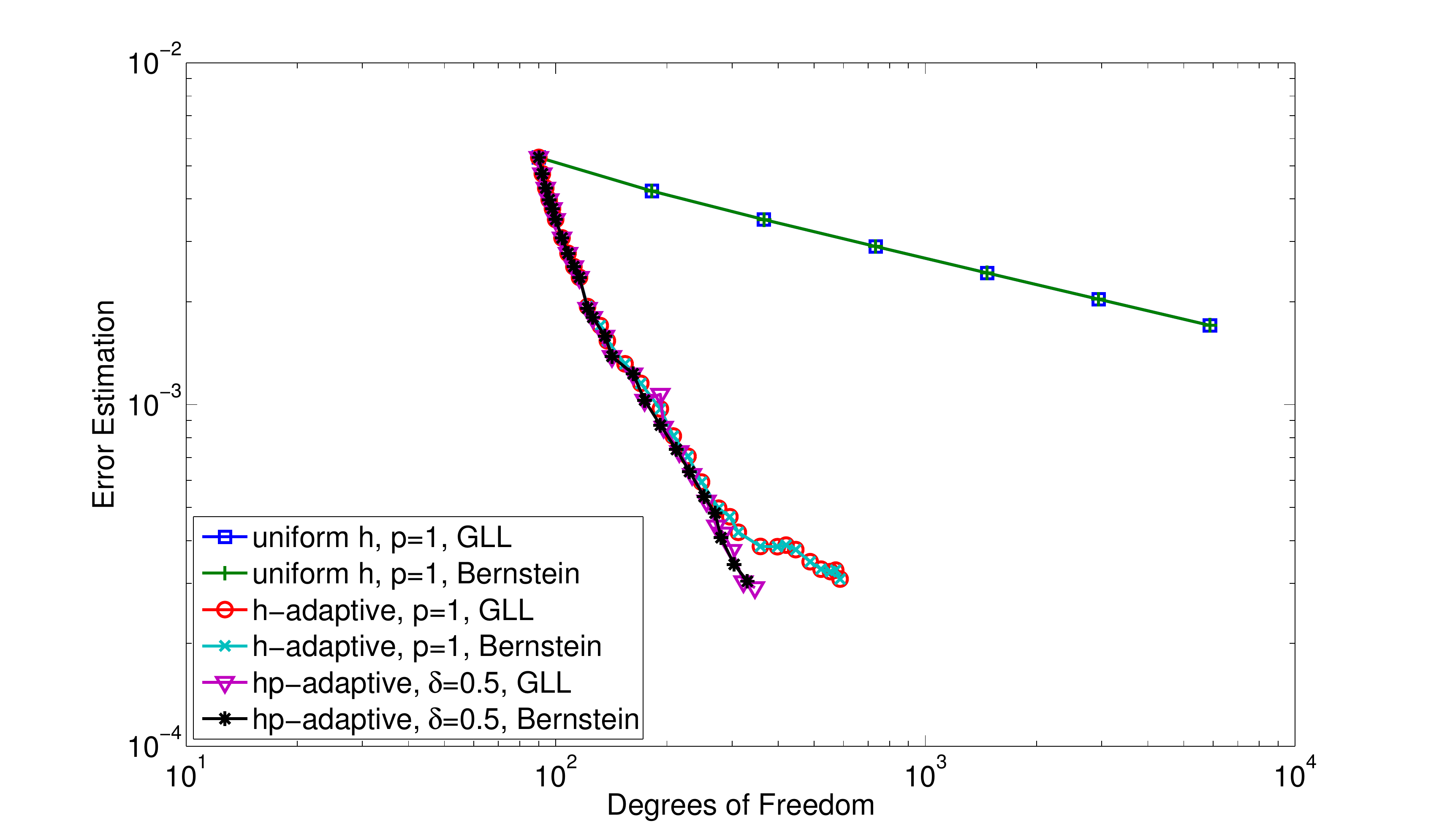}
   \caption{Error estimations for different families of discrete solutions (Tresca-friction)}
   \label{fig:Tresca:errorDiri}
 \end{figure}

 \begin{figure}
   \centering   \mbox{\subfigure[$h$-adaptive, mesh nr.~16 (inner), nr.~24 (outer)]{
   \begin{overpic}[trim = 41mm 31mm 35mm 7mm, clip,width=60.0mm, keepaspectratio]{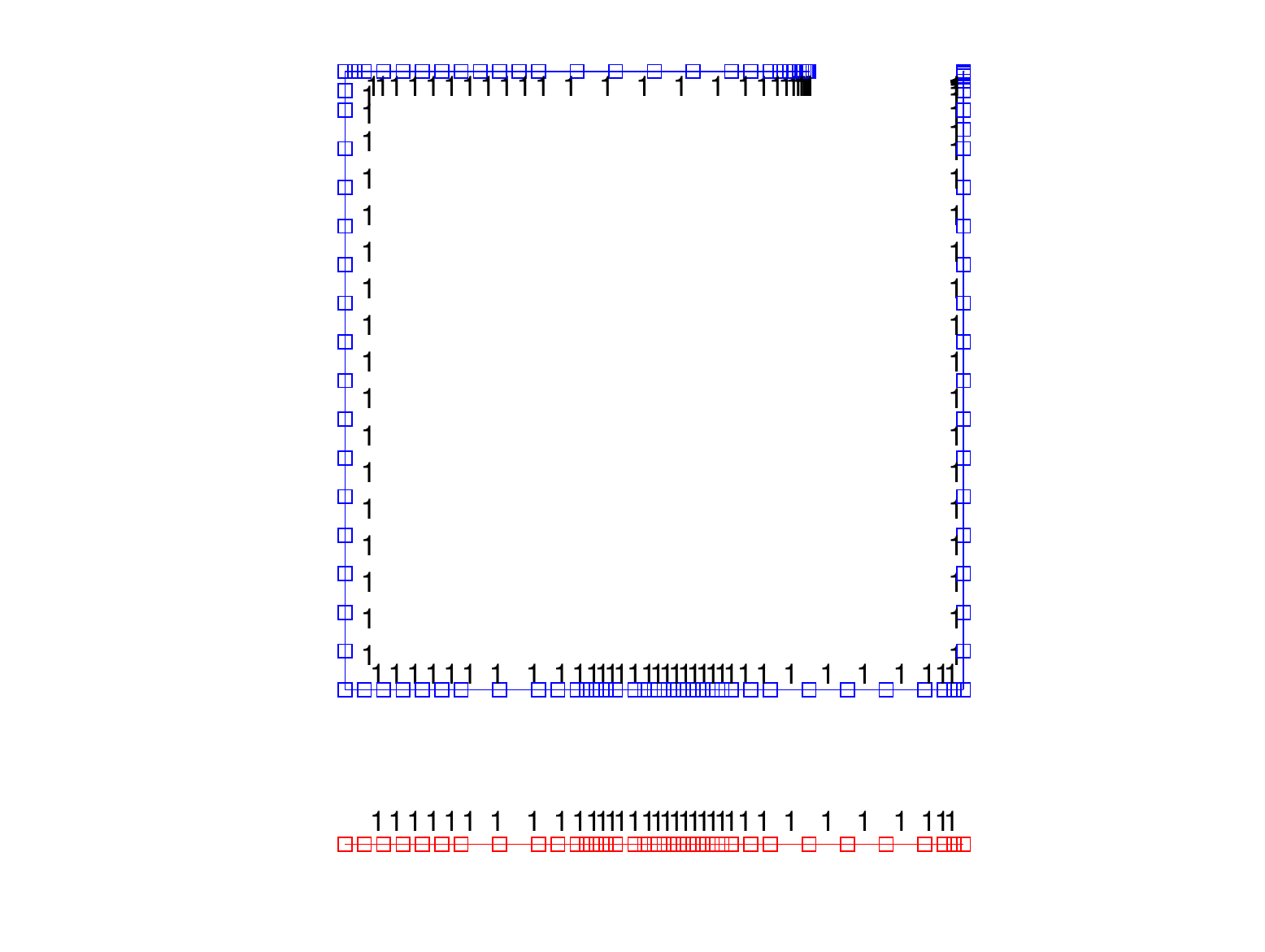}
     \put(7.1,8.4){ \includegraphics[trim = 41mm 31mm 35mm 7mm, clip, width=50.0mm, keepaspectratio]{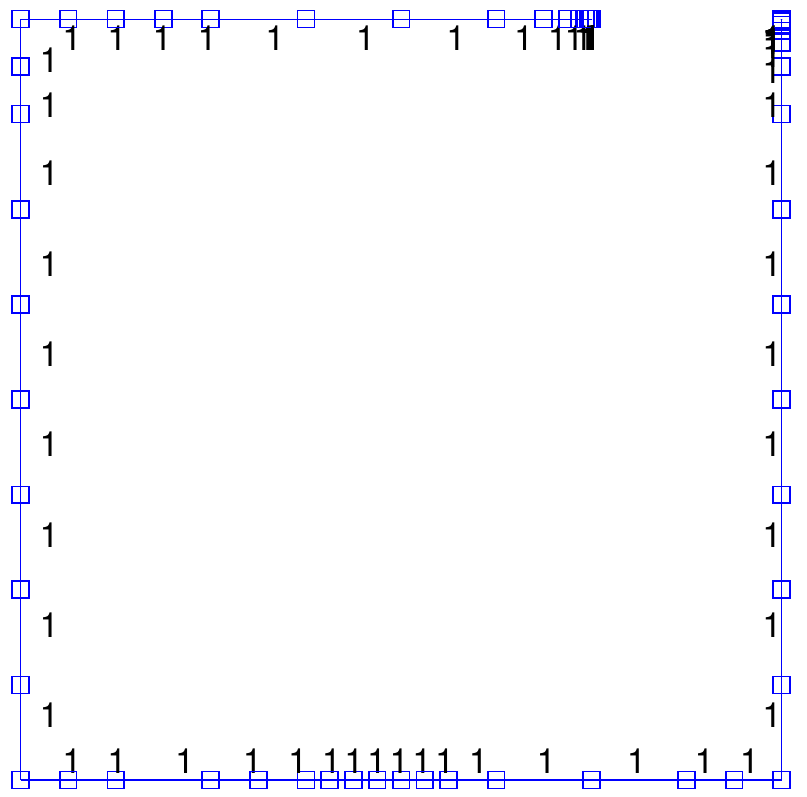}}
   \end{overpic} } \hspace{+1.5cm}
 	\subfigure[$hp$-adap.~(Bernstein), mesh nr.~16 (in), nr.~24 (out)]{
   \begin{overpic}[trim = 41mm 31mm 35mm 7mm, clip,width=60.0mm, keepaspectratio]{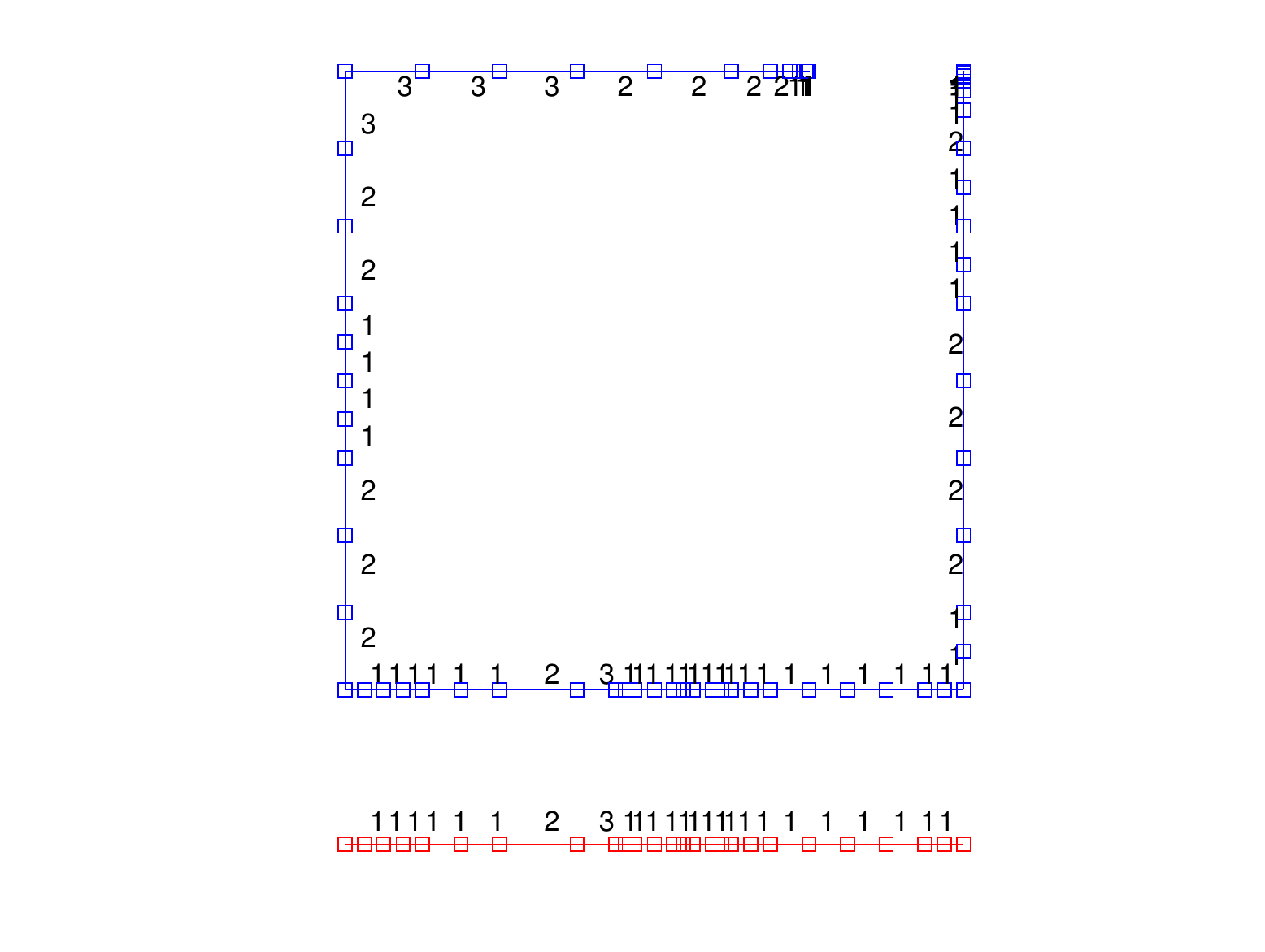}
     \put(7.1,8.4){ \includegraphics[trim = 41mm 31mm 35mm 7mm, clip, width=50.0mm, keepaspectratio]{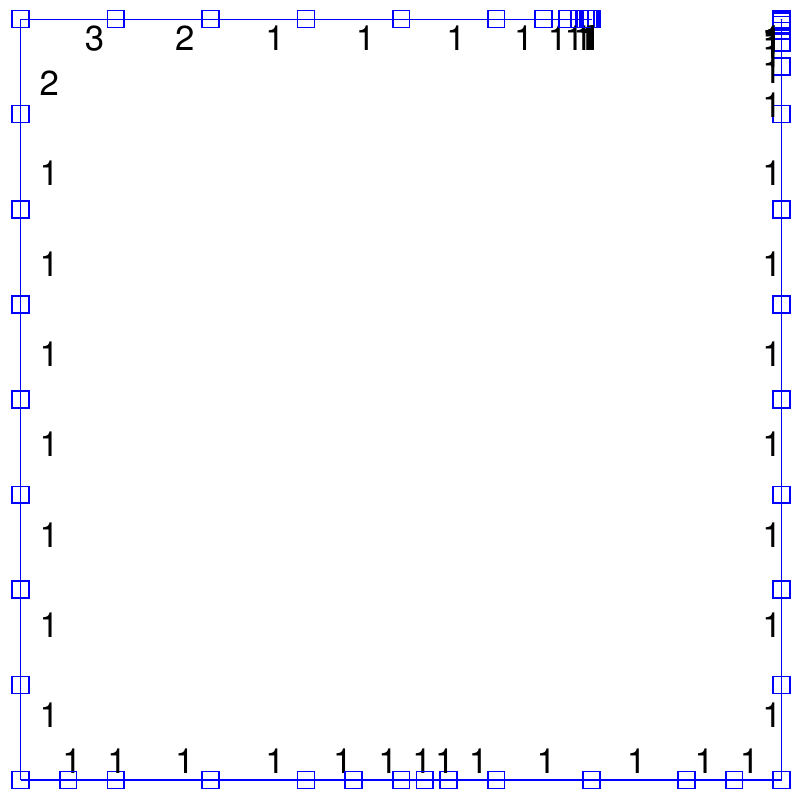}}
   \end{overpic}	} }
	
 	\subfigure[$hp$-adap.~(GLL), mesh nr.~16 (inner), nr.~24 (outer)]{
   \begin{overpic}[trim = 41mm 31mm 35mm 7mm, clip,width=60.0mm, keepaspectratio]{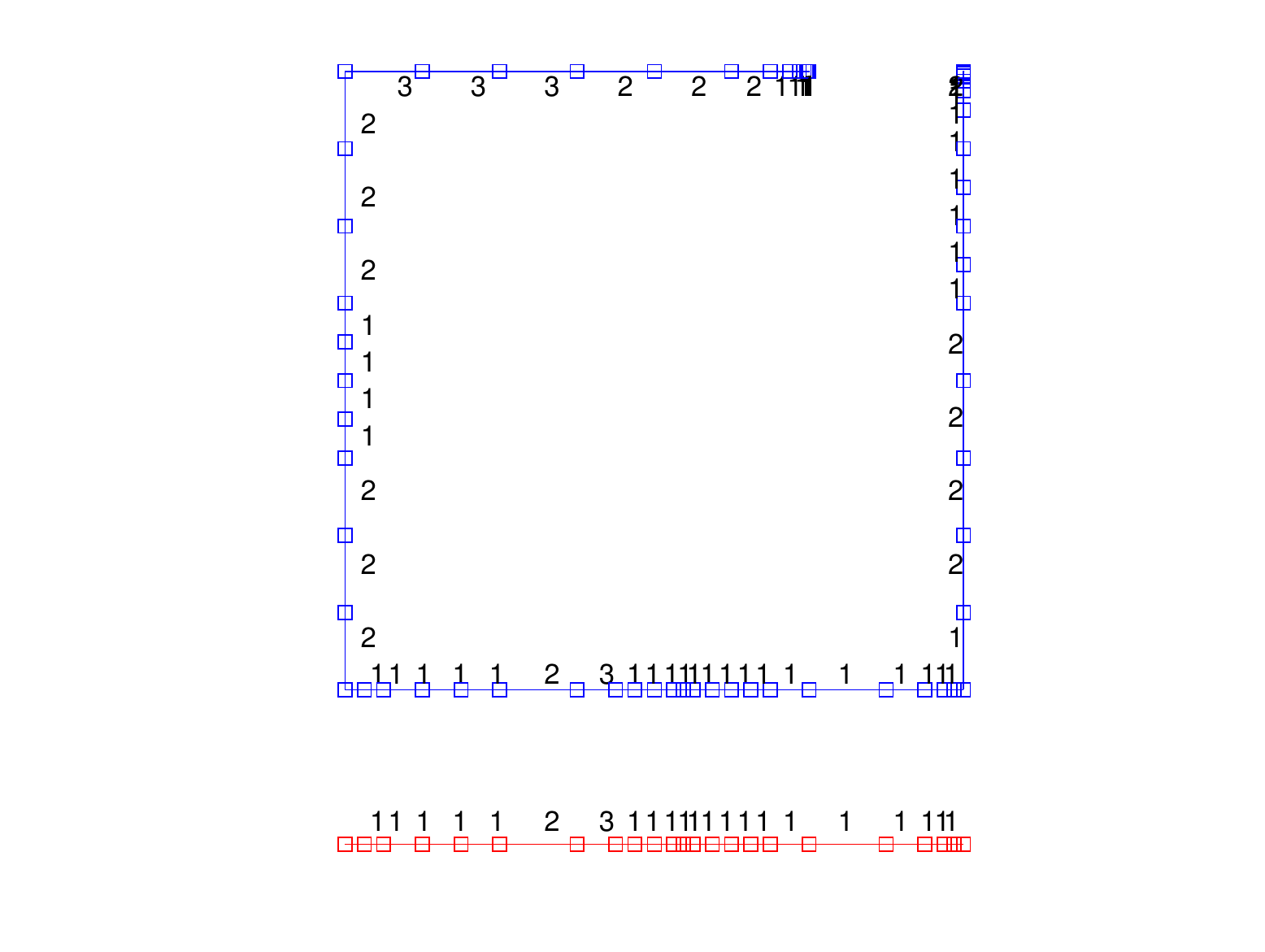}
     \put(6.9,8.7){ \includegraphics[trim = 41mm 31mm 35mm 7mm, clip, width=50.0mm, keepaspectratio]{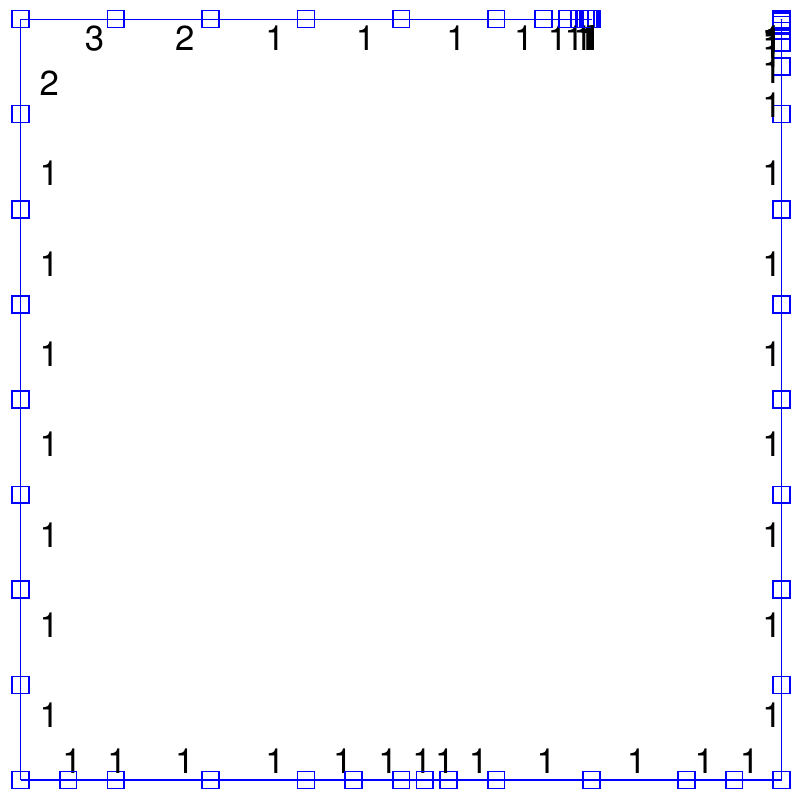}}
   \end{overpic}	}
 \caption{Adaptively generated meshes (Tresca-friction)}
 \label{fig:Tresca:AdaptiveMeshes}
 \end{figure}

 \begin{figure}[tbp]
   \centering \mbox{
   \subfigure[uniform $h$-version with $p=1$ (Bernstein/GLL)]{
 	\includegraphics[trim = 30mm 8mm 25mm 10mm, clip,width=70.0mm, keepaspectratio]{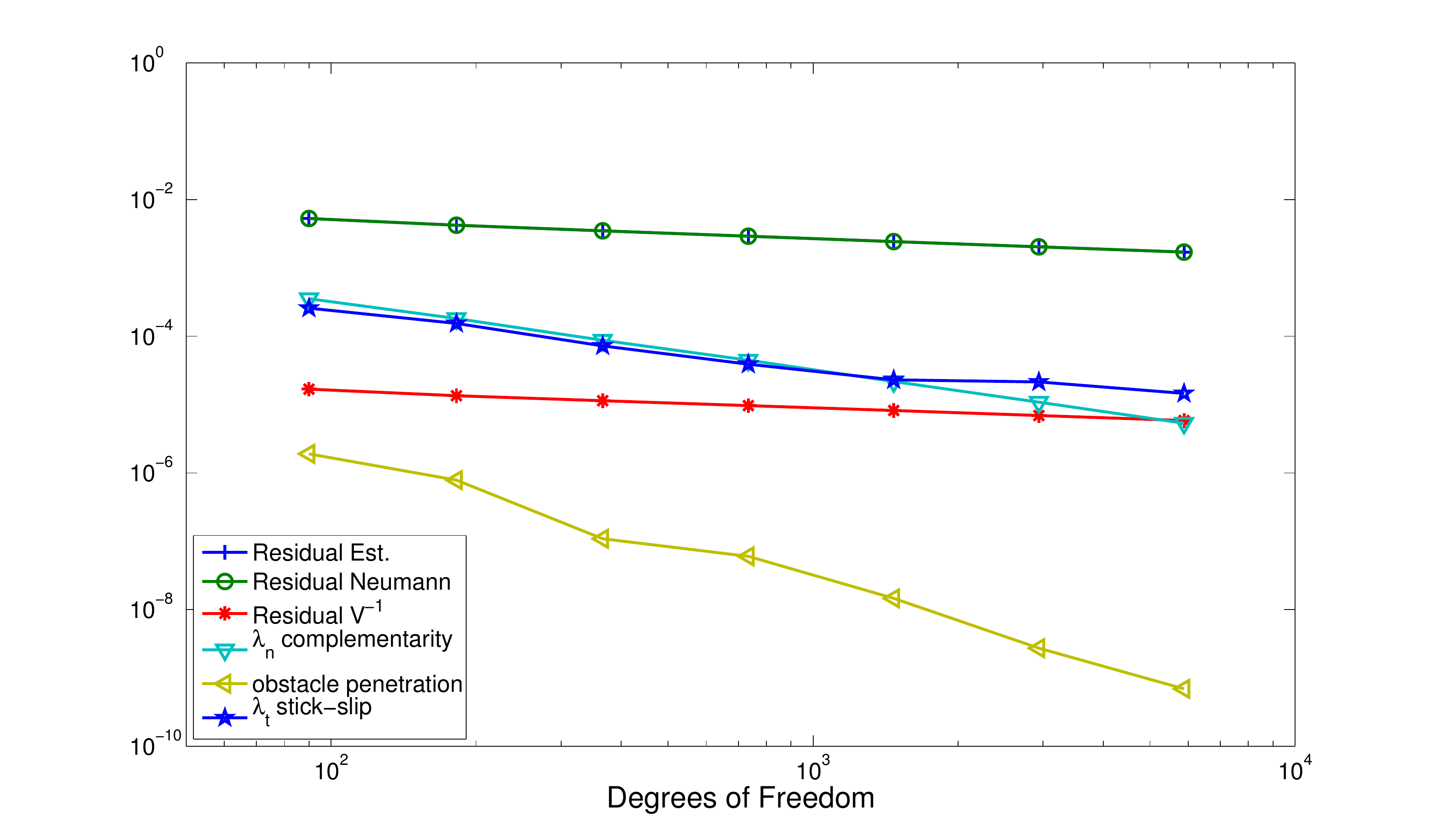}}
   \subfigure[$hp$-adaptive (Bernstein)]{
 	\includegraphics[trim = 30mm 8mm 25mm 10mm, clip,width=70.0mm, keepaspectratio]{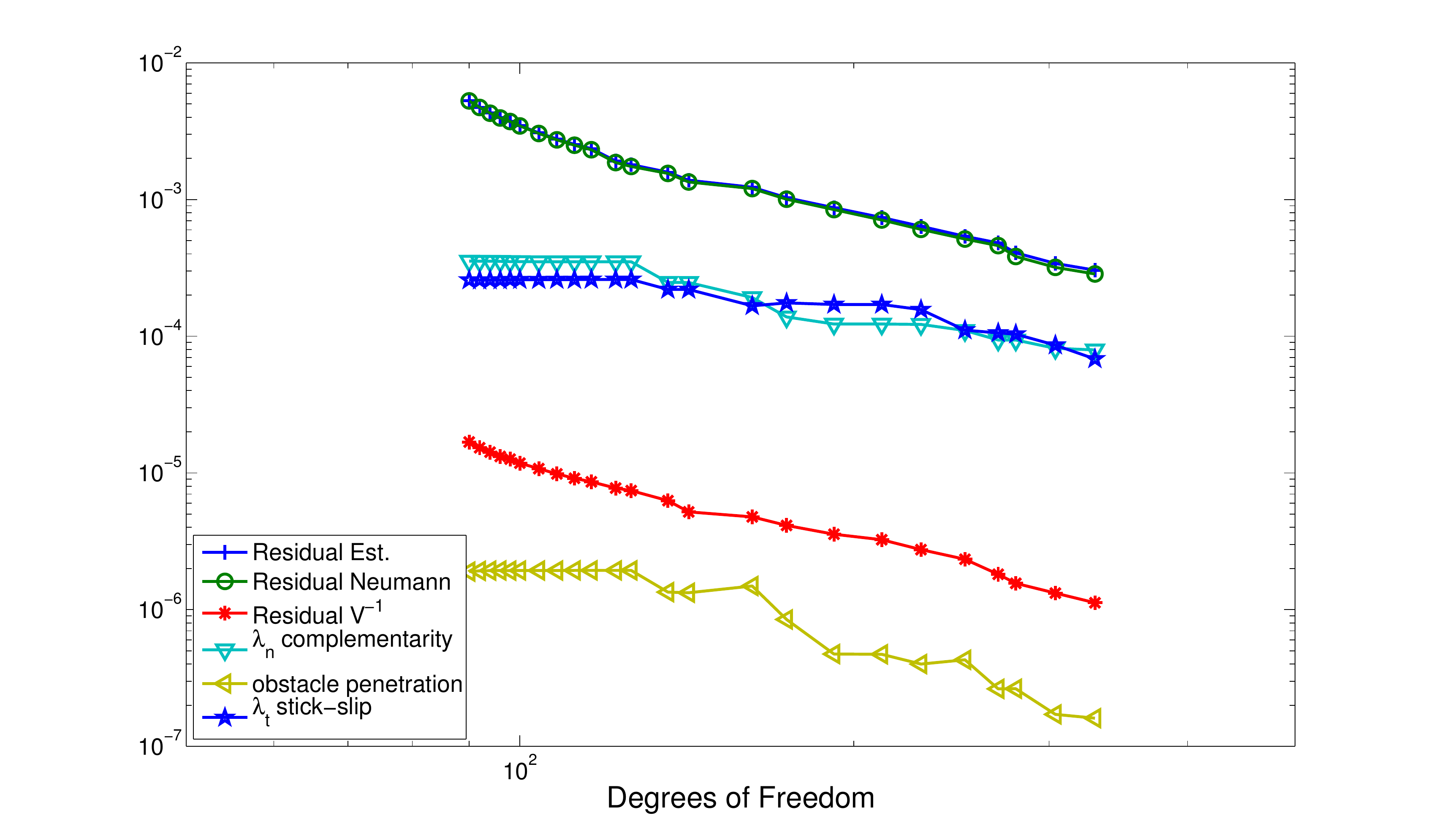}} }\\
	   \subfigure[$hp$-adaptive (GLL)]{
 	\includegraphics[trim = 30mm 8mm 25mm 10mm, clip,width=70.0mm, keepaspectratio]{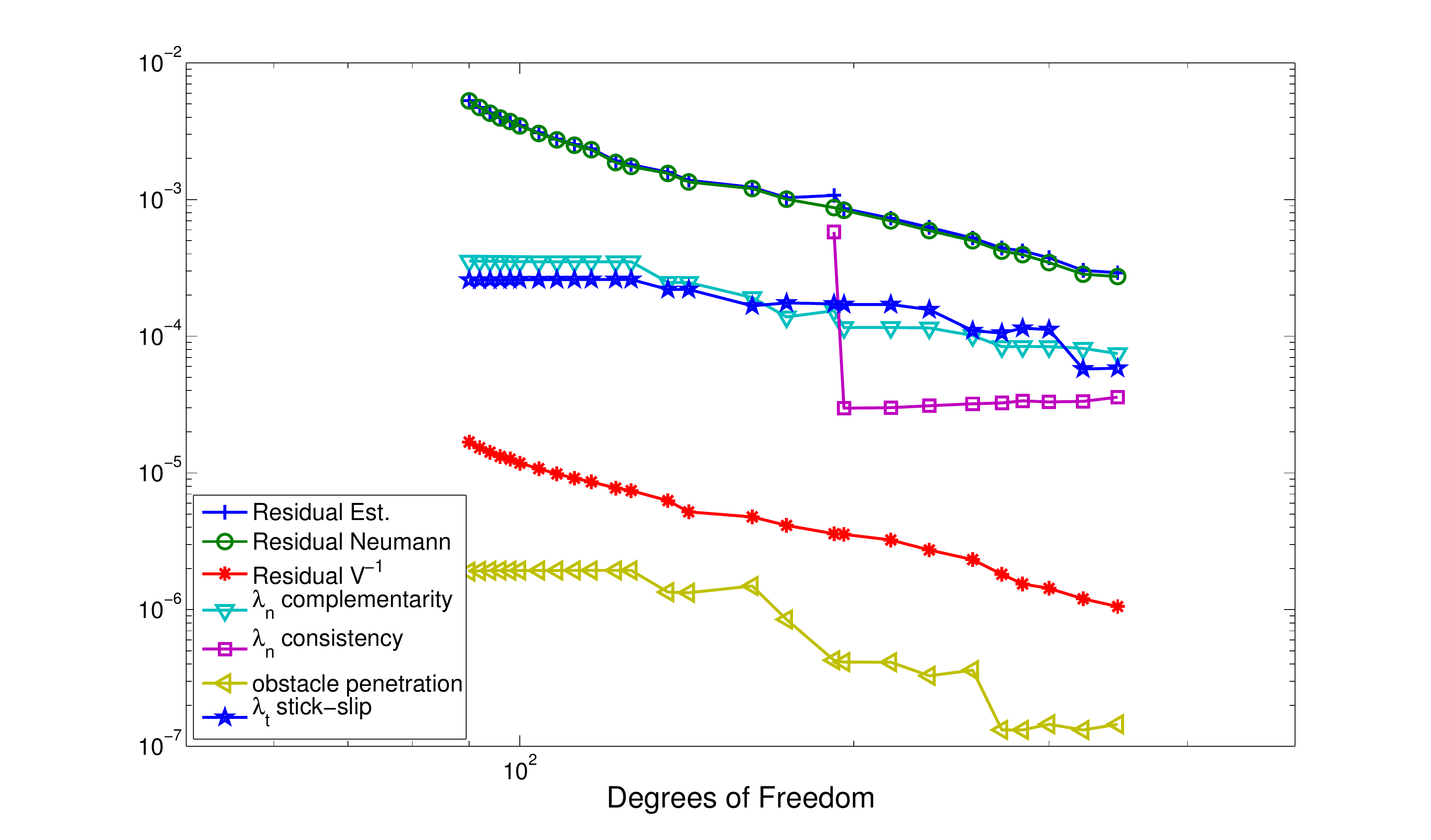}}
	
   \caption{Error contributions of the residual based error indicator (Tresca-friction) }
   \label{fig:Tresca:errorContributions}
 \end{figure}

 \subsection{Neumann boundary value problem with Coulomb-friction}
 For the following numerical experiments, the domain is $\Omega=[-\frac{1}{2},\frac{1}{2}]^2$ with $\Gamma_C=[-\frac{1}{2},\frac{1}{2}]\times \left\{-\frac{1}{2}\right\}$ and $\Gamma_N=\partial \Omega \setminus \Gamma_C$. Since no Dirichlet boundary has been prescribed, the kernel of the Steklov operator consists of the three rigid body motions $\operatorname{ker}(S)=\operatorname{span}\left\{ (x_1,0)^\top, (0,x_2)^\top, (x_2,-x_1)^\top \right\}$. Nevertheless, to obtain a unique solution the rigid body motions are forced to zero during the simulation. The material parameters are $E=5$ and $\nu=0.45$, and the Coulomb friction coefficient is $0.3$. The Neumann force is
 \begin{align*}
 t_{\text{side}}&=  \left(
 \begin{array}{c}
 	-10\operatorname{sign}(x_1)(\frac{1}{2}+x_2)(\frac{1}{2}-x_2)exp(-10(x_2+\frac{4}{10})^2)\\
 	\frac{7}{8}(\frac{1}{2}+x_2)(\frac{1}{2}-x_2)
 \end{array}\right) \\
 t_{\text{top}}&=  \left(
 \begin{array}{c}
 	0\\
 	-\frac{25}{2}(\frac{1}{2}-x_1)^2(\frac{1}{2}+x_1)^2
 \end{array}\right)
 \end{align*}
  on the side, top respectively, and the gap to the obstacle is zero. A similar example is considered in \cite{hueber2012equilibration} for FEM and the same in \cite{Banz2014BEM} for BEM with biorthogonal basis functions. The solution is characterized by a large contact set and that the Lagrangian multiplier has a kink, jump in the normal, tangential component, respectively, at $x_1=0$, Figure~\ref{fig:Coulomb:solution}.\\

Figure~\ref{fig:Coulomb:errorDiri} shows the reduction of the error indicator for different families of discrete solutions. The residual error indicator for the uniform $h$-version with $p=1$ has a convergence rate of almost $1.5$. Employing an $h$-adaptive scheme improves the convergence in the preasymptotic range but then the estimated error runs parallel to the uniform case as only quasi uniform mesh refinements are carried out, Figure~\ref{fig:Coulomb:AdaptiveMeshes} (a). If both, $h$- and $p$-refinements are carried out, the convergence rate is improved to over $2.8$ after a preasymptotic range in which only $h$-refinements are been carried out. The estimated error for the GLL- and Bernstein approach is the same even for the $hp$-adaptive case, since the basis functions for the Lagrange multiplier and the contact conditions only differ where $p\geq 2$. This however, is only the case outside the actual contact area, Figure~\ref{fig:Coulomb:AdaptiveMeshes} (b)-(c), but there $\lambda=0$ due to Coulomb's
friction law.
The error reduction and adaptivity behavior is again very different to the non-stabilized case with biorthogonal basis function \cite[Sec.~6.2]{Banz2014BEM}. There the convergence rate is larger with $1.9$ for $h$-adaptivity and $3.3$ for $hp$-adaptivity and the refinements on $\Gamma_C$ are more localized. There, the dominant error source is the slip-stick contribution, and thus explaining the local mesh refinements on $\Gamma_C$, whereas here the residual of the variational equation and the violation of the complementarity condition in $\lambda_n$ are dominant. Interestingly, here, the slip-stick contribution is the smallest non-zero error contribution and is several decades smaller then the  other remaining ones, Figure~\ref{fig:Coulomb:errorContributions}.

  \begin{figure}[tbp]
    \centering \mbox{
    \subfigure[Reference (gray), deformed (blue)]{
  	\includegraphics[trim = 10mm 7mm 10mm 8mm, clip,width=65.0mm, keepaspectratio]{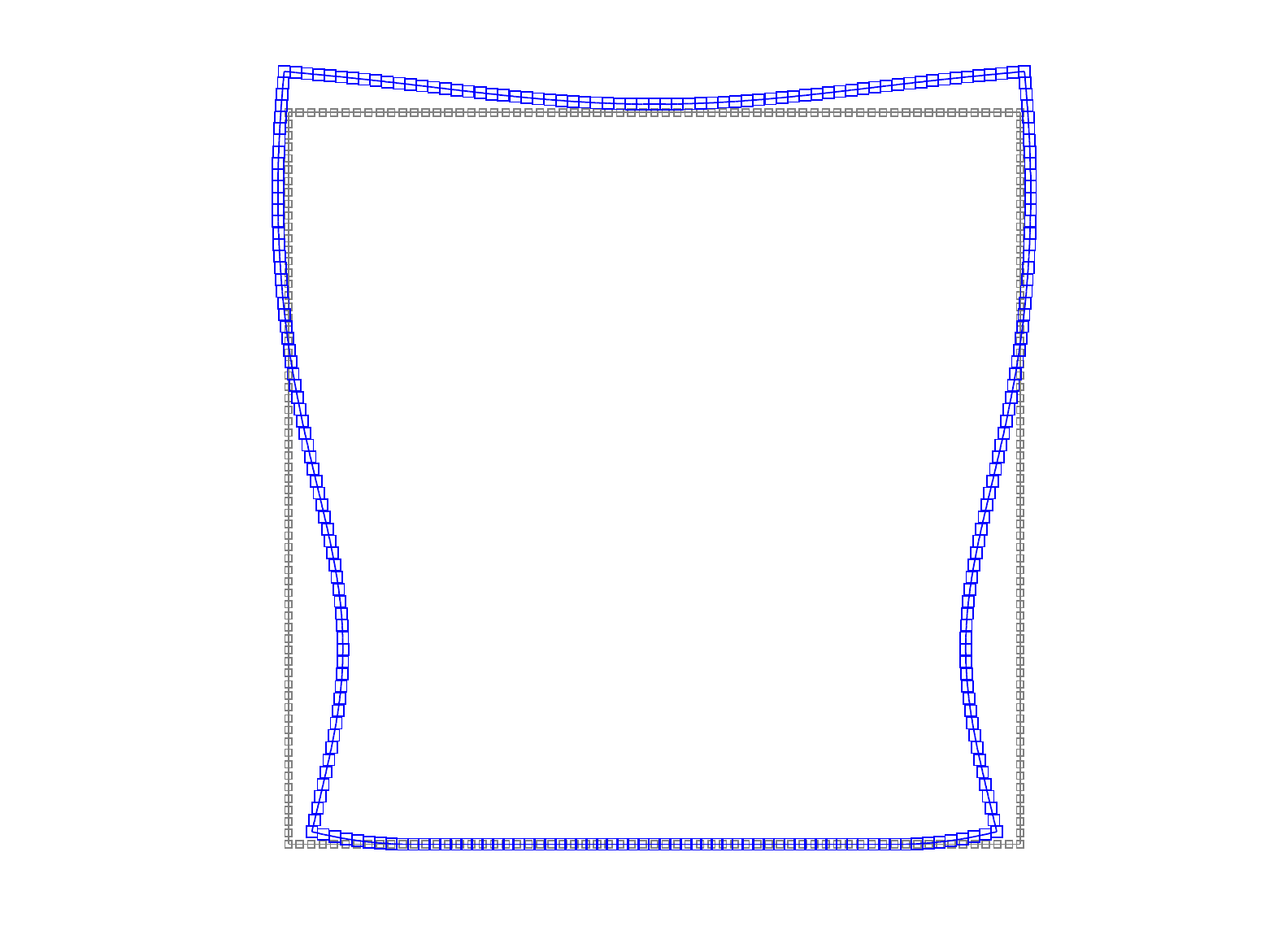}}
    \subfigure[$\lambda_n$ (blue), $\lambda_t$ (red)]{
  	\includegraphics[trim = 10mm 7mm 10mm 8mm, clip,width=65.0mm, keepaspectratio]{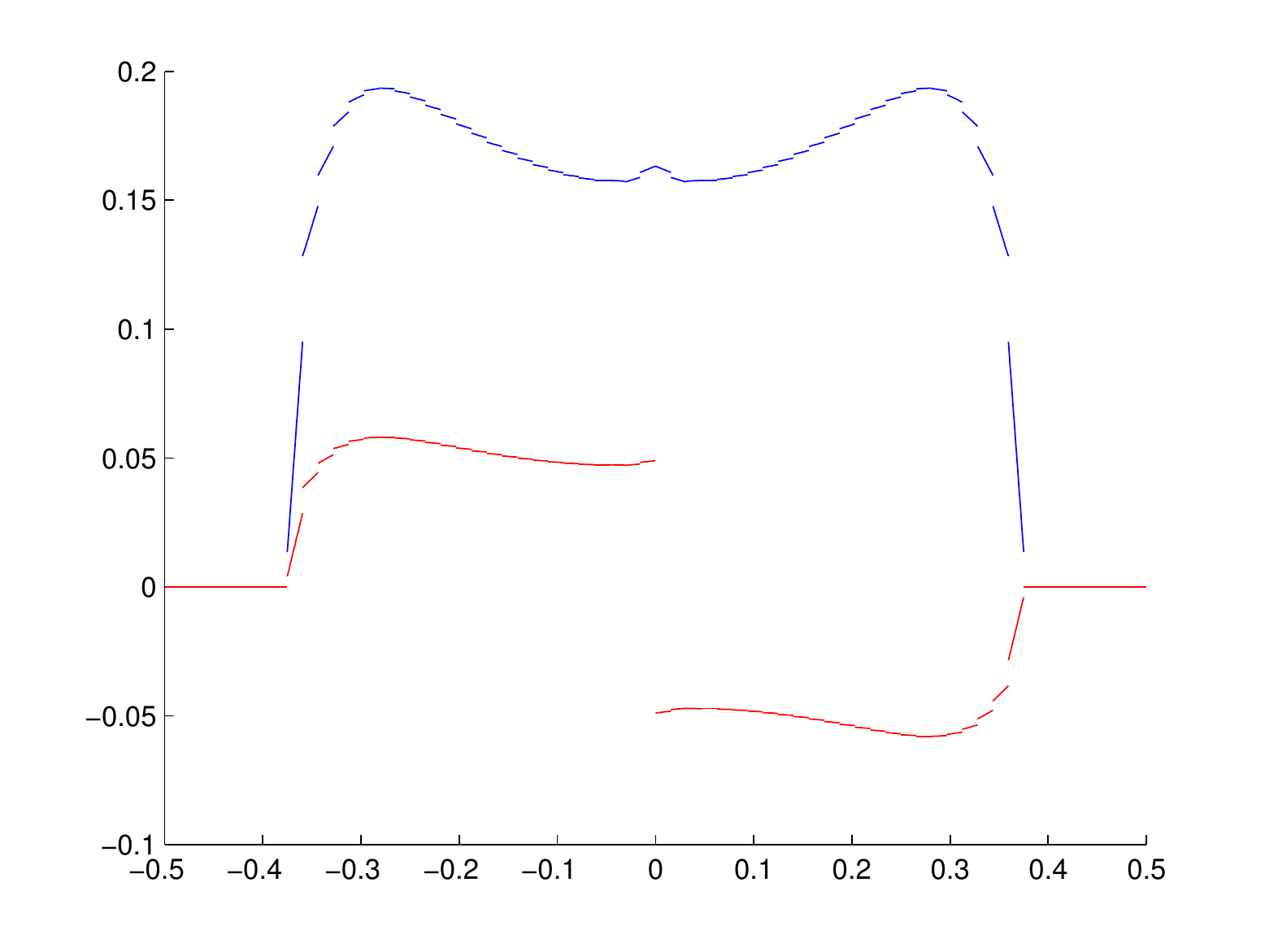}} }
    \caption{Solution of the Coulomb-frictional problem, uniform mesh 256 elements, $p=1$ }
    \label{fig:Coulomb:solution}
  \end{figure}

  \begin{figure}[tbp]
    \centering
    \includegraphics[trim = 10mm 2mm 15mm 5mm, clip,width=100.0mm, keepaspectratio]{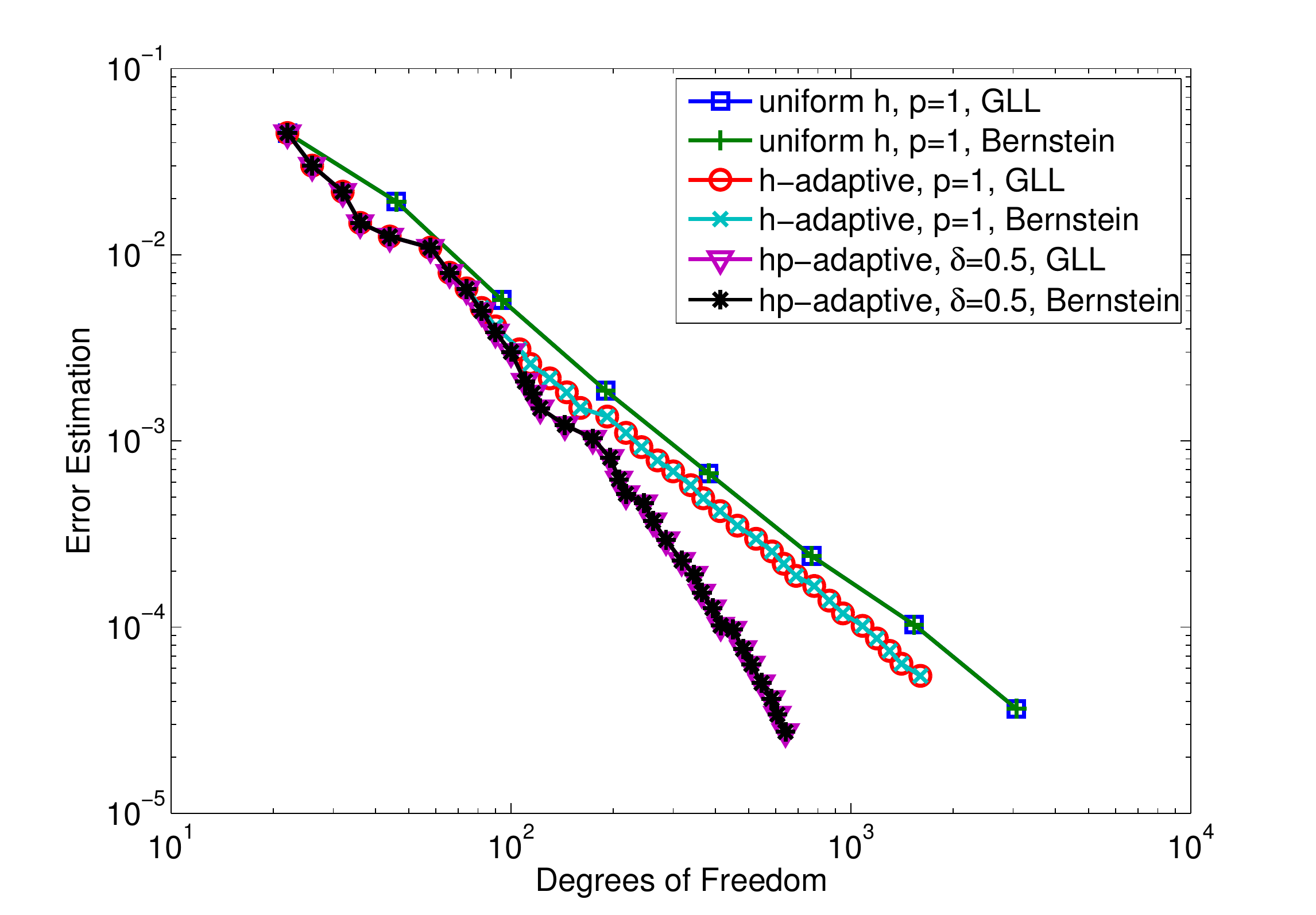}
    \caption{Error estimations for different families of discrete solutions (Coulomb-friction)}
    \label{fig:Coulomb:errorDiri}
  \end{figure}

  \begin{figure}[tbp]
    \centering \mbox{
    \subfigure[uniform $h$-version with $p=1$ (GLL/Bernstein)]{
  	\includegraphics[trim = 25mm 5mm 25mm 9mm, clip,width=75.0mm, keepaspectratio]{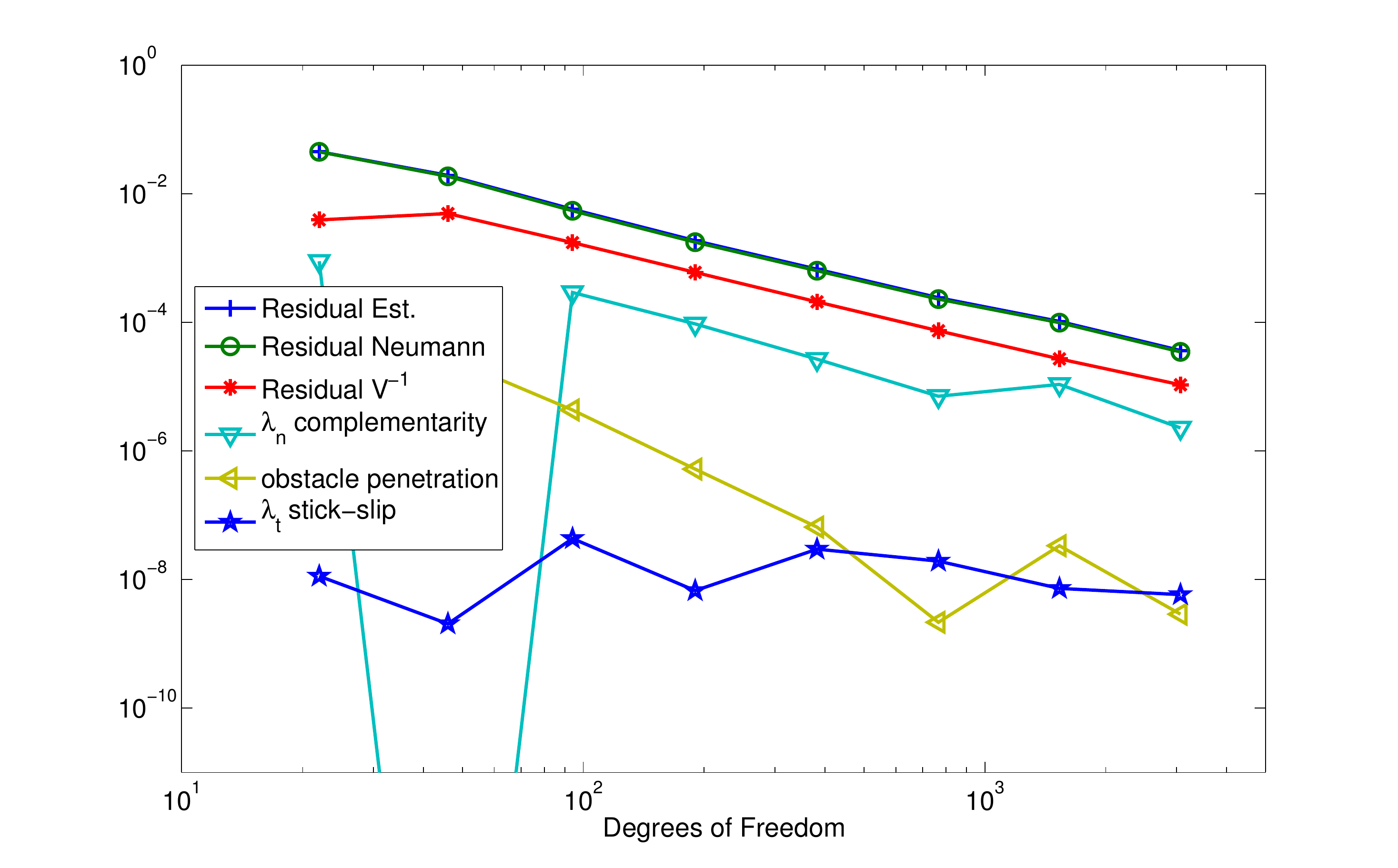}}
    \subfigure[$hp$-adaptive (GLL/Bernstein)]{
  	\includegraphics[trim = 25mm 5mm 25mm 9mm, clip,width=75.0mm, keepaspectratio]{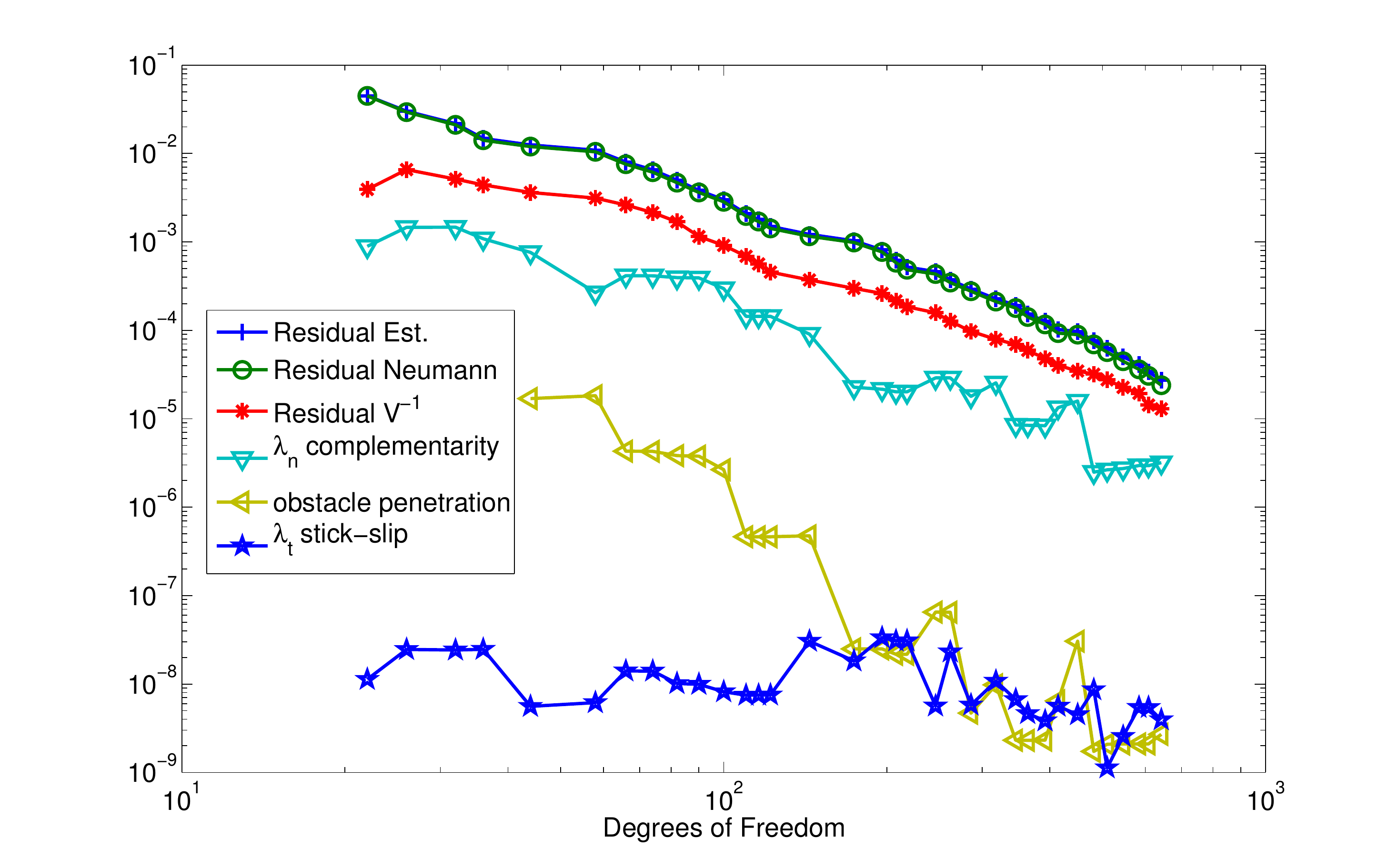}} }
    \caption{Error contributions of the residual based error indicator (Coulomb-friction) }
    \label{fig:Coulomb:errorContributions}
  \end{figure}


 \begin{figure}
   \centering   \mbox{\subfigure[$h$-adap.~(GLL/Bernstein), mesh nr.~16 (in), 25 (out)]{
   \begin{overpic}[trim = 41mm 31mm 35mm 7mm, clip,width=60.0mm, keepaspectratio]{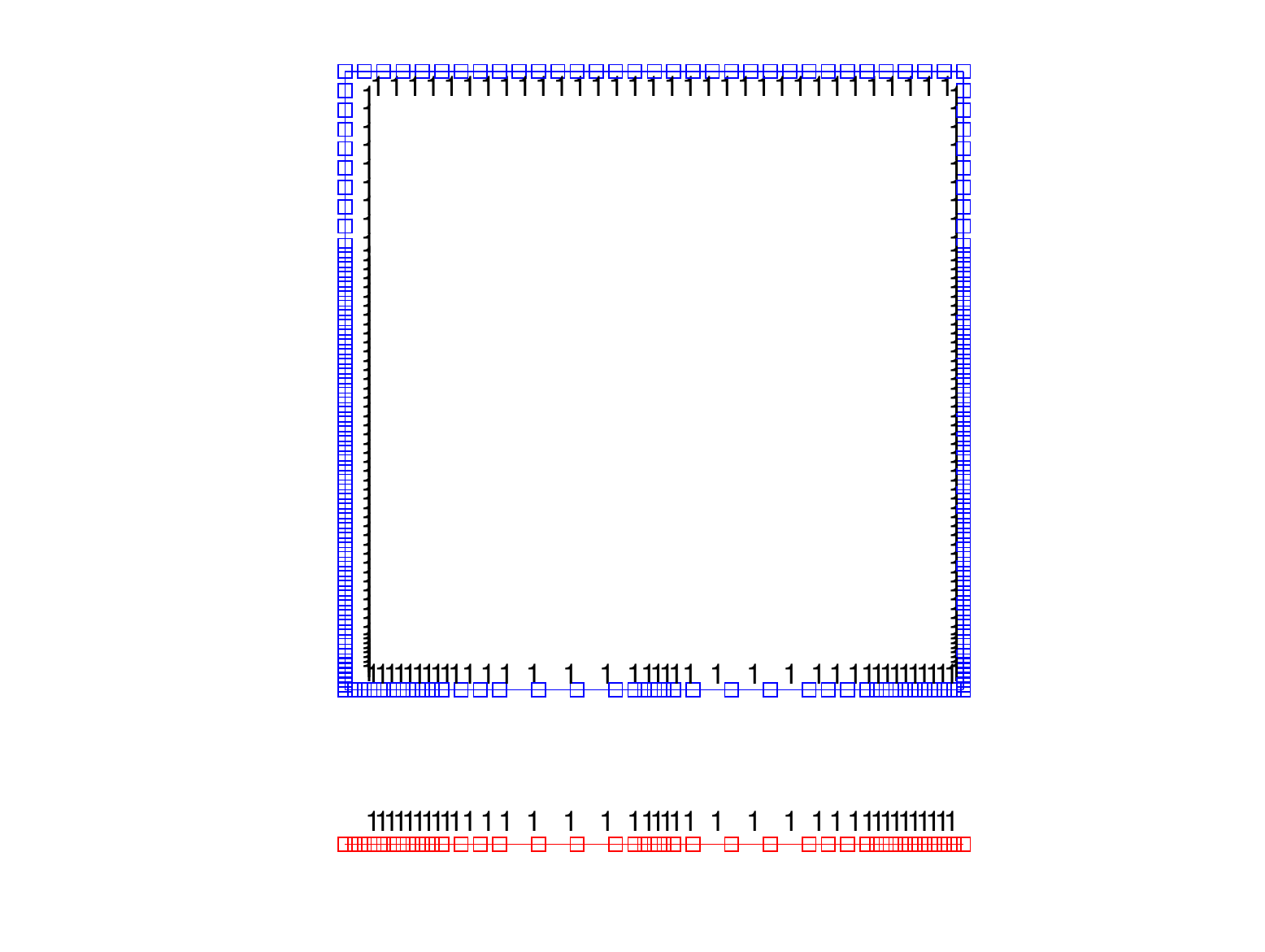}
     \put(7.1,8.4){ \includegraphics[trim = 41mm 31mm 35mm 7mm, clip, width=50.0mm, keepaspectratio]{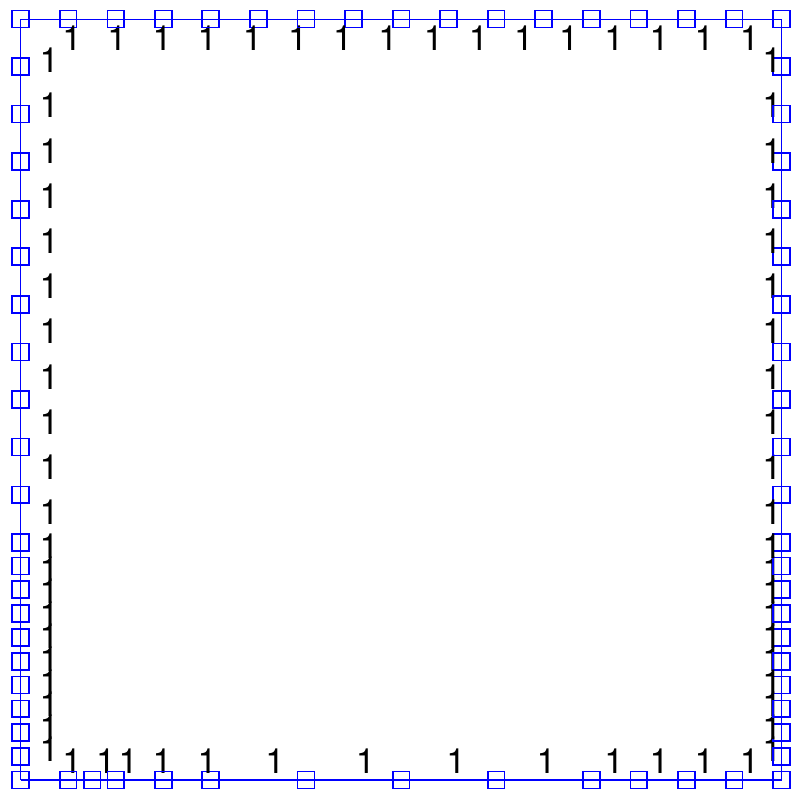}}
   \end{overpic} } \hspace{+1.5cm}
 	\subfigure[$hp$-adap.~(Bernstein), mesh nr.~16 (in), nr.~25 (out)]{
   \begin{overpic}[trim = 41mm 31mm 35mm 7mm, clip,width=60.0mm, keepaspectratio]{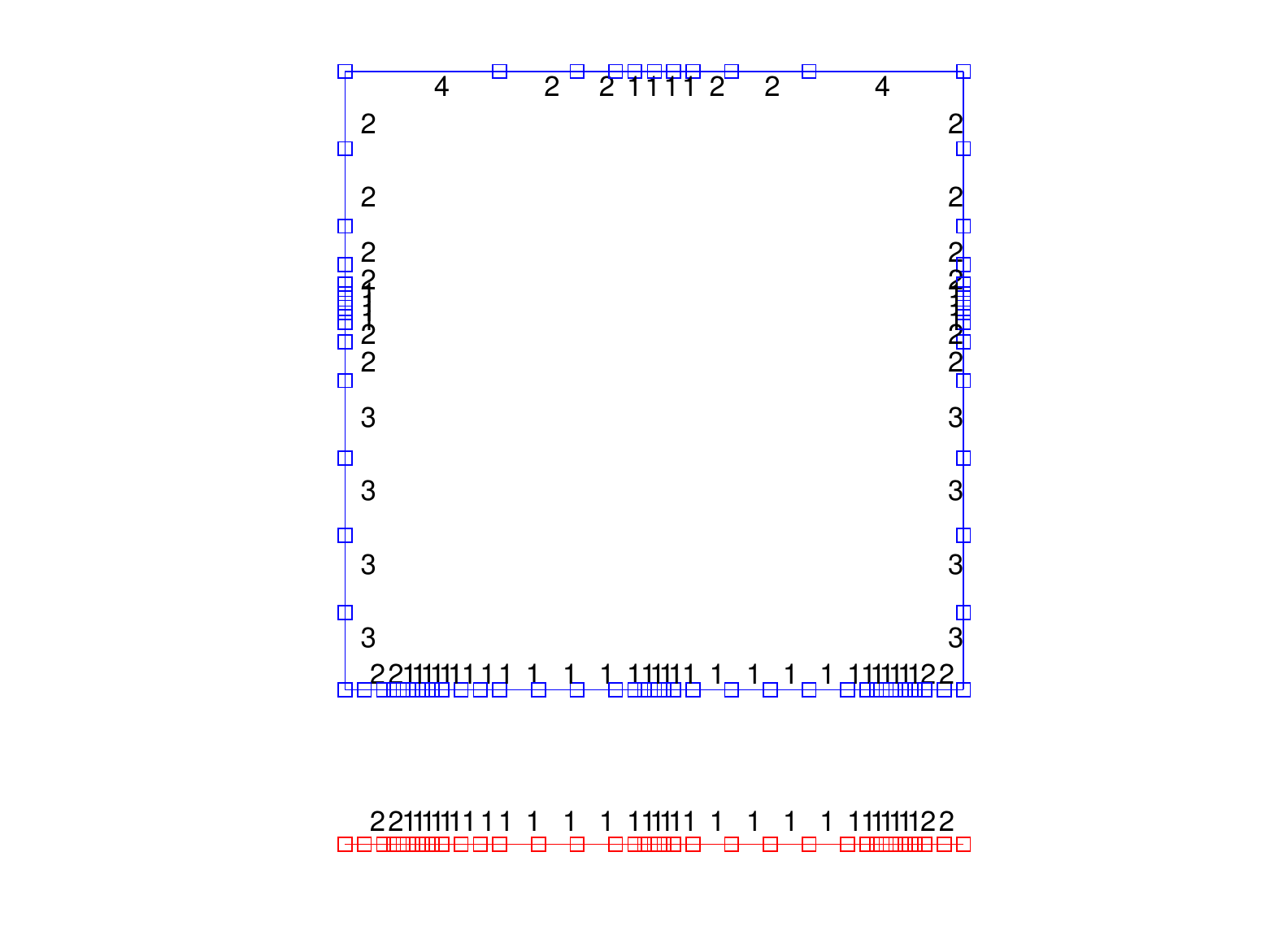}
     \put(7.1,8.4){ \includegraphics[trim = 41mm 31mm 35mm 7mm, clip, width=50.0mm, keepaspectratio]{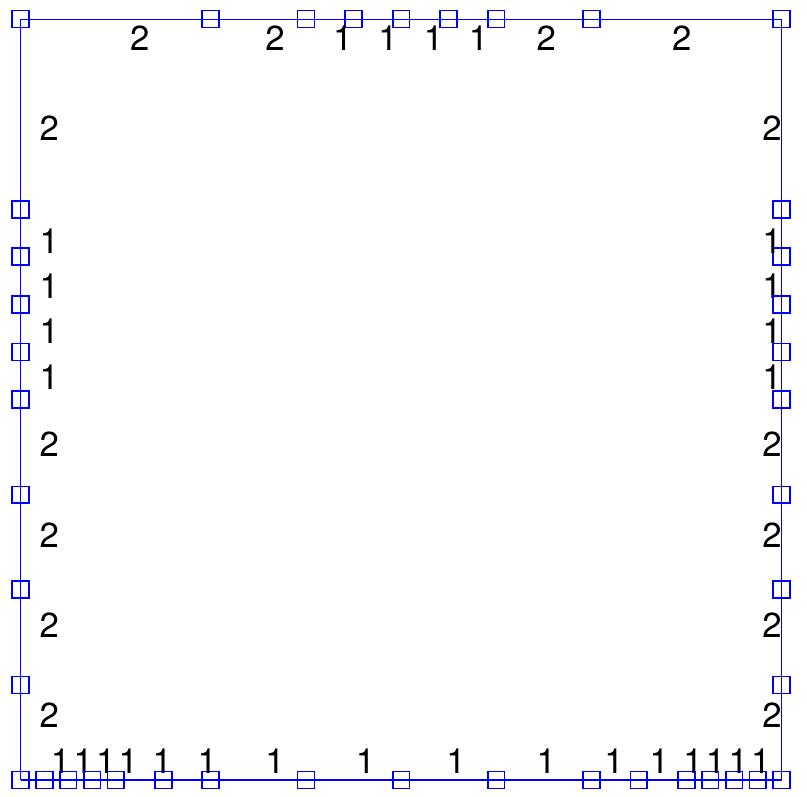}}
   \end{overpic}	} }
	
 	\subfigure[$hp$-adap.~(GLL), mesh nr.~16 (inner), nr.~25 (outer)]{
   \begin{overpic}[trim = 41mm 31mm 35mm 7mm, clip,width=60.0mm, keepaspectratio]{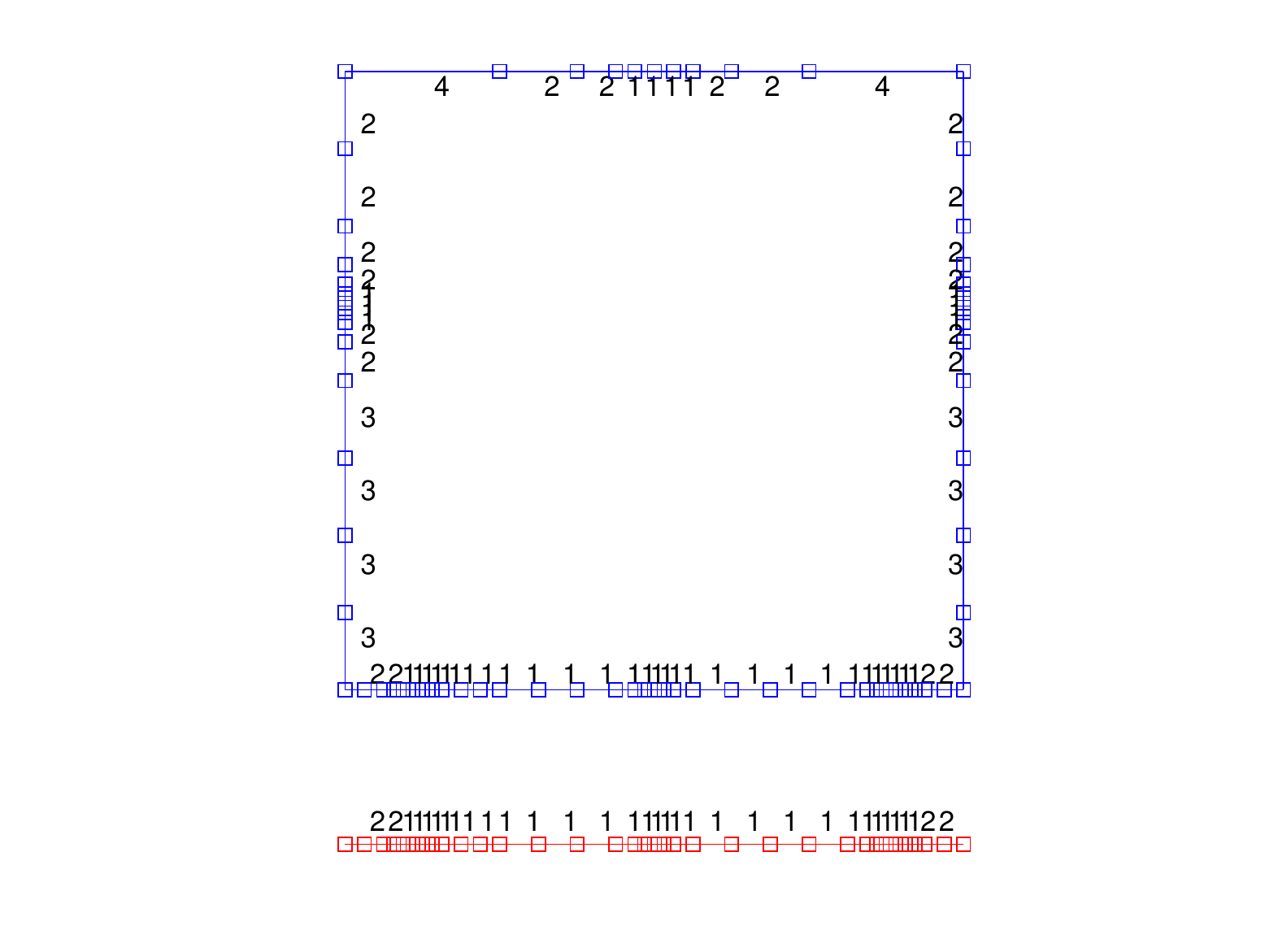}
     \put(7.1,8.4){ \includegraphics[trim = 41mm 31mm 35mm 7mm, clip, width=50.0mm, keepaspectratio]{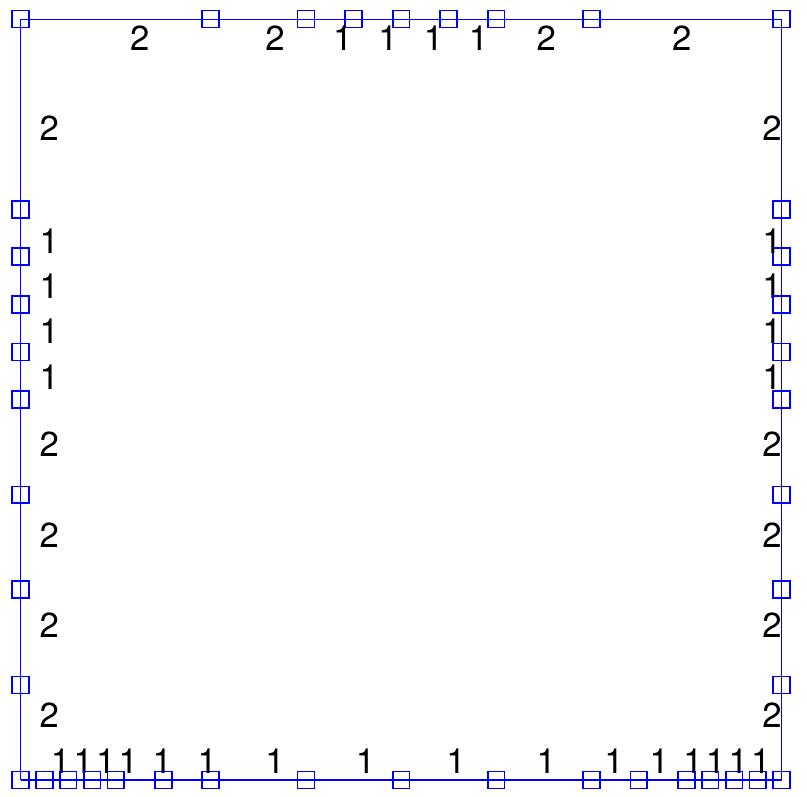}}
   \end{overpic}	}
 \caption{Adaptively generated meshes (Coulomb-friction)}
 \label{fig:Coulomb:AdaptiveMeshes}
 \end{figure}

 \subsection{Influence of the stabilization for the Neumann boundary value problem with Coulomb-friction}

From Lemma~\ref{lem:coerciveDiscrete} it is clear that if $\gamma_0$ is chosen to be too large the system matrix has at least one negative eigenvalue and the entire theory may no longer hold. Figure~\ref{fig:gamma0_Test:errorCoul} shows the error estimation for a uniform mesh with $256$ elements and $p=1$ versus $\gamma_0$. In all cases the iterative solver converges to a solution of the discrete problem. But for $\gamma_0 \geq 0.152$ the system matrix has a negative eigenvalue and the discrete solution looks unphysical or even simply useless. Interestingly, the error estimation captures this partly, the red curve in Figure~\ref{fig:gamma0_Test:errorCoul}, even though the error estimation may not be an upper bound of the discretization error. Once $\gamma_0$ is sufficiently small, here $1.9 \cdot 10^{-12} \leq \gamma_0 \leq 6.6 \cdot 10^{-2}$, there is (almost) no dependency on the absolute value of $\gamma_0$ itself, neither in the error estimate nor in the discrete solution itself. Only if $\gamma_0$ is
further decreased, i.e.~the stabilization is effectively switched off, the Lagrangian multiplier starts to oscillate as it is typical for the non-stabilized case, when using the same mesh and polynomial degree for $u^{hp}|_{\Gamma_C}$ and $\lambda^{kq}$ and no special basis functions. This is captured by the increase in the error estimation.\\

 \begin{figure}[tbp]
    \centering
    \includegraphics[trim = 4mm 1mm 18mm 5mm, clip,width=80.0mm, keepaspectratio]{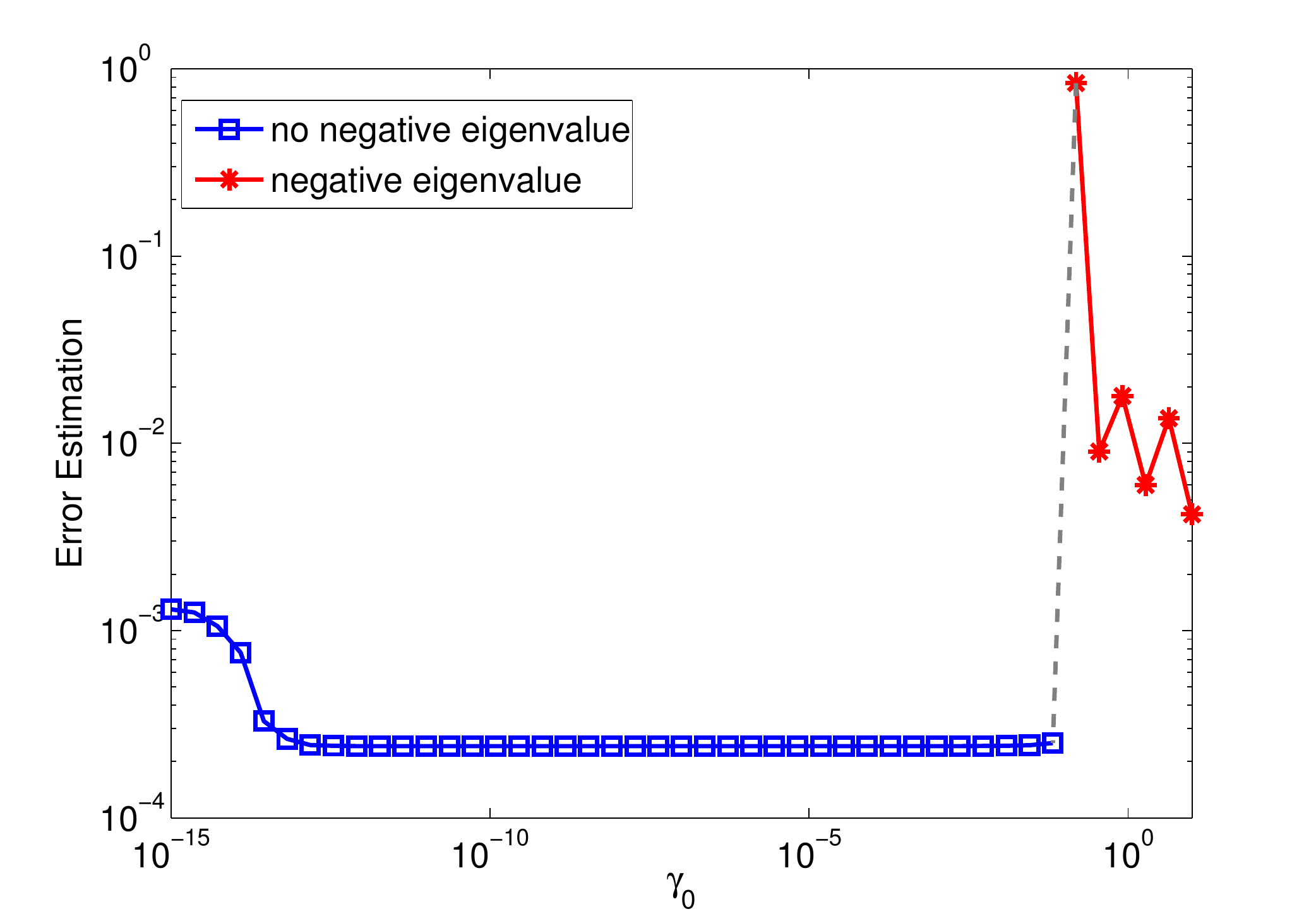}
    \caption{Dependency of error estimation on $\gamma_0$ for uniform mesh with 256 elements and $p=1$ (Coulomb-friction)}
    \label{fig:gamma0_Test:errorCoul}
  \end{figure}

Within the simulation, the most time consuming contribution is the computation of the matrices $\widehat{\mathbf{WW}}$, $\widehat{\mathbf{K^\top K^\top}}$ and $\widehat{\mathbf{WK^\top}}$ for the stabilization matrix $\hat{\mathbf{S}}$. Since $\gamma_0$ is allowed to be very small it may be favorable to compute these matrices only approximately. In the following we replace $\widehat{\mathbf{WW}}$ by $\bar{\mathbf{W}}^\top \mathbf{M}_D^{-\top} \mathbf{M}_\gamma \mathbf{M}_D^{-1} \bar{\mathbf{W}}$ where
\begin{align}
  \left(\mathbf{M}_\gamma\right)_{i,j}:=\left\langle \gamma \phi_j, \phi_i\right\rangle_{\Gamma_C}, \quad
	\left(\mathbf{M}_D\right)_{i,j}:=\left\langle \phi_j, \phi_i\right\rangle_{\Gamma_\Sigma}, \quad
	\left(\bar{\mathbf{W}}\right)_{i,j}:=\left\langle W\varphi_j, \phi_i\right\rangle_{\Gamma_\Sigma}, \quad
	\left(\bar{\mathbf{K}}^\top\right)_{i,j}:=\left\langle K^\top \phi_j, \phi_i\right\rangle_{\Gamma_\Sigma}
\end{align}
with $\operatorname{span} \left\{\phi_i\right\}_i = \mathcal{V}_{hp+1}^D$ and $\operatorname{span} \left\{\varphi_i\right\}_i = \mathcal{V}_{hp}$. In particular $\mathbf{M}_D$ is only a block-diagonal matrix and thus its inverse is cheap. The difference to the original formulation in Section~\ref{sec:Implementation} is in an intermediate projection of $Wu^{hp}$, $Wv^{hp}$ onto the discontinuous finite element space $\mathcal{V}^D_{hp+1}$. Analogously, the matrices $\widehat{\mathbf{K^\top K^\top}}$, $\widehat{\mathbf{WK^\top}}$ are replaced by $(\bar{\mathbf{K}}^\top)^\top \mathbf{M}_D^{-\top} \mathbf{M}_\gamma \mathbf{M}_D^{-1} \bar{\mathbf{K}}^\top$, $(\bar{\mathbf{K}}^\top)^\top \mathbf{M}_D^{-\top} \mathbf{M}_\gamma \mathbf{M}_D^{-1} \bar{\mathbf{W}}$, respectively. Even though four instead of three matrices must now be computed, only two potentials (due to elementwise partial integration of $W$) must be evaluated and thus this is significantly faster.

Figure~\ref{fig:Stab_aaprox:errorCoul} shows the decay of the error estimation for the uniform $h$ version with $p=1$ and for the $hp$-adaptive scheme with Gauss-Lobatto-Lagrange basis functions when using the above approximation of the stabilization matrix. For comparison the corresponding curves from Figure~\ref{fig:Coulomb:errorDiri} are also depicted. The difference in the error estimation for the original stabilization approach and its approximation is $\pm 0.014\%$ for the uniform $h$-version with $p=1$ and $\pm 0.02\%$ for the $hp$-adaptive scheme. Hence, contact stabilized BEM is suitable for practical applications.

 \begin{figure}[tbp]
    \centering
    \includegraphics[trim = 4mm 1mm 18mm 5mm, clip,width=85.0mm, keepaspectratio]{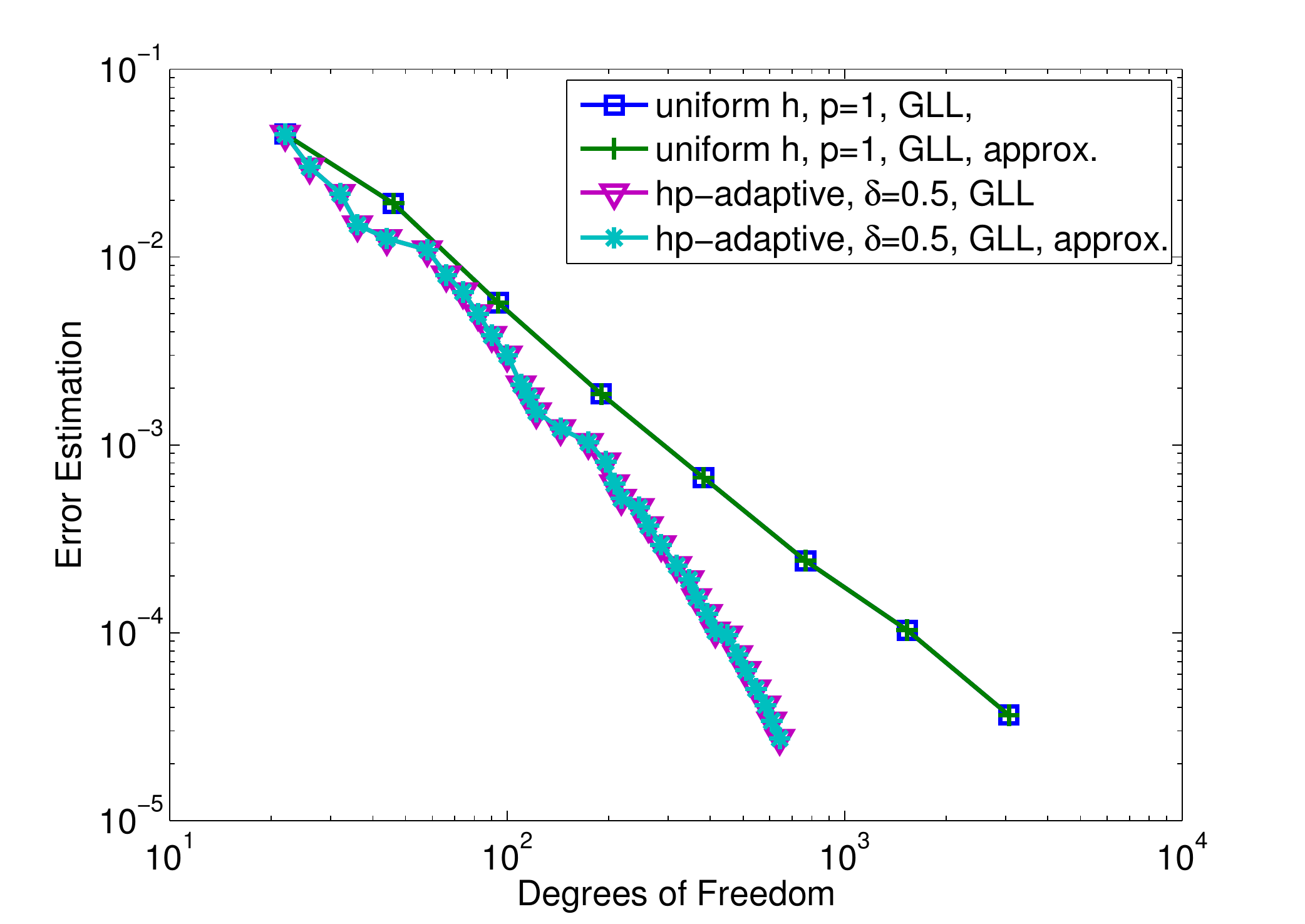}
    \caption{Error estimations for families of discrete solutions with and without approximation of the stabilization matrix $\hat{\mathbf{S}}$ (Coulomb-friction)}
    \label{fig:Stab_aaprox:errorCoul}
  \end{figure}

\noindent \emph{Acknowledgments:} H.G.~thanks the Danish Science Foundation (FNU) for partial support through research grant 10-082866.

%
%
%
%
%
%
%

  \bibliographystyle{abbrv}
  \bibliography{BE_kontakt_arxiv}

\begin{thebibliography}{10}

\bibitem{Banz2014BEM}
L.~Banz and E.~P. Stephan.
\newblock {On $hp$-adaptive BEM for frictional contact problems in linear
  elasticity}.
\newblock Preprint LUH, submitted for publication, 2014.

\bibitem{barbosahughes1991}
H.~Barbosa and T.~Hughes.
\newblock {The finite element method with the Lagrange multipliers on the
  boundary: circumventing the Babu\v{s}ka-Brezzi condition.}
\newblock {\em Comput. Methods Appl.Mech.Engrg.}, 85:109--128, 1991.

\bibitem{carstensen1996posteriori}
C.~Carstensen.
\newblock A posteriori error estimate for the symmetric coupling of finite
  elements and boundary elements.
\newblock {\em Computing}, 57(4):301--322, 1996.

\bibitem{carstensen1995adaptive}
C.~Carstensen and E.~Stephan.
\newblock Adaptive coupling of boundary elements and finite elements.
\newblock {\em Mod{\'e}lisation math{\'e}matique et analyse num{\'e}rique},
  29(7):779--817, 1995.

\bibitem{Chernov2006}
A.~Chernov.
\newblock {\em Nonconforming boundary elements and finite elements for
  interface and contact problems with friction - hp-version for mortar, penalty
  and Nitsche's methods}.
\newblock PhD thesis, Universit\"{a}t Hannover, 2006.

\bibitem{ChernovmaischakStephan2008}
A.~Chernov, M.~Maischak, and E.~P. Stephan.
\newblock hp-mortar boundary element method for two-body contact problems with
  friction.
\newblock {\em Mathematical Methods in the Applied Sciences.}, 31:2029--2054,
  2008.

\bibitem{costabel1988boundary}
M.~Costabel.
\newblock {Boundary integral operators on Lipschitz domains: Elementary
  results}.
\newblock {\em SIAM J. Math. Anal.}, 19:613, 1988.

\bibitem{costabelstephan1990}
M.~Costabel and E.~P. Stephan.
\newblock {Integral equations for transmission problems in linear elasticity}.
\newblock {\em J. Integral Equations Applications}, 2:211--223, 1990.

\bibitem{hild2009residual}
P.~Hild and V.~Lleras.
\newblock Residual error estimators for coulomb friction.
\newblock {\em SIAM Journal on Numerical Analysis}, 47(5):3550--3583, 2009.

\bibitem{hildrenard2007}
P.~Hild and Y.~Renard.
\newblock {An error estimate for the Signorini problem with Coulomb friction
  approximated by finite elements}.
\newblock {\em Numer. Anal.}, 45:2012--2031, 2007.

\bibitem{hildrenard2010}
P.~Hild and Y.~Renard.
\newblock {A Stabilized Lagrange multiplier method for the finite element
  approximation of contact problems in elastostatics}.
\newblock {\em Numer. Math.}, 115:101--129, 2010.

\bibitem{hueber2012equilibration}
S.~H{\"u}eber and B.~Wohlmuth.
\newblock Equilibration techniques for solving contact problems with coulomb
  friction.
\newblock {\em Computer Methods in Applied Mechanics and Engineering},
  205:29--45, 2012.

\bibitem{IssaouiDiss}
A.~Issaoui.
\newblock {\em $hp$-BEM for Contact Problems and Extended Ms-FEM in Linear
  Elasticity}.
\newblock PhD thesis, Leibniz Universit\"{a}t Hannover, 2014.

\bibitem{Karkulik2012}
M.~Karkulik.
\newblock {\em {Zur Konvergenz und Quasioptimalit\"{a}t aadaptive
  Randelementmethoden}}.
\newblock PhD thesis, TU Wien, 2012.

\bibitem{maischak2005adaptive}
M.~Maischak and E.~Stephan.
\newblock {Adaptive hp-versions of BEM for Signorini problems}.
\newblock {\em {Appl. Numer. Math.}}, 54(3-4):425--449, 2005.

\bibitem{melenk2001residual}
J.~Melenk and B.~Wohlmuth.
\newblock {On residual-based a posteriori error estimation in hp-FEM}.
\newblock {\em Adv. Comput. Math.}, 15(1):311--331, 2001.

\bibitem{schroder2012posteriori}
A.~Schr{\"o}der.
\newblock A posteriori error estimates of higher-order finite elements for
  frictional contact problems.
\newblock {\em Computer Methods in Applied Mechanics and Engineering},
  249:151--157, 2012.

\end{thebibliography}

%
%
%
%
%
%

\end{document}